\documentclass[12pt,a4paper]{amsart}

\usepackage[usenames,dvipsnames]{xcolor}
\usepackage[bookmarksopen,bookmarksdepth=2]{hyperref}
\hypersetup{colorlinks=true,citecolor=NavyBlue,linkcolor=NavyBlue,urlcolor=Green} 

\usepackage{amsmath,amsthm,amssymb,amscd,mathtools}
\usepackage{amsfonts}
\usepackage{lastpage}
\usepackage{fancyhdr}
\usepackage{enumitem}

\usepackage{mathrsfs}
\usepackage{newtxtext}
\usepackage{newtxmath}
\usepackage{tikz-cd}
\usepackage{microtype}     
\usepackage{stmaryrd}
\usepackage[T1]{fontenc}
\usepackage{setspace}
\usepackage[backend=biber, style=alphabetic,maxbibnames=99,maxalphanames=99]{biblatex}
\newcommand{\mbb}[1]{\mathbb{#1}}

\newcommand{\NN}{\mathbb{N}}
\newcommand{\RR}{\mathbb{R}}
\newcommand{\ZZ}{\mathbb{Z}}
\newcommand{\QQ}{\mathbb{Q}}
\newcommand{\CC}{\mathbb{C}}
\newcommand{\FF}{\mathbb{F}}

\newcommand{\Gal}{\mathrm{Gal}}

\newcommand{\Ind}{\mathrm{Ind}}
\DeclareMathOperator{\im}{Im}

\DeclareMathOperator{\Hom}{Hom}
\DeclareMathOperator{\Ext}{Ext}

\DeclareMathOperator{\Nm}{Nm}

\DeclareMathOperator{\Spec}{Spec}

\DeclareMathOperator{\Lie}{Lie}

\DeclareMathOperator{\Ad}{Ad}

\newcommand{\GL}{\mathrm{GL}}
\newcommand{\PGL}{\mathrm{PGL}}

\newcommand{\p}{\mathfrak{p}}

\newcommand{\mO}{\mathcal{O}}

\newcommand{\mA}{\mathbb{A}}

\DeclareMathOperator{\SL}{SL}

\DeclareMathOperator{\End}{End}

\newtheorem{thm}{Theorem}[section]
\newtheorem{prop}[thm]{Proposition}

\newtheorem{lem}[thm]{Lemma}
\newtheorem{cor}[thm]{Corollary}

\theoremstyle{remark}
\newtheorem{rmk}[thm]{Remark}
\newtheorem{eg}[thm]{Example}
\theoremstyle{definition}
\newtheorem{dfn}[thm]{Definition}

\numberwithin{equation}{section}

\newcommand{\et}{\mathrm{\acute{e}t}}

\newcommand{\TT}{\mbb{T}}

\newcommand{\la}{{\mathrm{la}}}
\newcommand{\an}{{\mathrm{an}}}

\newcommand{\mrm}[1]{\mathrm{#1}}

\newcommand{\Sym}{\mrm{Sym}}

\newcommand{\Sh}{\mathrm{Sh}}
\newcommand{\diag}{\mathrm{diag}}

\newcommand{\oInd}{\otimes\!-\!\mrm{Ind}}

\DeclareMathOperator{\Fr}{Fr}

\newcommand{\CH}{\mrm{CH}}

\DeclareMathOperator{\oIndFQ}{\otimes-\mrm{Ind}^{\QQ}_{F}}

\usepackage{mathtools}

\newcommand{\ti}[1]{\tilde{#1}}

\begin{document}
\title{Plectic Lie Algebra Action on Hilbert modular varieties}
 \author{Yuanyang Jiang}
 \author{Lue Pan}
\begin{abstract}
We construct some part of the plectic Galois action on the locally analytic $p$-adically completed cohomology of Hilbert modular varieties, predicted by the plectic conjecture of Nekov\'a\v r-Scholl. As a byproduct, we prove a vanishing result for the completed cohomology of Hilbert modular varieties below middle degree. To do this, we make use of the partial Sen operators constructed earlier by a geometric method.
\end{abstract}


\maketitle

\setcounter{tocdepth}{10}
 \tableofcontents

\section{Introduction}

Let $p$ be a prime number. This work concerns the $p$-adic cohomology of Hilbert modular varieties. We fix some notation first.
Let \( F \) be a totally real field of degree \( d \) over \( \QQ \) and $\mathrm{G}=\mathrm{Res}_{F/\QQ}\GL_2$ the restriction of scalars of $\GL_2$ from $F$ to $\QQ$. For a neat open compact subgroup \( K\subset \mathrm{G}(\mA_{f}) \), we denote by $\mrm{Sh}_{K} $ the Hilbert modular variety of level \( K \) over \( \QQ \). See \S \ref{subsec-Main result} below for more details. Let $S_{0}$ be a finite set of places of $\QQ$ such that $p\in S_0$ and $K$ contains $\GL_2(\widehat{\ZZ}^{S_{0}}\otimes_{\ZZ} \mathcal{O}_F)$.  There is a natural action of \( \TT^{S_{0}}\times G_{\QQ} \) on \( H^{i}_{\et}(\Sh_{K,\overline{\QQ}},\ZZ/p^{n}) \), where \( G_{\QQ}=\Gal(\overline{\QQ}/\QQ) \) denotes the absolute Galois group of $\QQ$ and $\TT^{S_{0}}$ denotes the spherical Hecke algebra over $\ZZ_p$ at places away from $S_0$.

The actions of \( \TT^{S_{0}} \) and \( G_{\QQ} \) are related as follows:  let \( \rho:\Gal(\overline{F}/F) \to \GL_{2}(\overline{\QQ_p}) \) be an absolutely irreducible continuous representation unramified outside of places above $S_0$,  which determines a character \( \lambda_{\rho}:\TT^{S_{0}}\to \overline{\QQ_p} \) via the usual Eichler-Shimura relation. 
It follows from the Langlands-Kottwitz method (\cite{Langlands1979Zeta}, \cite{ReimannHarry1997Tszf}, \cite{BrylinskiLabesse1984cohomologie}) that there is an isomorphism of \(G_{\QQ}\)-representations \begin{align}\label{alignLangKott}
H^{d}_{\et}(\Sh_{K,\overline{\QQ}}, \overline{\QQ_p})[\lambda_{\rho}]^{\mrm{ss}}\cong (\oInd_{F}^{\QQ}\rho)^{\oplus n}
\end{align} for some \(n\in \NN\), where \( [\lambda_{\rho}] \) denotes the \( \lambda_{\rho} \)-eigenspace for the \( \TT^{S_{0}} \)-action, \((-)^{\mrm{ss}}\) denotes the semisimplification, and \( \oIndFQ\rho \) denotes the tensor induction from \( \Gal(\overline{F}/F) \) to \( G_{\QQ} \) cf. Remark \ref{rmk-tenindvsplecind} below. 
A highly non-trivial result of Nekov\'a\v r \cite{Nekovar-2018} says that the Galois action on $H^{d}_{\et}(\Sh_{K,\overline{\QQ}}, \overline{\QQ_p})[\lambda_{\rho}]$ is semisimple, and hence the superscript \((-)^{\mrm{ss}}\) can be omitted. In fact, Nekov\'a\v r even allowed coefficient systems in his work and proved the semisimplicity for the intersection cohomology.

In our work, we obtain a similar result in the context of the completed cohomology introduced by Emerton (\cite{Emerton-2006}).   Fix a neat tame level \( K^{p}\subset G(\mA_{f}^{p}) \). The completed cohomology of tame level $K^p$ is defined as
\[
\tilde{H}^i(K^p):=\varprojlim_n\varinjlim_{K_p\subseteq \mathrm{G}(\QQ_p)} H^i_{\acute{e}t}(\Sh_{K^pK_p,\overline{\QQ}},\ZZ/p^n),
\]
where $K_p$ runs through all open compact subgroups of $\mathrm{G}(\QQ_p)$. 
It  carries an action of \( \TT^{S_{0}}\times \mathrm{G}(\QQ_{p})\times G_{\QQ} \). 
We define the rational completed cohomology $\tilde{H}^i(K^p)_{\QQ_p}:=\tilde{H}^i(K^p)\otimes_{\ZZ_p} \QQ_p$, and denote by \( \ti{H}^{i}(K^{p})^{\la} \) the subspace of \(\mathrm{G}(\QQ_{p}) \)-locally analytic vectors.  For the purpose of introduction, we state here a simplified version of our main result.

\begin{thm} \label{thm-intro-plec1}
Let \( \rho:\Gal(\overline{F}/F) \to \GL_{2}(\overline{\QQ_p}) \) be a strongly irreducible representation (i.e. the connected component of the Zariski closure of its image acts irreducibly on $\overline{\QQ_p}^2$). Assume that $\rho$ has distinct \( \tau \)-Hodge-Tate-Sen weights for some \( \tau:F\to\overline{\QQ_p} \). Denote by \( \lambda_{\rho}:\TT^{S_{0}}\to \overline{\QQ_p} \) the associated character. 

\begin{enumerate}
\item There is a natural isomorphism of $G_{\QQ}\times \mathrm{G}(\QQ_p)$-representations:
 \[
\left(\ti{H}^{d}(K^{p})^{\la}\otimes_{\QQ_p} \overline{\QQ_p}\right)[\lambda_{\rho}]
\cong \oIndFQ\rho\otimes_{\overline{\QQ_p}}\Pi^{\la}(\rho),
\] 
for some representation \( \Pi^{\la}(\rho) \) of \( \mathrm{G}(\QQ_{p}) \).
\item The localization $\ti{H}^{<d}(K^{p})_{\ker \lambda_{\rho}}=0$,  in particular, $\left(\ti{H}^{<d}(K^{p})\otimes_{\ZZ_p} \overline{\QQ_p}\right)[\lambda_{\rho}]=0$. 
\end{enumerate}
\end{thm}

\begin{rmk}
By the known relationship between the completed cohomology and the usual cohomology of Shimura varieties, we thus reprove Nekov\'a\v r's result when $\rho$ is strongly irreducible, equivalently the corresponding automorphic form is not CM.
\end{rmk}

\begin{rmk}
The second part of the main theorem implies that the completed cohomology below middle degree can be annihilated by certain elements in the Hecke algebra. Our actual theorem \ref{thm-belmidnsi} below make these elements more explicit.
\end{rmk}

\begin{rmk}
Nekov\'a\v r's argument also works for the Hecke eigenspace of the completed cohomology and gives the semisimplicity of the Galois action on $\left(\ti{H}^{d}(K^{p})\otimes_{\ZZ_p} \overline{\QQ_p}\right)[\lambda_{\rho}]$. Our work implies that all the simple factors of the tensor induction show up in the Hecke eigenspace, if it's non-zero, which does not seem to follow from the method of Nekov\'a\v r. This will be used by the first-named author in his thesis \cite{Jiang2025HMF} to prove classicality results for completed cohomology of Hilbert modular varieties. 
\end{rmk}

To discuss our proof strategy, it's helpful to recall some ideas behind Nekov\'a\v r's work. In \cite{NS-2016}, 
Nekov\'a\v r-Scholl conjectured that the \'etale cohomology carries a larger symmetry, called the plectic symmetry. More precisely, let \( G_{F}^{plec}:=\mrm{Aut}_{F}(F\otimes_{\QQ}\overline{\QQ}) \) be the plectic Galois group, which naturally contains \( G_{\QQ} \) as a subgroup. The work  \cite{NS-2016} conjectured that the \( G_{\QQ} \)-action can be extended to an action of $G_{F}^{plec}$ in some natural way. It was known that the tensor induction can be naturally extended to an \textit{irreducible} representation of $G_F^{plec}$, cf. Definition \ref{dfnPlecticInduction}. So the existence of the plectic symmetry on the cohomology will imply both Nekov\'a\v r's result and the first part of Theorem \ref{thm-intro-plec1}.

In his work \cite{Nekovar-2018},  Nekov\'a\v r constructed actions of the local plectic group $G_{F,\ell}^{plec}:=\mrm{Aut}_{F\otimes_{\QQ}\QQ_{\ell}}(F\otimes_{\QQ}\overline{\QQ_\ell})$ on the cohomology for rational primes $\ell\notin S_0$ completely split in $F$. Since we expect the local plectic action to be unramified in this case, essentially this boils down to construction of $d$ commuting Frobenus operators (called the partial Frobenii) satisfying the Eichler-Shimura relations on the \'etale cohomology. Nekov\'a\v r achieved this by exploring the moduli interpretation of some PEL type Shimura variety closely related to the Hilbert modular varieties. 

Our work instead  focuses on the local plectic symmetry at $\ell=p$ and constructs a plectic Galois group action on the locally analytic completed cohomology after localization.

\begin{thm} \label{thm-introd-plec-2}
Let \( \rho:\Gal(\overline{F}/F) \to \GL_{2}(\overline{\QQ_p}) \) be a strongly irreducible representation such that $\rho$ has distinct \( \tau \)-Hodge-Tate-Sen weights for some \( \tau:F\to\overline{\QQ_p} \).  There is an action of $G_{F}^{plec}$ on $\ti{H}^{d}(K^{p})^{\la}_{\ker \lambda_{\rho}}$ extending the action of $G_{\QQ}$.
\end{thm}

The key player in our construction is the so-called \textit{partial Sen operators}. It was first observed by the second-named author in \cite{Pan-2022} that when $d=1$,  the Sen operator on $\ti{H}^{*}(K^{p})^{\la}\widehat\otimes_{\QQ_p}\CC_p$ can be constructed by studying the $p$-adic geometry of modular curves at infinite level. This was later generalized in \cite{RC-2025} to all Shimura varieties. Explicitly in the Hilbert case, the geometric construction produces $d$ commuting operators, whose sum is the Sen operator. We call them the partial Sen operators and justify its name below. 

\begin{rmk}
It was pointed out by George Boxer to us that for general Shimura varieties, e.g.  unitary Shimura varieties for $U(m,n)$ with $m\geq n\geq 2$, there exist more  operators (of higher-order) than these partial Sen operators. These extra symmetries should also be considered as some sort of plectic symmetry which seems to be beyond the original scope of \cite{NS-2016}.
\end{rmk}

To show that the partial Sen operators we construct behave in the desired manner, by the continuity we only need to do this for a dense subspace of $\ti{H}^{d}(K^{p})^{\la}$. Here we look at the subspace of locally algebraic vectors of \textit{very general} weights, meaning that the corresponding tensor induction representation has $2^d$ distinct Hodge-Tate weights, i.e. regular. The key observation is that when the  tensor induction representation is regular, the partial Sen operators are determined by the Sen operator.

Once these partial Sen operators are constructed and shown to have the desired property,  we apply a result of Sen  \cite{Sen-1973} to get an action of the Lie algebra of the inertia subgroup of $G_{F,p}^{plec}$ on $\ti{H}^{d}(K^{p})^{\la}$. Under the strongly irreducible assumption, it turns out that this Lie algebra action is enough to extend the $G_{\QQ}$-action to a $G_F^{plec}$-action. 

Ideally if we can construct a plectic Galois group action on $\ti{H}^{<d}(K^{p})^{\la}$, then the first part of Theorem \ref{thm-intro-plec1} holds for $\ti{H}^{<d}(K^{p})^{\la}$ as well. This forces $\left(\ti{H}^{<d}(K^{p})^{\la}\otimes_{\QQ_p} \overline{\QQ_p}\right)[\lambda_{\rho}]
\cong \oIndFQ\rho\otimes_{L}\Pi^{\la}(\rho)=0$ because not all Hodge-Tate-Sen weights  of $\oIndFQ\rho$ can appear in  $\left(\ti{H}^{<d}(K^{p})^{\la}\otimes_{\QQ_p} \overline{\QQ_p}\right)[\lambda_{\rho}]
$ by the main result of \cite{Lan-Pan}. This would give the vanishing result in the second part of Theorem \ref{thm-intro-plec1}. In practice, we deduce the vanishing result by combining Theorem \ref{thm-introd-plec-2} and the Poincar\'e duality for completed cohomology to relate the plectic symmetry we construct to $\ti{H}^{<d}(K^{p})^{\la}$.

In order to talk about the ``naturalness" of the plectic Galois action, for example satisfying the Eichler-Shimura relations, we find it convenient to consider Galois (pseudo) deformation rings and the Cayley-Hamilton quotients. This will be the content of Section \ref{sec-GDR}. We will also introduce the Sen operator in this context and explain how to reinterpret Sen's results there.
In Section \ref{sec-PLA}, we recall some background on the plectic Galois group and introduce the notion of partial Sen operators. Finally in Section \ref{sec-CCHMV}, we carry out our construction of the plectic Galois group action and prove the main theorems. 

Notation: We fix an embedding \( \iota_{p}:\overline{\QQ}\hookrightarrow\overline{\QQ_{p}} \), and let \( \CC_{p} \) denote the completion of \( \overline{\QQ_{p}} \).
We fix \( L \) to be a finite extension of \( \QQ_{p} \) with ring of integer \( \mO_{L} \) and residue field \( \FF \).  In what follows, we work only with continuous representations, so we will omit the adjective ``continuous''. For any field \( K \), we write \( G_{K}:=\Gal(\bar{K}/K) \).

\section{Galois deformation rings} \label{sec-GDR}
In this section, we review some known results about (pseudo) Galois deformation rings and define a notion of Lie algebra of Galois groups over the generic fiber of these deformation rings. 

Recall that $L$ is a finite extension of $\QQ_p$ with ring of integers $\mO_L$ and residue field $\FF$. Throughout this section we fix $n\geq 1$.

\subsection{Determinant and Deformation ring} \label{subsec-Defring}
Let $\Gamma$ be a profinite group satisfying Mazur's condition $\Phi_p$: 
\begin{itemize}
\item for each open subgroup $\Delta$ of $\Gamma$, the pro-$p$ quotient of $\Delta$ is topologically finitely generated.
\end{itemize} 
Let $\bar{D}$ be a continuous $n$-dimensional determinant of $\Gamma$ (in the sense of Chenevier \cite{Chenevier-2014}) valued in $\FF$. Consider the deformation problem parametrizing all continuous lifts of  $\bar{D}$. It follows from our assumption that the deformation problem is represented by a Noetherian completed  local $\mO_L$-algebra $R_{\bar{D}}$ by \cite[Prop.E]{Chenevier-2014} or \cite[Theorem 3.1.4.6]{WE-2013moduli}. 
Denote by $D^u:R_{\bar{D}}[[\Gamma]]\to R_{\bar{D}}$ the universal determinant. 

In \cite[\S  1.17]{Chenevier-2014} Chenevier introduced the notion of Cayley-Hamilton algebra. We briefly recall it here. Suppose $D:R\to A$ is a determinant of dimension $n$. For $B$ a commutative $A$-algebra and $r\in R\otimes_A B$, its characteristic polynomial $\chi(r,t)\in B[t]$ is defined as $D_{B[t]}(t-r)$. We say $(R,D)$ is Cayley-Hamilton if $\chi(r,r)=0$ in $R\otimes_A B$ for any $B$ and $r$. It follows from the usual Cayley-Hamilton theorem that $(M_n(A),\det)$ is Cayley-Hamilton. In general Chenevier defined a two-sided ideal $\CH(D)\subset \ker(D)$ of $R$ such that the resulting determinant of $R/\CH(D)$ is the maximal Cayley-Hamilton quotient of $(R,D)$. We apply this to $(R_{\bar{D}}[[\Gamma]],D^u)$ and get the universal Cayley-Hamilton algebra 
\[E_{\bar{D}}:= R_{\bar{D}}[[\Gamma]]/\CH(D^u).\]

\begin{thm} \label{thmWE}
    \(E_{\bar{D}} \)
    is a finitely generated module of \(R_{\bar{D}}\), in particular, $E_{\bar{D}}$ is Noetherian.
\end{thm}

\begin{proof}
This is due to Wang-Erickson \cite[Prop.3.6]{WE-2018}.
\end{proof}

There is a natural group homomorphism $\rho:\Gamma\to E_{\bar{D}}^\times$. We want to consider the induced morphism of tangent spaces. It's well-known that in order to do this one has to pass to the rigid generic fiber because the usual logarithm function $\log(1-T)=-\sum_{i=1}\frac{T^i}{i}$ is analytic but not bounded on the open unit disc. 

Fix a presentation of $R_{\bar{D}}$ as a quotient
\[R_{\bar{D}}=\mO_L[[x_1,\cdots,x_g]]/I\]
for some ideal $I$.

\begin{dfn} \label{dfnE(r)}
For $r\in p^{\QQ}\cap (0,1)$, we denote by 
\[R_{\bar{D}}^{(r)}:=R_{\bar{D}}\otimes_{\mO_L[[x_1,\cdots,x_g]]} L\langle r{x_1},\cdots,r{x_g}\rangle,\]
\[E_{\bar{D}}^{(r)}:=E_{\bar{D}}\otimes_{R_{\bar{D}}} R_{\bar{D}}^{(r)}=E_{\bar{D}}\otimes_{\mO_L[[x_1,\cdots,x_g]]} L\langle r{x_1},\cdots,r{x_g}\rangle\]
where $L\langle r{x_1},\cdots,r{x_g}\rangle$ is the completion of $\mO_L[[x_1,\cdots,x_g]]\otimes_{\ZZ_p}\QQ_p$ with respect to the norm 
\[||\sum_{\alpha=(a_1,\cdots,a_g)\in \NN^g}c_\alpha x_1^{a_1}\cdots x_g^{a_g}||_r=\sup_{\alpha=(a_1,\cdots,a_g)} ||c_\alpha|| r^{a_1+\cdots+a_g}\]
where $||c_\alpha||$ denotes the $p$-adic norm of $c_\alpha\in L$ normalized by $||p||=1/p$. Similarly define
\[R_{\bar{D}}^{(r),+}:=\im\left(R_{\bar{D}}\otimes_{\mO_L[[x_1,\cdots,x_g]]} \mO_L\langle r{x_1},\cdots,r{x_g}\rangle\to R_{\bar{D}}^{(r)}\right),\]
\[E_{\bar{D}}^{(r),+}:=\im(E_{\bar{D}}\otimes_{R_{\bar{D}}}R_{\bar{D}}^{(r),+}\to E_{\bar{D}}^{(r)})\]
where $\mO_L\langle r{x_1},\cdots,r{x_g}\rangle$ is the unit ball of $L\langle r{x_1},\cdots,r{x_g}\rangle$. 
\end{dfn}

Note that $\mO_L\langle r{x_1},\cdots,r{x_g}\rangle$ is Noetherian and $p$-adically complete.
By Theorem \ref{thmWE}, $E_{\bar{D}}^{(r),+}$ is a finitely generated $\mO_L\langle r{x_1},\cdots,r{x_g}\rangle$-module, and hence also a $p$-adically complete Noetherian algebra (by the Artin-Rees lemma). Thus $E_{\bar{D}}^{(r)}$ is an $L$-Banach space with unit ball $E_{\bar{D}}^{(r),+}$. We will frequently use the fact that any finitely generated $R_{\bar{D}}^{(r)}$-submodule of $E_{\bar{D}}^{(r)}$ is  closed by the Artin-Rees Lemma, and we will always equip it with the subspace topology.

If $r'\in p^{\QQ}\cap (0,1)$ and $r<r'$, there are natural maps $R_{\bar{D}}^{(r')}\to R_{\bar{D}}^{(r)}$ and $E_{\bar{D}}^{(r')}\to E_{\bar{D}}^{(r)}$. Both have dense images.

Now for $m\geq 1$ consider the composite morphism
\[\bar{\rho}^m_r:\Gamma\xrightarrow{\rho} E_{\bar{D}}^\times\to (E_{\bar{D}}^{(r),+})^\times \to (E_{\bar{D}}^{(r),+}/(p^m))^\times.\]

\begin{lem} \label{lemlaGalact}
$\bar{\rho}^m_r$ has finite image, equivalently $\Gamma\to E_{\bar{D}}^{(r)}$ is continuous.
\end{lem}

\begin{proof}
 Suppose $r^i\in p^\ZZ$ for some integer $i>0$. Note that the image $J$ of the map
 \[R_{\bar{D}}/(p^m)\to   R_{\bar{D}}^{(r),+}/(p^m)\]
 is a quotient of 
 $\mO_L[[x_1,\cdots,x_g]]/(p^m,x_1^{im},\cdots,x_g^{im})$
  which is clearly finite. By Theorem \ref{thmWE} the image of
   \[E_{\bar{D}}/(p^m)\to   E_{\bar{D}}^{(r),+}/(p^m)\]
 is a finite module of $J$ and hence is finite as well.
\end{proof}

If $h\in \ker\bar{\rho}^1_r$, we have $1-h\in pE_{\bar{D}}^{(r),+}$ and hence
\[\log h:=-\sum_{i=1}\frac{(1-h)^i}{i}\]
converges in $E_{\bar{D}}^{(r),+}$.

\begin{dfn} \label{defnconstLiealg}
For $\Delta$ a closed subgroup of $\Gamma$, we define  $\Lie(\Delta,r)\subseteq E_{\bar{D}}^{(r)}$ as the $R_{\bar{D}}^{(r)}$-submodule generated by $\log (\ker\bar{\rho}^1_r\cap\Delta)$. 
\end{dfn}

\begin{rmk}
$\Lie(\Delta,r)=\Lie(\Delta',r)$ if $\Delta'$ is an open subgroup of $\Delta$. Indeed $h^{[\Delta:\Delta']}\in \Delta'$ if $h\in\Delta$ and $\log h^{[\Delta:\Delta']}=[\Delta:\Delta']\log h$ if $h\in\ker\bar{\rho}^1_r$.
\end{rmk}

\begin{lem}
$\Lie(\Delta,r)\subseteq E_{\bar{D}}^{(r)}$ is a Lie subalgebra with respect to the standard Lie bracket of $E_{\bar{D}}^{(r)}$, i.e. $[X,Y]=XY-YX$.
\end{lem}

We will regard $\Lie(\Delta,r)$ as the Lie algebra of $\Delta$ in $E_{\bar{D}}^{(r)}$. It's a finitely generated $R_{\bar{D}}^{(r)}$-module.

\begin{proof}
Let $\Lie'$ denote the $R_{\bar{D}}^{(r)}$-submodule of $E_{\bar{D}}^{(r)}$ generated by all $[\log h_1,\log h_2]$, $\log h_3$, $h_1,h_2,h_3\in\ker\bar{\rho}^1_r\cap \Delta$. It's enough to show the natural inclusion $\Lie(\Delta,r)\subseteq \Lie'$ is an equality. 

Let $\mathfrak{m}$ be a maximal ideal of $R_{\bar{D}}^{(r)}$. The residue field $k(\mathfrak{m})$ is a finite extension of $L$. Hence for any $k\geq 0$,  $E_{\bar{D}}^{(r)}/\mathfrak{m}^k E_{\bar{D}}^{(r)}$ is a finite-dimensional $L$-algebra. Fix $E_{\bar{D}}^{(r)}/\mathfrak{m}^k E_{\bar{D}}^{(r)}\cong L^d$. The left multiplication gives an embedding of $L$-algebras $E_{\bar{D}}^{(r)}/\mathfrak{m}^k E_{\bar{D}}^{(r)}\to M_d(L)$. It follows from Lemma \ref{lemlaGalact} that the composite map 
\[\Gamma \to (E_{\bar{D}}^{(r)})^\times \to (E_{\bar{D}}^{(r)}/\mathfrak{m}^k E_{\bar{D}}^{(r)})^\times \subseteq \GL_d(L) \]
is continuous. Therefore the image $H$ of $\Delta$ in $(E_{\bar{D}}^{(r)}/\mathfrak{m}^k E_{\bar{D}}^{(r)})^\times$ is a compact $p$-adic Lie group. Take a uniform pro-$p$ open subgroup $H_0$ of $H$. By standard results in theory of $p$-adic Lie groups \cite{DdSM-1999},  $\log H_0$ forms a Lie algebra over $\ZZ_p$. Hence both $\Lie'$  and $\Lie(\Delta,r)$ have the same images in $E_{\bar{D}}^{(r)}/\mathfrak{m}^k E_{\bar{D}}^{(r)}$.
We finish the proof by invoking the following standard result. 
\end{proof}

\begin{lem} \label{leminccomp}
Suppose that $R$ is a Noetherian ring, and $N_1$ and $N_2$ are submodules of a finitely generated $R$-module $N$. If for any maximal ideal $\mathfrak{m}$ of $R$ and $k\geq 0$, the image of $N_1$  in $N/\mathfrak{m}^k N$ contains the image of $N_2$ in $N/\mathfrak{m}^k N$, then $N_1\supseteq N_2$.
\end{lem}

\begin{proof}
For any maximal ideal $\mathfrak{m}$,
by our assumption and the Artin-Rees lemma, 
the $\mathfrak{m}$-adic completion $\widehat{(N_1)_{\mathfrak{m}}}$ of $N_1$ contains the $\mathfrak{m}$-adic completion of $N_2$. Hence $(N_1)_\mathfrak{m} \supseteq (N_2)_\mathfrak{m}$ because $\widehat{(N_i)_{\mathfrak{m}}}=(N_i)_{\mathfrak{m}}\otimes_{A_{\mathfrak{m}}} \widehat{A_{\mathfrak{m}}}$ and $\widehat{A_{\mathfrak{m}}}$ is faithfully flat over $A_{\mathfrak{m}}$. Since this is true for all maximal ideals, we get $N_1\supseteq N_2$.
\end{proof}

At the end of this subsection, we discuss a result about the tensor product of determinants that will be needed in the next section.
\begin{thm} \label{thm-tensorprod}
Let $A$ be a commutative ring. For $i=1,2$, let $D_i$ be an $A$-valued determinant on $\Gamma$ of dimension $d_i$. Then there is a naturally defined $A$-valued $d_1d_2$-dimensional determinant $D_1\otimes D_2$ on $\Gamma$, such that for any algebra homomorphism $f:A\to B$ and group homomorphisms $\rho_i:\Gamma\to \GL_{d_i}(B)$ satisfying $\det\circ\rho_i=f\circ D_i$ for $i=1,2$,
\[f\circ (D_1\otimes D_2)=\det \circ (\rho_1\otimes\rho_2).\]
Moreover if both $D_1$ and $D_2$ are continuous, $D_1\otimes D_2$ is continuous as well.
\end{thm}

\begin{proof}
There is a bijection between Chenvier's determinant and Lafforgue's pseudocharacters for $\GL_n$ cf. \cite{EM-2023} and \cite[\S 5.4.3]{Quast-2026} for a continuous version. See  \cite[\S  3.4]{Quast-2026} which constructed the tensor product of determinants in Lafforgue's setup.
\end{proof}

\subsection{A result of Sen} \label{subsecSen}
In this subsection we focus on Galois deformation rings and  reinterpret a result of Sen in terms of the Lie algebra introduced in the previous subsection. First we recollect some facts about the Sen operator. Fix a $p$-adic field extension $K$ of $\QQ_p$ in an algebraic closure $\overline{\QQ_p}$ of $\QQ_p$, i.e. the $p$-adic valuation on $K$ is discrete. Examples of $K$ include finite extensions of $\QQ_p$ and their maximal unramified extensions. Let  $\CC_p$ be the ($p$-adic) completion of $\overline{\QQ_p}$ and $G_K=\Gal(\overline{\QQ_p}/K)$.  

The following is a summary of results in \cite{Sen-1973}.

\begin{thm}[Sen] \label{thm-Sen}
For each finite-dimensional $\QQ_p$-representation $\rho:G_K\to \GL(V)$,  there is an element $\theta_V\in \End_{\CC_p}(V\otimes_{\QQ_p} \CC_p)$ (called the Sen operator) satisfying the following 
\begin{enumerate}
\item $\theta_V$ is functorial in $V$ and commutes with the semilinear action of $G_K$ on $V\otimes_{\QQ_p} \CC_p$.
\item $\theta_{V_1\oplus V_2}=(\theta_{V_1},\theta_{V_2})$.
\item $\theta_{V_1\otimes V_2}=\theta_{V_1}\otimes 1 +1\otimes \theta_{V_2}$.
\item $\theta_{\QQ_p(1)}=1$, where $\QQ_p(1)$ denotes the Tate twist.
\item  $\theta_V\in \Lie\rho(I_K)\otimes_{\QQ_p}\CC_p$ where $I_K$ denotes the inertia subgroup of $G_K$. Moreover $\Lie\rho(I_K)$ is the smallest such Lie subalgebra, i.e. if $\Lie'\subseteq \End_{\QQ_p}(V)$ is a  Lie subalgebra (over $\QQ_p$) and $\theta_V\in \Lie'\otimes_{\QQ_p} \CC_p$, then $\Lie\rho(I_K)\subseteq \Lie'$. (In fact, Sen showed that $\Lie\rho(I_K)$ is even the smallest $\QQ_p$-subspace having this property.)
\item If $L$ is a $p$-adic field extension of $K$ in $\overline{\QQ_p}$, the Sen operator obtained from $\rho|_{G_L}$ agrees with $\theta_V$.
\end{enumerate}
\end{thm}

In \cite[27]{Sen-1993} Sen defined the Sen operator for a large class of infinite-dimensional representations.

\begin{thm}[Sen] \label{thm-Sen-inf}
Let $A$ be a $\QQ_p$-Banach unital associative algebra (not necessarily commutative) and $\rho: G_K\to A^\times$ a continuous homomorphism with respect to the spectral topology on $A^\times$, i.e. a family of open neighborhood of $1\in A^\times$ is given by $1+p^nA^o$, where $A^o$ denotes the unit ball. Then there is an element $\theta_{(\rho,A)}\in A\widehat\otimes_{\QQ_p}\CC_p$ that is functorial in the pair $(\rho,A)$ and recovers the Sen operator for finite-dimensional representations when $A$ is a finite-dimensional matrix algebra.
\end{thm}

We give a sketch of the construction here. We view $A$ as a Banach space representation of $G_K$ by left multiplication. Let $K_m=K(\zeta_{p^m})$ and $K_\infty=K(\zeta_{p^\infty})$ in $\overline{\QQ_p}$. Denote by $H_K=\Gal(\overline{\QQ_p}/K_\infty)$ and $\Gamma_m=\Gal(K_\infty/K_m)$. Sen showed that \footnote{This reformulation is due to Berger-Colmez.} for $m$ sufficiently large, there is a natural isomorphism
\[(A\widehat\otimes_{\QQ_p}\CC_p)^{H_K,\Gamma_m-\an} \widehat\otimes_{K_m} \CC_p = A\widehat\otimes_{\QQ_p}\CC_p\]
where the superscript $\Gamma_m-\an$ denotes taking the $\Gamma_m$-analytic vectors. The logarithm of $p$-adic cyclotomic character gives a natural isomorphism $\Lie \Gamma_m\cong \Lie \ZZ_p^\times=\QQ_p$. The element $1\in\QQ_p\cong\Lie \Gamma_m$ acts $K_m$-linearly on $(A\widehat\otimes_{\QQ_p}\CC_p)^{H_K,\Gamma_m-\an}$ and its $\CC_p$-linear extension to $A\widehat\otimes_{\QQ_p}\CC_p$ commutes with the right multiplication by $A\widehat\otimes_{\QQ_p}\CC_p$, and hence it gives an element $\theta_{(\rho,A)}\in A\widehat\otimes_{\QQ_p}\CC_p$. This is independent of the choice of $m$.

In the rest of this section, we assume that $K$ is a finite extension of $\QQ_p$. Fix a continuous $n$-dimensional determinant of $G_K$ valued in $\FF$. We can apply constructions in the previous subsection and get $R^{(r)}_{\bar{D}},E^{(r)}_{\bar{D}}$.

\begin{thm} \label{thmexistSenop}
There is a unique element $\theta^{(r)}\in E_{\bar{D}}^{(r)}\widehat\otimes_{\QQ_p}\CC_p$ (the $p$-adic completion of $E_{\bar{D}}^{(r)}\otimes_{\QQ_p} \CC_p$) having the following universal property: for any $N\geq 1$ and continuous homomorphism of $\QQ_p$-algebras
\[f: E_{\bar{D}}^{(r)}\to M_N(\QQ_p),\]
the $N$-dimensional representation $G_K\xrightarrow{\rho} E_{\bar{D}}^{(r)}\xrightarrow{f} M_N(\QQ_p)$ has Sen operator  $f\otimes 1(\theta^{(r)})$, where $f\otimes 1: E_{\bar{D}}^{(r)}\widehat\otimes_{\QQ_p}\CC_p\to M_N(\QQ_p)\otimes_{\QQ_p} \CC_p$ is the completed extension of scalars of $f$ to $\CC_p$.
\end{thm}

\begin{proof}
The existence follows from Sen's result by noting that $G_K\to (E_{\bar{D}}^{(r)})^\times$ is continuous by Lemma \ref{lemlaGalact}. 

For the uniqueness, note that since $E_{\bar{D}}^{(r)}$ is finite over $R_{\bar{D}}^{(r)}$, the natural map 
\[E_{\bar{D}}^{(r)} \to \prod_{\mathfrak{m}}\prod_{k\geq 0} E_{\bar{D}}^{(r)}/(\mathfrak{m}^k)\]
is injective, where $\mathfrak{m}$ runs through all maximal ideals of $R_{\bar{D}}^{(r)}$. Each $E_{\bar{D}}^{(r)}/(\mathfrak{m}^k)$ is a finite-dimensional $\QQ_p$-representation of $G_K$. It follows from the universal property
that the composite map
\[E_{\bar{D}}^{(r)}\widehat\otimes_{\QQ_p} \CC_p \to E_{\bar{D}}^{(r)}/(\mathfrak{m}^k)\otimes_{\QQ_p}\CC_p \to \End_{\QQ_p}(E_{\bar{D}}^{(r)}/(\mathfrak{m}^k))\otimes_{\QQ_p}\CC_p\]
sends $\theta^{(r)}$ to $\theta_{E_{\bar{D}}^{(r)}/(\mathfrak{m}^k)}$, where the second map is induced by left multiplication of $ E_{\bar{D}}^{(r)}/(\mathfrak{m}^k)$. The uniqueness can be deduced from  the following lemma by taking $W=E_{\bar{D}}^{(r)}$ and $W_i=E_{\bar{D}}^{(r)}/(\mathfrak{m}^k)$.
\end{proof}

\begin{lem} \label{lemtensorCp}
Suppose $W$ is a $\QQ_p$-Banach space and $\{f_i:W\to W_i\}_{i\in I}$ is a family of continuous $\QQ_p$-linear maps to finite-dimensional $\QQ_p$-vector spaces such that
\[\prod f_i: W\to \prod_{i\in I} W_i\]
is injective. Then the natural map
\[\prod f_i\otimes 1: W\widehat\otimes_{\QQ_p} \CC_p \to  \prod_{i\in I} (W_i\otimes_{\QQ_p}\CC_p)\]
is injective as well.
\end{lem}

\begin{proof}
Choose an orthonormal basis $\{e_j\}_{j\in\ZZ_{\geq 0}}$ of $\CC_p$ over $\QQ_p$. Every element of $\CC_p$ can be written uniquely as $\sum_{j\geq 0} a_je_j$ for $(a_j)_j\in\prod_{j\geq 0} \QQ_p$ such that $\lim a_j=0 $. Similarly every element of $W\widehat\otimes_{\QQ_p} \CC_p$ is written uniquely as $\sum_{j\geq 0} b_j\otimes e_j$ for $(b_j)_j\in\prod_{j\geq 0} W$, $\lim b_j=0$. The claim follows easily from this description.  
\end{proof}

From the existence of $\theta^{(r)}$ we can easily construct the Sen polynomial over the generic fiber of $R_{\bar{D}}$ over $\CC_p$.

\begin{dfn}[Sen polynomial] \label{dfn-Senpoly}
Recall that $D^u:R_{\bar{D}}[[\Gamma]]\to R_{\bar{D}}$ denotes the universal determinant. We get a determinant of dimension $n$ 
\[D_{R_{\bar{D}}^{(r)}\widehat\otimes_{\QQ_p}\CC_p}: E_{\bar{D}}^{(r)}\widehat\otimes_{\QQ_p}\CC_p \to R_{\bar{D}}^{(r)}\widehat\otimes_{\QQ_p}\CC_p.\]
The characteristic polynomial $\chi(\theta^{(r)},t)=D_{R_{\bar{D}}^{(r)}\widehat\otimes_{\QQ_p}\CC_p}(t-\theta^{(r)}) \in R_{\bar{D}}^{(r)}\widehat\otimes_{\QQ_p}\CC_p[t]$ of $\theta^{(r)}$ is called the \textit{Sen polynomial} over $R_{\bar{D}}^{(r)}\widehat\otimes_{\QQ_p}\CC_p$. The negatives of its roots should be regarded as the Hodge-Tate-Sen weights over $\Spec R_{\bar{D}}^{(r)}$.

\begin{rmk}
Using that the Sen operator commutes with the semilinear action of $G_K$, one can further show that  $\chi(\theta^{(r)},t)\in R_{\bar{D}}^{(r)}\otimes_{\QQ_p}K[t]$.
\end{rmk}
\end{dfn}

\begin{rmk}
This gives a new proof of the main result of \cite{PQ-2025} for $\GL_n$. Note that our approach doesn't require any knowledge on the normality of the pseudo-deformation ring.
\end{rmk}

As in Definition \ref{defnconstLiealg}, we defined Lie algebra $\Lie(I_K,r)$ in $E_{\bar{D}}^{(r)}$, which is a finitely generated $R_{\bar{D}}^{(r)}$-submodule hence a closed subspace of $E_{\bar{D}}^{(r)}$.

\begin{thm} \label{thmfamSen}
$\theta^{(r)}\in \Lie(I_K,r)\widehat\otimes_{\QQ_p}\CC_p$. Moreover $\Lie(I_K,r)$ is the smallest $R_{\bar{D}}^{(r)}$-Lie subalgebra of  $E_{\bar{D}}^{(r)}$ having this property, i.e. if $\Lie'\subseteq E_{\bar{D}}^{(r)}$ is an $R_{\bar{D}}^{(r)}$-Lie subalgebra of $E_{\bar{D}}^{(r)}$ such that $\Lie'\widehat\otimes_{\QQ_p} \CC_p$ contains $\theta^{(r)}$, then $\Lie(I_K,r)\subseteq \Lie'$.
\end{thm}
\begin{proof}
Suppose that $\Lie'\subseteq E_{\bar{D}}^{(r)}$ and $\Lie'\widehat\otimes_{\QQ_p} \CC_p$ contains $\theta^{(r)}$ as in the theorem. Let $\mathfrak{m}$ be a maximal ideal of $R_{\bar{D}}^{(r)}$ and $k\geq 0$. 
Consider 
\[f:E_{\bar{D}}^{(r)}\to E_{\bar{D}}^{(r)}/(\mathfrak{m}^k) \to \End_{\QQ_p}(E_{\bar{D}}^{(r)}/(\mathfrak{m}^k))\]
as in the proof of Theorem \ref{thmexistSenop}. We have shown that $f\otimes 1(\theta^{(r)})=\theta_{E_{\bar{D}}^{(r)}/(\mathfrak{m}^k)}\in E_{\bar{D}}^{(r)}/(\mathfrak{m}^k)\otimes_{\QQ_p}\CC_p$. 
By Sen's result (Theorem \ref{thm-Sen}), $f(\Lie(I_K,r))\subseteq f(\Lie')$, equivalently the image of $\Lie(I_K,r)$ in $E_{\bar{D}}^{(r)}/(\mathfrak{m}^k)$ is contained in the image of $\Lie'$. We conclude that $\Lie(I_K,r)\subseteq \Lie'$ by invoking Lemma \ref{leminccomp}. Note that the same argument also shows that the image of $\Lie(I_K,r)\widehat\otimes_{\QQ_p} \CC_p$ in $E_{\bar{D}}^{(r)}/(\mathfrak{m}^k)\widehat\otimes_{\QQ_p} \CC_p$ contains $f\otimes 1 (\theta^{(r)})$, and hence again by Lemma \ref{leminccomp}, we have $\theta^{(r)}\in \Lie(I_K,r)\widehat\otimes_{\QQ_p}\CC_p$.
\end{proof}

In practice, it's better to work with an algebra than a Lie algebra.

\begin{dfn}
We define $B_{\bar{D}}^{(r)}\subseteq E^{(r)}_{\bar{D}}$ as the $R_{\bar{D}}^{(r)}$-subalgebra generated by $\Lie(I_K,r)$. It is a finitely generated $R_{\bar{D}}^{(r)}$-submodule hence a natural $p$-adic Banach subalgebra of $E_{\bar{D}}^{(r)}$.
\end{dfn}

\begin{lem} \label{lem-famSenthm}
As an $R_{\bar{D}}^{(r)}$-module, $B_{\bar{D}}^{(r)}$ is generated by $\rho(\ker \bar{\rho}_r^2\cap I_K)$, and hence we get a continuous homomorphism
\[\rho_K:=\rho|_{\ker \bar{\rho}_r^2\cap I_K} \ker \bar{\rho}_r^2\cap I_K\to (B_{\bar{D}}^{(r)})^\times. \]
\end{lem}

\begin{proof} 
For $g\in \ker \bar{\rho}_r^2\cap I_K$, we can write $\rho(g)=\exp(\log(\rho(g)))=\sum_{i=0} \log(\rho(g))^i/i!$ which converges in $B_{\bar{D}}^{(r)}$. On the other hand, the $R_{\bar{D}}^{(r)}$-submodule generated by $\rho(\ker \bar{\rho}_r^2\cap I_K)$ is closed and is a subalgebra. Thus it contains $\log(\rho(\ker \bar{\rho}_r^2\cap I_K))$, which agrees with $\log(\rho( I_K))$ because $\ker \bar{\rho}_r^2\cap I_K$ is an open subgroup of $I_K$ by Lemma \ref{lemlaGalact}.
\end{proof}

Clearly $\theta^{(r)}\in B_{\bar{D}}^{(r)}\widehat\otimes_{\QQ_p}\CC_p$.

\begin{cor} \label{cor-famSen}
Let $\bar{B}$ be a quotient of $B_{\bar{D}}^{(r)}$ by a two-sided ideal and denote by $\bar{\theta}\in\bar{B}\widehat\otimes_{\QQ_p} \CC_p$ the image of $\theta^{(r)}$. For any $R_{\bar{D}}^{(r)}$-subalgebra $\bar{B}'$ of $\bar{B}$, if $\bar{\theta}\in \bar{B}'\widehat\otimes_{\QQ_p} \CC_p$, then $\bar{B}'=\bar{B}$.
\end{cor}

\begin{proof}
For every maximal ideal $\mathfrak{m}$   of $R_{\bar{D}}^{(r)}$ and $k\geq 0$, consider the continuous homomorphism
\[ f:\ker \bar{\rho}_r^2\cap I_K \xrightarrow{\rho_K} (B^{(r)}_{\bar{D}})^\times \to (\bar{B}/\mathfrak{m}^k)^\times.\]
Since $\overline{\QQ_p}^{\ker \bar{\rho}_r^2\cap I_K}$ is a finite extension of the maximal unramified extension of $K$, we can apply Sen's result and deduce that there is a Sen operator $\theta_{\bar{B}/\mathfrak{m}^k} \in (\bar{B}/\mathfrak{m}^k) \widehat\otimes_{\QQ_p}\CC_p$ which agrees with $f\otimes 1 (\theta^{(r)})$. Now we can repeat the same argument as in the proof of Theorem \ref{thmfamSen} and conclude that $\bar{B}'$ contains the image of $\Lie(I_K,r)$ in $\bar{B}$, and hence is equal to $\bar{B}$ because it is an $R_{\bar{D}}^{(r)}$-subalgebra.
\end{proof}

\subsection{Extension lemma}
In this subsection, we explain a technical lemma which allows one to extend an action of  $B_{\bar{D}}^{(r)}$. This lemma will be used later to construct some plectic Lie algebra action. Our setup will be the same as in the previous subsection.
\begin{dfn}\label{dfn-weakTopology}
Let \( W,V \) be Banach spaces over \( \QQ_{p} \) with unit balls \( W^{\circ},V^{\circ} \). We denote by \( \Hom_{\mrm{cont}}(W,V) \)
 the space of continuous linear homomorphisms from \( W \) to \( V \), which has a lattice \( \Hom(W^{\circ},V^{\circ}) \). We equip \( \Hom(W^{\circ},V^{\circ}) \)
with the weak topology (i.e. the pointwise convergence topology), and \( \Hom_{\mrm{cont}}(W,V)=\Hom(W^{\circ},V^{\circ}) [1/p] \) with the colimit topology.
We write \( \End^{\mrm{cont}}(W):=\Hom_{\mrm{cont}}(W,W) \). 
\end{dfn}
\begin{lem} \label{lem-extSen}
Let $W$ be a Banach $R_{\bar{D}}^{(r)}$-module and $W^a=\bigoplus_{i=0}^\infty W_i$ a dense subspace of $W$. Assume that each $W_i$ is  finite-dimensional over $\QQ_p$ and is a $R_{\bar{D}}^{(r)}$-submodule.
Denote by $\End^a\subseteq \End^{\mathrm{cont}} W$ the closed subspace of continuous $R_{\bar{D}}^{(r)}$-linear endomorphisms preserving each $W_i$. Suppose that 
\begin{itemize}
\item there is an element $\theta_W\in \End^a\widehat\otimes_{\QQ_p}\CC_p$ and
\item there is an  $R_{\bar{D}}^{(r)}$-linear continuous action of $B_{\bar{D}}^{(r)}$ on each $W_i$ such that the induced action of $\theta^{(r)}$ on $W_i\otimes\CC_p$ agrees with $\theta_W|_{W_i\otimes\CC_p}$.
\end{itemize}
Then the action of  $B_{\bar{D}}^{(r)}$ on $W^a$ can be extended (necessarily uniquely) to a continuous action on $W$ and the induced action of $\theta^{(r)}$ on $W\widehat\otimes_{\QQ_p}\CC_p$ is equal to $\theta_W$.
\end{lem}

\begin{proof} 
All the $\End$'s below are considered as the algebra of continuous endomorphisms. So we will drop superscripts $\mathrm{cont}$ from now on. 
Consider $\prod_{i=0} \End_{R_{\bar{D}}^{(r)}}(W_i)$, viewed as a Fr\'echet space by expressing it as a projective limit of Banach spaces $\varprojlim_i (\End_{R_{\bar{D}}^{(r)}}(W_0)\oplus \cdots\oplus \End_{R_{\bar{D}}^{(r)}}(W_i))$.
Since $W^a$ is dense, we have a natural continuous inclusion of $R_{\bar{D}}^{(r)}$-algebras
\[\End^a \subseteq \prod_{i=0} \End_{R_{\bar{D}}^{(r)}}(W_i). \]
On the other hand, the action of $B_{\bar{D}}^{(r)}$  on each $W_i$ gives a continuous $R_{\bar{D}}^{(r)}$-linear homomorphism
\[ B_{\bar{D}}^{(r)} \to \prod_{i=0} \End_{R_{\bar{D}}^{(r)}}(W_i).\]
Let $\bar{B}$ be its image equipped with the quotient topology. It's enough to show that $\bar{B}\subseteq \End^a$. Set $\bar{B}'=\bar{B}\cap \End^a$. It is an $R_{\bar{D}}^{(r)}$-subalgebra of $\bar{B}$. By Corollary \ref{cor-famSen}, it suffices to show $\bar\theta\in \bar{B}'\widehat\otimes_{\QQ_p}\CC_p$ where  $\bar\theta\in \bar{B}\widehat\otimes_{\QQ_p}\CC_p$ denotes the image of $\theta^{(r)}$. Our assumption implies that
\[\bar\theta\in \bar{B}\widehat\otimes_{\QQ_p}\CC_p\cap \End^a\widehat\otimes_{\QQ_p}\CC_p\]
agrees with $\theta_W$ in $\displaystyle\left(\prod_{i}\End_{R_{\bar{D}}^{(r)}}(W_i)\right)\widehat\otimes_{\QQ_p}\CC_p=\prod_{i}(\End_{R_{\bar{D}}^{(r)}}(W_i)\otimes_{\QQ_p}\CC_p)$. We finish the proof by noting that
\[\bar{B}\widehat\otimes_{\QQ_p}\CC_p\cap \End^a\widehat\otimes_{\QQ_p}\CC_p=(\bar{B}\cap \End^a)\widehat\otimes_{\QQ_p}\CC_p,\]
which can be proved directly by choosing an orthonormal basis of $\CC_p$ over $\QQ_p$ and arguing as  in the proof of Lemma \ref{lemtensorCp}.
\end{proof}

\begin{rmk}
Let $\End^a_{\CC_p}\subseteq \End^{\mathrm{cont}} W\widehat\otimes_{\QQ_p}\CC_p =\End_{\CC_p}^{\mathrm{cont}}(W\widehat\otimes_{\QQ_p}\CC_p)$ be the subspace of endomorphisms preserving each $W_i\otimes_{\QQ_p}\CC_p$ and commuting with $R^{(r)}_{\bar{D}}$. There is a natural isomorphism
\[\End^a\widehat\otimes_{\QQ_p}\CC_p \cong \End^a_{\CC_p}.\]
Indeed by choosing topological generators $g_1,\cdots, g_n$ of $R^{(r)}_{\bar{D}}$, we can express $\End^a$ as follows
\[\End^a=\ker\left( \End(W) \xrightarrow{([\bullet,g_1],\cdots,[\bullet,g_n];g)}\bigoplus_{i=1}^{n} \End (W)\oplus \prod_{i=0}^{\infty} \Hom_{\QQ_p}(W_i,W/W_i)\right)
\]
where $[\bullet,g_i]$ sends $f\in\End(W)$ to $fg_i-g_if$ 
and $g$ is the product of natural maps $\End(W)=\Hom(W,W)\to \Hom_{\QQ_p}(W_i,W/W_i)$. The desired isomorphism follows upon taking the completed tensor products with $\CC_p$.
\end{rmk}

\section{Plectic Lie algebra} \label{sec-PLA}

The goal of this section is to study the Lie algebra of the plectic Galois group introduced by Nekov\'a\v{r}-Scholl \cite{NS-2016} using results from the previous section. Throughout this section, $F$ denotes a finite extension of $\QQ$ of degree $d$, and we fix an algebraic closure $\overline{\QQ}$ of $\QQ$ and an embedding $ \iota_p:\overline{\QQ}\to\overline{\QQ_p} $ into an algebraic closure of $\QQ_p$. 

\subsection{Plectic Galois group} 
\begin{dfn}
The \emph{plectic Galois group} of $F$ is defined as 
\(G^{plec}_{F}:=\mrm{Aut}_{F}(F\otimes_{\QQ}\overline{\QQ})\) the group of automorphisms of $F$-algebra $F\otimes_{\QQ}\overline{\QQ}$.
\end{dfn}

The absolute Galois group $G_{\QQ}=\Gal(\overline{\QQ}/\QQ)$ is a natural subgroup of $G^{plec}_{F}$ by acting on the $\overline{\QQ}$ in $F\otimes_{\QQ}\overline{\QQ}$. There is a natural isomorphism
\[F\otimes_{\QQ}\overline{\QQ}\cong\prod_{\tau:F\to \overline{\QQ}}\overline{\QQ}\]
mapping $a\otimes b\in F\otimes_{\QQ}\overline{\QQ}$ to $(\tau(a)b)_{\tau}$, where $\tau$ runs through all embeddings of $F$ into $\overline{\QQ}$. It follows that $\displaystyle \prod_{\tau:F\to \overline{\QQ}} \Gal(\overline{\QQ}/\tau(F))$ is a normal subgroup of $G^{plec}_{F}$. The conjugate action of $\sigma\in G_{\QQ}\subseteq G^{plec}_{F}$ on $\displaystyle \prod \Gal(\overline{\QQ}/\tau(F))$ will be denoted by $\Ad_\sigma$. Explicitly for $(g_\tau)_\tau\in \prod \Gal(\overline{\QQ}/\tau(F))$,
\[\Ad_\sigma(g_\tau)_\tau=(\sigma g_{\sigma^{-1}\circ\tau} \sigma^{-1})_{\tau}.\]
Hence $\Ad_\sigma$ sends $\Gal(\overline{\QQ}/\tau(F))$ to $\Gal(\overline{\QQ}/\sigma\circ \tau(F))$.

Fix an algebraic closure $\overline{F}$ of $F$. Let $E$ be a nonarchimedean field extension of $\QQ_p$.

\begin{dfn}[Plectic induction]\label{dfnPlecticInduction}
Given a continuous  representation \( \rho:\Gal(\overline{F}/F)\to \GL_{n}(E)\), the \textit{plectic induction} $\Ind^{plec}_{F}\rho: G_{F}^{plec}\to \GL_{n^d}(E)$ is defined as follows: $\rho$ defines an $n$-dimensional $E$-local system $\mathbb{E}$ on the pro-\'etale site of $\Spec F$. The plectic group $G_F^{plec}$ acts naturally on the pro-\'etale cover $\Spec (F\otimes_{\QQ} \overline\QQ)=\bigcup_\tau \Spec F\otimes_{F,\tau}\overline\QQ$ of $\Spec F$, permuting its connected components, hence we get a natural action $\Ind^{plec}_{F}\rho$ of $G_F^{plec}$ on the $n^d$-dimensional $E$-vector space
\[\bigotimes_{\tau:F\to \overline\QQ}\Gamma(\Spec F\otimes_{F,\tau}\overline\QQ,\mathbb{E}).\]
\end{dfn}

\begin{rmk} \label{rmk-tenindvsplecind}
$\Ind^{plec}_{F}\rho|_{G_{\QQ}} \cong\oInd^{\QQ}_{F}\rho$, where $\oInd^{\QQ}_{F}\rho$ is the \textit{tensor induction} of $\rho$ from $F$ to $\QQ$.
\end{rmk}

\begin{rmk} \label{rmk-indcomptensor}
Given two continuous  representations \( \rho_i:\Gal(\overline{F}/F)\to \GL_{n_i}(E)\), $i=1,2$, there is a natural isomorphism of representations of $G_{F}^{plec}$
\[\Ind^{plec}_{F}(\rho_1\otimes\rho_2)\cong \Ind^{plec}_{F}\rho_1\otimes \Ind^{plec}_{F}\rho_2.\]
This can be seen easily from the construction.
\end{rmk}

\begin{rmk} \label{rmk-resprodGtauF}
For each $\tau:F\to \overline\QQ$, fix an isomorphism $\tilde\tau:\overline{F}\to\overline\QQ$ extending $\tau$. We thus get an identification $\Gal(\overline{\QQ}/\tau(F))\cong \Gal(\overline{F}/F)$ and, under this isomorphism $\rho$ gives rise to a representation 
\[\rho_{\tau}:\Gal(\overline{\QQ}/\tau(F))\to\GL_d(E)\] 
whose isomorphism class doesn't depend on the choice of $\tilde\tau$. Indeed a different choice of $\tilde\tau$ only changes $\Gal(\overline{\QQ}/\tau(F))\cong \Gal(\overline{F}/F)$ by an inner automorphism. It follows from the construction that
\[\Ind^{plec}_{F}\rho|_{\prod_{\tau}\Gal(\overline{\QQ}/\tau(F))}\cong \boxtimes_{\tau} \rho_\tau.\] 
In particular, $\Ind^{plec}_{F}\rho$ is irreducible if $\rho$ is irreducible. Let $\tilde{F}$ denote the Galois closure of $F$ in $\overline{\QQ}$, then $\Gal(\overline{\QQ}/\tilde{F})= G_\QQ\cap \prod_{\tau}\Gal(\overline{\QQ}/\tau(F))$. Hence
\begin{equation} \label{tenind|_GQ}
\oInd^{\QQ}_{F} \rho|_{\Gal(\overline{\QQ}/\tilde{F})}=\Ind^{plec}_{F}\rho|_{\Gal(\overline{\QQ}/\tilde{F})}\cong \otimes_{\tau} \rho_\tau|_{\Gal(\overline{\QQ}/\tilde{F})}.
\end{equation}
\end{rmk}

Let $\Sigma_F$ be the set of embeddings $F\to \overline{\QQ}$ and $S_{\Sigma_F}$ denote the symmetric group on the set $\Sigma_F$. After fixing an isomorphism $\tilde\tau:\overline{F}\to\overline\QQ$ for  each $\tau\in\Sigma_F$, we get an isomorphism of $F$-algebras
\[F\otimes_{\QQ}\overline{\QQ}\cong\prod_{\tau:F\to \overline{\QQ}}\overline{F}=\overline{F}^{\Sigma_F},\]
and this gives an isomorphism
\[G_{F}^{plec}\cong\Gal(\overline{F}/F)^{\Sigma_F}\rtimes S_{\Sigma_F}.\]
Now let $R$ be a commutative ring. Given a $R[\Gal(\overline{F}/F)]$-module $V$, the tensor product (over $R$)
\[\bigotimes_{\tau\in\Sigma_F} V\]
carries a natural action of $\Gal(\overline{F}/F)^{\Sigma_F}$ and an action of $S_{\Sigma_F}$ by permuting $\Sigma_F$, which defines a natural representation of $\Gal(\overline{F}/F)^{\Sigma_F}\rtimes S_{\Sigma_F}$ over $R$. This is nothing but the plectic induction.

\begin{rmk} \label{rmk-transfer}
Fix an isomorphism $\overline{F}\cong \overline{\QQ}$, and hence  $\Gal(\overline{F}/F)$ can be viewed as a subgroup of $G_\QQ$ of finite index.
If $\psi: \Gal(\overline{F}/F)\to E^\times $ is a character, then $\Ind^{plec}_{F}\psi$ is a character as well. The tensor induction $\oInd^{\QQ}_{F} \psi: G_\QQ \to E^\times$ is nothing but  $\psi\circ Ver$ where $Ver: G_{\QQ} \to \Gal(\overline{F}/F)^{\mathrm{ab}}$ denotes the \textit{transfer} map.
\end{rmk}

We end this subsection with some miscellaneous results that will be used  in the next section. 

\begin{prop} \label{proptenindirred}
Let  \( \rho:\Gal(\overline{F}/F)\to \GL_{2}(\CC_p)\) be a continuous representation such that $\Lie\rho(\Gal(\overline{F}/F))$ contains an $sl_2$; equivalently, the Lie algebra representation of $\Lie\rho(\Gal(\overline{F}/F))$  on $\CC_p^{2}$
is absolutely irreducible. 
\begin{enumerate}
\item The tensor induction $\oInd^{\QQ}_{F} \rho$ has a unique irreducible $G_{\QQ}$-subrepresentation of the largest dimension among all the subrepresentations that are irreducible $\Lie(\oInd^{\QQ}_{F} \rho(G_{\QQ}))$-representations. It has the same Hodge-Tate-Sen weights as those of $\oInd^{\QQ}_{F} \rho$ without counting multiplicities.
\item For each $\tau:F\to\overline{\QQ_p}$, we denote the $\tau$-Hodge-Tate-Sen weights of $\rho$ by $a_\tau,b_\tau$ and let $k_\tau=a_\tau-b_\tau$, cf. Definition \ref{defn-partSenpoly} below. Then $\oInd^{\QQ}_{F} \rho$ is absolutely irreducible if the following holds
\begin{itemize}
\item all $k_\tau$'s and $-k_\tau$'s are distinct. In particular, $k_\tau\neq 0$.
\end{itemize}
\end{enumerate}
\end{prop}

\begin{proof}
We may twist $\rho$ with a character of $\Gal(\overline{F}/F)$ and assume that $\det \rho$ is of finite order. Hence, $\Lie\rho(\Gal(\overline{F}/F))\subseteq sl_2(\CC_p)$.

\begin{lem}
Let $I$ be a finite set. Let $\mathfrak{g}$ be a Lie subalgebra of $\prod_{i\in I}sl_2(\CC_p)$ which projects surjectively onto each factor $sl_2(\CC_p)$. Then there exist:
\begin{itemize}
\item a partition $I=I_1\cup\cdots\cup I_m$ for non-empty subsets $I_j$ of $I$;
\item for each $j\in\{1,\cdots,m\}$ and each $i\in I_j$ an isomorphism of Lie algebras $f_{ji}:sl_2(\CC_p)\xrightarrow{\sim} sl_2(\CC_p)$
\end{itemize}
such that 
\begin{eqnarray} \label{eqn-Lemma2.1Nekovar}
\mathfrak{g}=\im\left( \prod_{j=1}^msl_2(\CC_p) \xrightarrow{\Delta} \prod_{j=1}^m(sl_2(\CC_p))^{I_j} \xrightarrow{(f_{ji})} \prod_{j=1}^m \prod_{i\in I_j} sl_2(\CC_p) =\prod_{i\in I}sl_2(\CC_p)
\right)
\end{eqnarray}
where $\Delta=(\Delta_j)_{j\in J}$ and each $\Delta_j: sl_2(\CC_p)\to (sl_2(\CC_p))^{I_j} $ is the diagonal map.
\end{lem}

\begin{proof}
\cite[Prop. 2.1]{Nekovar-2018}.
\end{proof}

Denote by $M$ the $\CC_p$-Lie subalgebra generated by $\Lie(\oInd^{\QQ}_{F} \rho(G_{\QQ}))$ in $\End_{\CC_p}(\CC_p^{2^d}) =gl_{2^d}(\CC_p)$. Using the notation in Remark \ref{rmk-resprodGtauF}, for each $\tau:F\to\overline{\QQ_p}$, we denote by $M_\tau\subseteq gl_{2}(\CC_p)$ the $\CC_p$-Lie subalgebra generated by $\Lie\rho_\tau(\Gal(\overline{\QQ}/\tau(F)))$, and $M_\tau= sl_{2}(\CC_p)$ by our assumption. There is a natural inclusion map
\[M\subseteq \prod_{\tau} M_\tau\]
which  projects surjectively onto each factor. Hence, by the previous lemma,  there exist a partition $\Sigma_F=I_1\cup\cdots\cup I_n$ for non-empty subsets $I_j$ of $I$ and isomorphisms $f_{ji}:sl_2(\CC_p)\xrightarrow{\sim} sl_2(\CC_p)$, $j=1,\cdots,m$, $i\in I_j$
such that $M$ is of the form in \eqref{eqn-Lemma2.1Nekovar}. By Theorem \ref{thm-Sen-inf}, we have the Sen operator in $M\widehat\otimes\CC_p$ for $\rho$, and we let $\theta\in M$ be its base change along the natural map \( \CC_{p}\widehat{\otimes}\CC_{p}\to \CC_{p} \). Similarly let $\theta_\tau\in M_\tau$ be the base change of the $\tau$-Hodge-Tate-Sen operator, i.e. the Sen operator for $\rho_\tau$ at the $p$-adic place induced by $F\xrightarrow{\tau}\overline{\QQ}\to \overline{\QQ_p}$. See also Definition \ref{defn-partSenpoly} below.
 The negatives of the eigenvalues of $\theta_\tau\in sl_2(\CC_p)$ are $a_\tau,b_\tau$. It follows from \eqref{tenind|_GQ} that in $\prod_{\tau} M_\tau$,
\[\theta=\sum_\tau \theta_\tau.\] 
Hence, for each $I_j$, the set of eigenvalues of $\theta_\tau$ for $\tau$ in $I_j$ must be the same and will be denoted by $a_j,b_j$. This gives the second part of Proposition \ref{proptenindirred} because by our assumption $M=\prod_{\tau} M_\tau$ and acts irreducibly on $\CC_p^{2^d}$.

In general, after identifying $M$ with  $sl_2(\CC_p)^{m}$ as in  \eqref{eqn-Lemma2.1Nekovar}, we see that $\oInd^{\QQ}_{F} \rho$ as a representation of $M$ is
\[\mathrm{Std}^{\otimes n_1}\boxtimes \mathrm{Std}^{\otimes n_2}\boxtimes \cdots \boxtimes \mathrm{Std}^{\otimes n_m}\]
where $n_j=|I_j|$, $j=1,\cdots,m$. It has an irreducible subrepresentation $V:=\Sym^{\otimes n_1}\boxtimes  \cdots \boxtimes \Sym^{\otimes n_m}$ with $\Sym^{\otimes n_j}\mathrm{Std}$ denoting the $n_j$-th symmetric power of the standard representation. The eigenvalues of $\theta$ in $\oInd^{\QQ}_{F} \rho$ are
\[(s_1a_1+(n_1-s_1)b_1)+\cdots+(s_ma_m+(n_m-s_m)b_m),\,s_1\in\{0,\cdots,n_1\},\cdots,s_m\in\{0,\cdots,n_m\},\]
This is exactly the set of eigenvalues of $\theta$ in $V$ (without counting multiplicities). A consideration of the highest weight shows that $V$ has multiplicity one and has the largest 
dimension in $\oInd^{\QQ}_{F} \rho$ as an irreducible $M$-subrepresentation. Thus $V$ is $G_{\QQ}$-stable as well.
\end{proof}

\subsection{Partial Sen operators and Lie algebra} \label{subsec-PSoaLa}
Next we consider the deformation ring.   Fix a finite set of primes $S$ of $F$ containing all $p$-adic places and denote by $G_{F,S}$ the Galois group of the maximal extension of $F$ unramified away from $S$.
 Let $\bar{D}$ be a continuous $n$-dimensional determinant of $G_{F,S}$ valued in $\FF$ and $R_{\bar{D}}$ the universal completed local $\mO_L$-algebra  parametrizing all continuous lifts of  $\bar{D}$.  We use $D^u:R_{\bar{D}}[[G_{F,S}]]\to R_{\bar{D}}$ to denote the universal determinant.
 
 For an embedding $\tau: F\to \overline{\QQ}$, we choose an isomorphism $\tilde\tau:\overline{F}\to\overline\QQ$ extending $\tau$ and thus get an isomorphism $\Gal(\overline{\QQ}/\tau(F))\cong \Gal(\overline{F}/F)$. Denote by $G_{\tau(F),\tau(S)}$ the quotient of $\Gal(\overline{\QQ}/\tau(F))$ corresponding to the quotient $G_{F,S}$ of \( \Gal(\bar{F}/F) \). Then $\bar{D}$ gives rise to a continuous $n$-dimensional determinant $\bar{D}_\tau$ of $G_{\tau(F),\tau(S)}$, which is independent of the choice $\tilde\tau$ (because $\bar{D}$ is invariant by conjugation).  The universal deformation ring $R_{\bar{D}_\tau}$ of $\bar{D}_\tau$ is canonically isomorphic to $R_{\bar{D}}$. Now we view each $G_{\tau(F),\tau(S)}$ as a quotient of the product $\prod_{\tau\in\Sigma_F}  G_{\tau(F),\tau(S)}$, and apply repeatedly Theorem \ref{thm-tensorprod} to $D_\tau$ on $G_{\tau(F),\tau(S)}$ over $R_{\bar{D}}$, and get a continuous $n^d$-dimensional determinant
 \[D^{up}=\otimes_{\tau\in\Sigma_F} D_\tau:R_{\bar{D}}[[\prod_{\tau\in\Sigma_F}  G_{\tau(F),\tau(S)}]]\to R_{\bar{D}}.\]
 Note that $\prod_{\tau\in\Sigma_F}  G_{\tau(F),\tau(S)}$ defines a quotient of $\prod\Gal(\overline{\QQ}/\tau(F))$. Let $I_S$ denote the kernel of this quotient map which is a normal subgroup of \( G^{plec}_{F} \), and let 
 \[G_{F,S}^{plec}:=G^{plec}_F/I_S.\]

 \begin{prop} \label{prop-plecinddet}
 There is a continuous $n^d$-dimensional determinant
 \[D^{plec}: R_{\bar{D}}[[G^{plec}_{F,S}]][1/p] \to R_{\bar{D}}[1/p]\]
 extending $D^{up}$ such that for each continuous homomorphism $f:R_{\bar{D}}\to\CC_p$, the  semisimple representation $G^{plec}_{F,S}\to \GL_{n^d}(\CC_p)$ associated to $f\circ D^{plec}$ is isomorphic to $\Ind^{plec}_{F}\rho_f$, where $\rho_f:G_{F,S}\to \GL_d(\CC_p)$ is the semisimple representation associated to $f\circ D^u$. 
 \end{prop}
 
 \begin{rmk}
 It is probably true that $D^{plec}$ exists without inverting $p$ but we are satisfied with this rational statement here as it is enough for the purpose of this paper.
 \end{rmk}
 
\begin{proof}
Note that $ R_{\bar{D}}[1/p]$ is a $\QQ$-algebra. By  \cite[Prop.1.27]{Chenevier-2014}, it's enough to construct a $n^d$-dimensional pseudocharacter $T^{plec}: G^{plec}_{F,S}\to R_{\bar{D}}[1/p]$ satisfying the desired property. Recall that the notion of pseudocharacter was introduced by Taylor in \cite{Taylor-1991} (called pseudorepresentation in the reference). Explicitly, it is a map such that
\begin{itemize}
\item $T^{plec}(1)=n^d$;
\item $T^{plec}(g_1g_2)=T^{plec}(g_2g_1)$ for all $g_1,g_2\in G^{plec}_{F,S}$ and
\item $\sum_{\sigma\in S_{n^d+1}} \mathrm{sgn}(\sigma) T^{plec}_\sigma(g_1,\cdots,g_{n^d+1})=0$ for all $g_1,\cdots,g_{n^d+1}\in G^{plec}_{F,S}$.
\end{itemize}
Here $S_{n^d+1}$ denotes the symmetric group of degree $n^d+1$, $\mathrm{sgn}(\sigma)$ denotes the sign of $\sigma$ and $T^{plec}_\sigma: (G^{plec}_{F,S})^{n^d+1}\to R_{\bar{D}}[1/p]$ is the function 
\[(g_1,\cdots,g_{n^d+1})\mapsto T^{plec}(g_{i_1^{(1)}}\cdots g_{i_{r_1}^{(1)}})T^{plec}(g_{i_1^{(2)}}\cdots g_{i_{r_2}^{(2)}})\cdots T^{plec}(g_{i_1^{(s)}}\cdots g_{i_{r_s}^{(s)}})\]
where $\sigma$ has the cycle decomposition $(i_1^{(1)}\cdots i_{r_1}^{(1)})(i_1^{(2)}\cdots i_{r_2}^{(2)})\cdots (i_1^{(s)}\cdots i_{r_s}^{(s)})$. It's known that the trace of an $n^d$-dimensional representation defines an $n^d$-dimensional pseudocharacter \cite[Th. 1]{Taylor-1991}.

To simplify notation, we will write $R=R_{\bar{D}}$, $H=G_{F,S}$ and identify $G^{plec}_{F,S}$ with $H^{\Sigma_F}\rtimes S_{\Sigma_F}$ as in the previous subsection.  Let $E$ be the maximal Cayley-Hamilton quotient of $R[[H]]$. 
The quotient determinant $\bar{D}^{CH}: E\to R$ is Cayley-Hamilton of dimension $n$  (cf. Section \ref{subsec-Defring}). Hence if we invert $p$ and apply Procesi's result \cite{Procesi-1987} (see also \cite[Rem.1.25]{Chenevier-2014}), we can embed $(E[1/p], \bar{D}^{CH})$ into $(M_{n}(B), \det)$ for some commutative faithful $R[1/p]$-algebra $B$. In particular, by composing with the map $H\to R[[H]]\to E$, we obtain an $n$-dimensional representation $\rho$ of $H$ over $B$ whose trace agrees with the pseudocharacter defined by $D^u$. Consider the plectic induction  $\Ind^{plec}_{F}\rho: G^{plec}_{F,S}\to \GL_{n^d}(B)$ and let
\[T^{plec}=\mathrm{Tr} \circ \Ind_F^{plec} \rho: G^{plec}_{F,S}\to B\]
be its trace.
We claim that $T^{plec}$ takes values in $R[1/p]\subseteq B$ and thus gives the desired pseudocharacter. To see this, it's enough to compute the trace explicitly. Denote the trace function of $\rho$ by $T$. A direct calculation shows that for $(g,\sigma)\in H^{\Sigma_F}\rtimes S_{\Sigma_F}$ with $g=(g_\tau)_{\tau\in\Sigma_F}$ and $\sigma\in S_{\Sigma_F}$, 
\[T^{plec}(g,\sigma)=T_\sigma(g),\]
where $T_\sigma$ is defined similarly using the cycle decomposition of $\sigma$ as above. Since $T$ takes values in $R$, we conclude that $T_\sigma$ takes value in the image of $R$ in $B$. 
\end{proof}
 
We denote by $E^{plec}$ the maximal Cayley-Hamilton quotient of $R_{\bar{D}}[[G^{plec}_{F,S}]][1/p]$.

\begin{prop}
$E^{plec}$ is a finitely generated module over $R_{\bar{D}}[1/p]$.
\end{prop}

\begin{proof}
Because $\Sigma_F$ is finite, there is an isomorphism of $R_{\bar{D}}$-modules
\[R_{\bar{D}}[[G^{plec}_{F,S}]]\cong R_{\bar{D}}[[\prod_{\tau\in\Sigma_F}  G_{\tau(F),\tau(S)}]]\otimes_{R_{\bar{D}}} R_{\bar{D}}[S_{\Sigma_F}].\]
It's enough to show that the image of $R_{\bar{D}}[[\prod_{\tau\in\Sigma_F}  G_{\tau(F),\tau(S)}]]$ in $E^{plec}$ is a finitely generated $R_{\bar{D}}$-module. By  Theorem \ref{thm-tensorprod}, the tensor product of $\bar{D}_\tau$ defines an $n^d$-dimensional determinant $\bar{D}'$ of  $\prod_{\tau\in\Sigma_F}  G_{\tau(F),\tau(S)}$ valued in  $\FF$. We denote by $(R',D')$ the universal determinant parametrizing all continuous lifts of  $\bar{D}'$.  (We note that $\prod_{\tau\in\Sigma_F}  G_{\tau(F),\tau(S)}$ also satisfies condition  $\Phi_p$.)
Since $D^{up}:R_{\bar{D}}[[\prod_{\tau\in\Sigma_F}  G_{\tau(F),\tau(S)}]]\to R_{\bar{D}}$ is such a lift, it follows from the universal property that there is a natural map
\[i:R'\to R_{\bar{D}}\]
along which $D^{up}$ is the pushforward of $D'$. 

\begin{lem}
$i:R'\to R_{\bar{D}}$ is a finite map. 
\end{lem}

\begin{proof}
Let $J$ be the maximal ideal of $R'$. Since both $R'$ and $R_{\bar{D}}$ are complete local Noetherian algebras, it's enough to show that $R_{\bar{D}}/(J)$ is Artinian. Suppose not, we can find an algebra homomorphism $f:R_{\bar{D}}\to k$ to some algebraically closed field $k$ whose kernel contains $J$ but is not the maximal ideal. Consider the induced determinant $D_k$ on $\prod  G_{\tau(F),\tau(S)}$ valued in $k$. We denote by $\rho_k:\prod  G_{\tau(F),\tau(S)}\to \GL_{n^d}(k)$ the semisimple representation with determinant $D_k$.

Since $D_k=f\circ (D' \mod J)=f\circ \bar{D}'$, representation $\rho_k$ has finite image. On the other hand $D_k=f\circ D^{up}$. Let $\rho_{\tau}: G_{\tau(F),\tau(S)}\to\GL_d(k)$ be the semisimple representation associated to $f\circ D_\tau$.  The box tensor product $\boxtimes_{\tau\in\Sigma_F} \rho_\tau$ is semisimple and isomorphic to $\rho_k$ by Theorem \ref{thm-tensorprod}. This contradicts our choice of $f$ because $\rho_\tau$ has infinite image. ($R_{\bar{D}}$ is topologically generated by the coefficients of the characteristic polynomials by \cite[Rem.3.5]{Chenevier-2014}.)
\end{proof}

Thanks to this lemma, $R_{\bar{D}}[[\prod_{\tau\in\Sigma_F}  G_{\tau(F),\tau(S)}]]=R'[[\prod_{\tau\in\Sigma_F}  G_{\tau(F),\tau(S)}]]\otimes_{R'} R_{\bar{D}}$. We finish the proof by observing that the maximal Cayley-Hamilton quotient of $R'[[\prod_{\tau\in\Sigma_F}  G_{\tau(F),\tau(S)}]]$ is a finite module of $R'$ by Theorem \ref{thmWE}. 
\end{proof}

 Fix a choice of topological generators of $R_{\bar{D}}$ and fix $r\in p^{\QQ}\cap (0,1)$. As in Definition \ref{dfnE(r)}, we can define $R^{(r)}_{\bar{D}}$ and $E^{plec,(r)}:=E^{plec}\otimes_{R_{\bar{D}}} R^{(r)}_{\bar{D}}$.
 
There is a natural group homomorphism $\rho:G_F^{plec} \to E^{plec,\times}$. Recall that $G_{\QQ}$ and $\prod_{\tau\in\Sigma_F}  \Gal(\overline{\QQ}/\tau(F))$ are natural subgroups of $G_F^{plec}$.  The fixed embedding $\iota_p:\overline{\QQ}\to\overline{\QQ_p}$ determines a $p$-adic place $\tau_p$ of $\tau(F)$ and an embedding 
\[\Gal(\overline{\QQ_p}/\tau(F)_{\tau_p})\subseteq \Gal(\overline{\QQ}/\tau(F))\subseteq \prod_{\tau'\in\Sigma_F}  \Gal(\overline{\QQ}/\tau'(F))\subseteq G_F^{plec}.\]

\begin{dfn} \label{defn-A(r)}
We define
$A^{(r)} \subseteq E^{plec,(r)}$
as the $R^{(r)}_{\bar{D}}$-subalgebra generated by $\rho(G_\QQ)$ and 
$\Lie(I_{\tau(F)_{\tau_p}},r)$ for all $\tau\in\Sigma_F$ (cf. Definition \ref{defnconstLiealg}), where $I_{\tau(F)_{\tau_p}}\subseteq \Gal(\overline{\QQ_p}/\tau(F)_{\tau_p})$ denotes the inertia subgroup. Clearly if $r'\in p^{\QQ}\cap (0,1)$ and $r<r'$, there is a natural map $A^{(r')}\to A^{(r)}$ with a dense image.
\end{dfn}

\begin{dfn}  \label{defn-partSen}
As in Subsection \ref{subsecSen}, using $\Gal(\overline{\QQ_p}/\tau(F)_{\tau_p})$ one can construct the Sen operator 
 \[\theta^{(r)}_{\tau}\in E^{plec,(r)}\widehat\otimes_{\QQ_p}\CC_p.\]
 All  these  $\theta^{(r)}_{\tau}$'s are called the ($\tau$-)\textit{partial Sen operators}. Similarly by considering $G_{\QQ_p}\subseteq G_{\QQ}\subseteq G_F^{plec}\stackrel{\rho}{\to} E^{plec,\times}$, one can construct the Sen operator
\[\theta^{(r)}\in E^{plec,(r)}\widehat\otimes_{\QQ_p}\CC_p.\]
\end{dfn}

\begin{lem} \label{lem-Sen=sumpartSen}
All $\theta^{(r)}_{\tau}$'s commutate with each other, and
$\sum_{\tau\in\Sigma_F} \theta^{(r)}_{\tau}=\theta^{(r)}$.
\end{lem}

\begin{proof}
Since $\Gal(\overline{\QQ_p}/\tau(F)_{\tau_p})$ commutes with each other in $G_F^{plec}$, by the functorial property of the Sen operator, the partial Sen operator commutes as well.
To compute the Sen operator, one can replace the local Galois group by an open subgroup.
Let $K$ be a finite extension of $\QQ_p$ in $\overline{\QQ_p}$ containing all embeddings of $F$. Then $\theta^{(r)}$ is also the Sen operator constructed from $G_{K}\subseteq G_{\QQ_p}\stackrel{\rho}{\to} E^{plec,\times}$. Similarly $\theta^{(r)}_{\tau}$ can be constructed from $G_{K}\subseteq \Gal(\overline{\QQ_p}/\tau(F)_{\tau_p})\stackrel{\rho}{\to} E^{plec,\times}$. The natural inclusion $G_{K}\to G_F^{plec}$ factors through 
\[G_K\to \prod_\tau G_K \subseteq \prod_\tau \Gal(\overline{\QQ_p}/\tau(F)_{\tau_p})\subseteq G_F^{plec}\]
where the first map is the diagonal embedding. Our claim follows from the property that the compatibility between the Sen operator and the tensor product cf. Theorem \ref{thm-Sen}.
\end{proof}

\begin{dfn}  \label{defn-partSenpoly}
For each $\tau:F\to \overline{\QQ}$,
recall that $D_\tau:R_{\bar{D}}[[G_{\tau(F),\tau(S)}]]\to R_{\bar{D}}$ denotes the $n$-dimensional determinant introduced in the beginning of this subsection. We restrict $D_\tau$ to $R_{\bar{D}}[[\Gal(\overline{\QQ_p}/\tau(F)_{\tau_p})]]$ and denote the Sen polynomial in Definition \ref{dfn-Senpoly} by $\chi_\tau^{(r)}(t)\in R_{\bar{D}}^{(r)}\widehat\otimes_{\QQ_p}\CC_p[t]$.

Let $f: R^{(r)}_{\bar{D}}\to\CC_p$ be a continuous homomorphism. The negatives of the roots of $(f\otimes 1)(\chi_\tau^{(r)})\in \CC_p[t]$ are called the \textit{$\tau$-Hodge-Tate-Sen weights} of $\rho_f$, where \( f\otimes 1 \) denotes the composition \( R^{(r)}_{\bar{D}}\widehat\otimes \CC_{p}\xrightarrow{f\otimes 1}\CC_{p}\widehat\otimes \CC_{p}\to \CC_{p} \) by abuse of notation; equivalently, they are the Hodge-Tate-Sen weights of the representation
\[\rho_{f,\tau}|_{\Gal(\overline{\QQ_p}/\tau(F)_{\tau_p})}: \Gal(\overline{\QQ_p}/\tau(F)_{\tau_p})\to \GL_n(\CC_p), \]
where $\rho_{f,\tau}$ was introduced in Remark \ref{rmk-resprodGtauF}.
\end{dfn}

\begin{lem} \label{lem-irredcsa}
Suppose that $\rho_f:G_{F,S}\to\GL_n(\CC_p)$ is irreducible. Then
\[E^{(r)}_{\bar{D}}\otimes_{R^{(r)}_{\bar{D}},f}\CC_p\cong \End_{\CC_p}(\Ind_F^{plec}\rho_f)\cong M_{n^d}(\CC_p).\]
\end{lem}

\begin{proof}
$\Ind_F^{plec}\rho_f$ is irreducible by our assumption. Our claim is \cite[Definition-Proposition 2.18]{Chenevier-2014}.
\end{proof}

\begin{prop} \label{prop-parSenev}
Let $f: R^{(r)}_{\bar{D}}\to\CC_p$ be as in the previous lemma. For $\tau\in\Sigma_F$, the image of $\theta^{(r)}_{\tau}\in E^{(r)}_{\bar{D}}\widehat\otimes_{\QQ_p}\CC_p$ in $E^{(r)}_{\bar{D}}\otimes_{R^{(r)}_{\bar{D}},f}\CC_p\cong M_{n^d}(\CC_p)$, denoted by $\theta_{\tau,f}$, satisfies
\[(f\otimes 1)(\chi_\tau^{(r)})( \theta_{\tau,f})=0.\]
In particular, its eigenvalues are the same as the negatives of the $\tau$-Hodge-Tate-Sen weights of $\rho_f$ (without counting multiplicities).
\end{prop}

\begin{proof}
 It follows from the discussion at the beginning of this subsection that 
\[\Ind_F^{plec}\rho_f|_{\prod_{\tau\in\Sigma_F}\Gal(\overline{\QQ}/\tau(F))}\cong \bigotimes_{\tau\in\Sigma_F} \rho_{f,\tau}, \]
and the operator $\theta_{\tau,f}$ acting on the left hand side is identified with the Sen operator acting on the factor $\rho_{f,\tau}|_{\Gal(\overline{\QQ_p}/\tau(F)_{\tau_p})}$ on the right hand side, which implies our claim by the Cayley-Hamilton theorem.
\end{proof}

The same argument yields the following.

\begin{prop} \label{prop-wT=sumpartialwt}
Let $f: R^{(r)}_{\bar{D}}\to\CC_p$ be a continuous homomorphism. Then $(\oInd^{\QQ}_F\rho_f)|_{G_{\QQ_p}}$  has Hodge-Tate-Sen weights  $\sum_{\tau\in\Sigma_F}a_\tau$ where $a_\tau$ runs through all $\tau$-Hodge-Tate-Sen weights of $\rho_f$ (with multiplicities). 
\end{prop}

Recall that $A^{(r)}$ was introduced in Definition \ref{defn-A(r)}.

\begin{prop} \label{prop-partialSen}
$\theta^{(r)}_{\tau}\in A^{(r)}\widehat\otimes_{\QQ_p}\CC_p$ for all $\tau$. Moreover $A^{(r)}$ is the smallest $R^{(r)}_{\bar{D}}$-subalgebra that contains $\rho(G_\QQ)$ and, after taking the completed tensor product with $\CC_p$, contains all $\theta^{(r)}_{\tau}$.
\end{prop}

\begin{proof}
By Theorem \ref{thmfamSen}, $\theta^{(r)}_{\tau}\in \Lie(I_{\tau(F)_{\tau_p}},r)\widehat\otimes_{\QQ_p}\CC_p$ and $\Lie(I_{\tau(F)_{\tau_p}},r)$ is the smallest $R^{(r)}_{\bar{D}}$-Lie subalgebra having this property. The claim follows from the construction of $A^{(r)}$.
\end{proof}

Next we discuss the difference between $A^{(r)}$ and $E^{plec,(r)}$.

\begin{dfn} \label{dfnstrirr}
For each continuous homomorphism $f:R^{(r)}_{\bar{D}}\to\CC_p$,
$\rho_f$ is called \textit{strongly irreducible} if $\rho_f|_{J}$ is irreducible for any open subgroup $J$ of $G_{F,S}$; equivalently, $\CC_p^n$ is an irreducible representation of the Lie algebra $\Lie \rho_f(G_{F,S})$. In this case, we also say that the maximal ideal $\ker f$ is strongly irreducible.
\end{dfn}

\begin{eg}
When $n=2$ and $\rho_f$ is not strongly irreducible, there are $3$ possibilities up to a twist by a  character:
\begin{enumerate}
\item $\rho_f\cong \chi_1\oplus\chi_2$ for some characters $\chi_i$ of $G_{F,S}$, i.e. $\rho_f$ is reducible.
\item $\rho_f(G_{F,S})$ is finite.
\item $\rho_f\cong \Ind_F^{F'} \chi$ for a quadratic extension $F'$ of $F$ and some character $\chi$ of $G_{F'}$.
\end{enumerate}
\end{eg}

\begin{prop} \label{prop-suppA(r)}
Suppose that $n=2$.
If $f:R^{(r)}_{\bar{D}}\to\CC_p$ is a continuous homomorphism and $\ker f$ belongs to 
the support of the $R_{\bar{D}}^{(r)}$-module $E^{plec,(r)}/A^{(r)}$, either $\rho_f$ is not strongly irreducible or $\Lie \rho_f(I_{F_v})\subseteq \mathfrak{gl}_2(\CC_p)$ are scalars for every $p$-adic place $v$ of $F$.
\end{prop}

\begin{rmk}
If $\rho_f$ has distinct $\tau$-Hodge-Tate-Sen weights for some $\tau:F\to\overline{\QQ_p}$, then $\Lie \rho_f(I_{F_v})$ cannot be scalars for every $p$-adic place $v$ of $F$.
Conjecturally, $\rho_f$ cannot be strongly irreducible if all $\Lie \rho_f(I_{F_v})$ are scalars. Indeed, in this case the adjoint representation $\Ad\rho_f$ would be potentially unramified at all $p$-adic places and one would expect it to be an Artin representation. In particular $\Lie \rho_f(G_{F,S})$ are scalars, and $\rho_f$ cannot be strongly irreducible. \end{rmk}

\begin{cor} \label{cor-G^plec_F maps to loc A^r}
Suppose that $f:R^{(r)}_{\bar{D}}\to\CC_p$ is a continuous homomorphism such that $\rho_f$ is strongly irreducible and $\Lie \rho_f(I_{F_v})\subseteq \mathfrak{gl}_2(\CC_p)$ are not scalars for  at least one $p$-adic place $v$ of $F$.
There is a natural isomorphism $E^{plec,(r)}_{\ker f}\cong A^{(r)}_{\ker f}$, and hence we get a homomorphism
\[G^{plec}_F\to (E^{plec,(r)})^\times \to (E^{plec,(r)}_{\ker f})^\times \cong (A^{(r)}_{\ker f})^\times.\]
\end{cor}

\begin{proof}[Proof of Proposition \ref{prop-suppA(r)}]
Suppose that $f$ is strongly irreducible and $\Lie \rho_f(I_{F_v})$ contains more than scalars for one $v$. 
By Lemma \ref{lem-irredcsa}, 
\[E^{plec,(r)}\otimes_{R^{(r)}_{\bar{D}},f}\CC_p=\End_{\CC_p}(\Ind_F^{plec}\rho_f) =\bigotimes_{\tau\in\Sigma_F} \End_{\CC_p}(\rho_{f,\tau}).\]
It suffices to show that the $\CC_p$-subspace $B$ spanned by the image of $A^{(r)}$  contains all $ \End_{\CC_p}(\rho_{f,\tau})$. 
Since $G_{\QQ}$ acts transitively on $\Sigma_F$, $B\cap \End_{\CC_p}(\rho_{f,\tau})$ contains $\rho_{f,\tau}(\sigma) \Lie \rho_{f,\tau}(I_{\tau(F)_v}) \rho_{f,\tau}(\sigma)^{-1}$ for all $\sigma\in \Gal(\overline{\QQ}/\tau(F))$ and $p$-adic places $v$ of $\tau(F)$. We complete the proof by the following lemma.
\end{proof}

\begin{lem}
Suppose $\rho_f:G_{F,S}\to\GL_2(\CC_p)$ is strongly irreducible and $\Lie \rho_f(I_{F_v})$ contains more than scalars for a $p$-adic place $v$. The $\CC_p$-subalgebra generated by all $\rho_{f}(\sigma) \Lie \rho_{f}(I_{F_v}) \rho_{f}(\sigma)^{-1}$, $\sigma\in G_{F,S}$ is equal to $M_2(\CC_p)$.
\end{lem}

\begin{proof}
Since $\rho_f$ is strongly irreducible, the $\CC_p$-subspace $\mathfrak{g}$ spanned by $\Lie \rho_f(G_{F,S})$ is either $\mathfrak{gl}_2(\CC_p)$ or $sl_2(\CC_p)$. The subspace 
spanned by all $\rho_{f}(\sigma) \Lie \rho_{f}(I_{F_v}) \rho_{f}(\sigma)^{-1}$ is a Lie ideal of $\mathfrak{g}$, and hence contains $sl_2(\CC_p)$ as it contains more than scalars.  Finally we note that $sl_2(\CC_p)$ generates $M_2(\CC_p)$ as an algebra.
\end{proof}

\begin{rmk}
For general $n$, a similar argument shows that conjecturally $\ker f$ belongs to 
the support of the $R_{\bar{D}}^{(r)}$-module $E^{plec,(r)}/A^{(r)}$ only if $\rho_f$ is not strongly irreducible. 
\end{rmk}

The following extension lemma is similar to Lemma \ref{lem-extSen}.
\begin{lem} \label{lem-extplecSen}
Let $W$ be a Banach $R_{\bar{D}}^{(r)}$-module equipped with a continuous \( R^{(r)}_{\bar{D}} \)-linear action $\rho_W$ of $G_\QQ$, and $W^a=\bigoplus_{i=0}^\infty W_i$ a dense subspace of $W$. Assume that each $W_i$ is finite-dimensional over $\QQ_p$ and is an $R_{\bar{D}}^{(r)}[G_\QQ]$-submodule.
Denote by $\End^a\subseteq \End^{\mathrm{cont}}_{R_{\bar{D}}^{(r)}}W$ the closed subspace of continuous $R_{\bar{D}}^{(r)}$-linear endomorphisms preserving each $W_i$. Suppose that 
\begin{itemize}
\item for each $\tau:F\to\overline\QQ$, there is an element $\theta_{W,\tau}\in \End^a\widehat\otimes_{\QQ_p}\CC_p$ and
\item there is an  $R_{\bar{D}}^{(r)}$-linear continuous action of $A^{(r)}$ on each $W_i$. Moreover the induced action of $\theta^{(r)}_\tau$ on $W_i\otimes_{\QQ_p}\CC_p$ agrees with $\theta_{W,\tau}|_{W_i\otimes_{\QQ_p}\CC_p}$ for each $\tau$, and the induced action of $G_{\QQ}$ via $A^{(r),\times}$ on $W_i$ agrees with the action via $\rho_W$.
\end{itemize}
The $A^{(r)}$-module structure on $W^a$ can be extended (necessarily uniquely) to $W$ and the induced action of $\theta^{(r)}_\tau$ on $W\widehat\otimes_{\QQ_p}\CC_p$ is equal to $\theta_{W,\tau}$ for each $\tau$ and the induced action of $G_{\QQ}$ via $A^{(r),\times}$ is equal to that via $\rho_W$.
\end{lem}

\begin{proof}
Same argument as in the proof of Lemma \ref{lem-extSen} plus Proposition \ref{prop-partialSen}.
\end{proof}

\begin{cor} \label{cor-extplecSen}
Let $W$ be a Banach $R_{\bar{D}}^{(r)}$-module equipped with a continuous \( R^{(r)}_{\bar{D}} \)-linear action $\rho_W$ of $G_\QQ$. Suppose that there exists a dense $R_{\bar{D}}^{(r)}[G_{\QQ}]$-submodule $W^a\subseteq W$  and a countable subset $I$ of the maximal ideals of $R_{\bar{D}}^{(r)}$ such that
\[W^a=\bigoplus_{\mathfrak{p}\in I} W^a[\mathfrak{p}]\]
where $W^a[\mathfrak{p}]\subseteq W^a$ denotes the subspace annihilated by $\mathfrak{p}$, satisfying that for every $\mathfrak{p}\in I$ and $\lambda:k(\mathfrak{p})(:=R_{\bar{D}}^{(r)}/\mathfrak{p})\to\overline{\QQ_p}$,
\begin{enumerate}
\item  \label{cor-extplecSen-1}The tensor induction $\oInd^{\QQ}_{F}\rho_{\lambda}:G_{\QQ}\to \GL_{n^d}(\overline{\QQ_p})$ is irreducible.
\item  \label{cor-extplecSen-2}$\oInd^{\QQ}_{F}\rho_{\lambda}|_{G_{\QQ_p}}$ has  $n^d$ distinct Hodge-Tate-Sen weights.
\item  \label{cor-extplecSen-3}$W^a[\mathfrak{p}]$ is finite-dimensional over $\QQ_p$ and $W^a[\mathfrak{p}]\otimes_{k(\mathfrak{p}),\lambda}\overline{\QQ_p}$ is $\oInd^{\QQ}_{F}\rho_{\lambda}$-isotypic. \label{assumption3}
\end{enumerate}
Moreover suppose that  for each $\tau:F\to\overline\QQ$, there is an element $\theta_{W,\tau}\in \End^{\mathrm{cont}}_{R_{\bar{D}}^{(r)}}W \widehat\otimes_{\QQ_p}\CC_p$ satisfying
\begin{enumerate}[resume]
\item  \label{cor-extplecSen-4} $\displaystyle (\sum_{\tau\in\Sigma_F} \theta_{W,\tau})|_{W^a[\mathfrak{p}]\otimes_{\QQ_p}\CC_p}$ agrees with the Sen operator on $W^a[\mathfrak{p}]\otimes_{\QQ_p}\CC_p$ for every $\mathfrak{p}$. 
\item  \label{cor-extplecSen-5} The induced action of $\theta_{W,\tau}$ on \( W^{a}[\p]\otimes_{R^{(r)}_{\bar{D}},\lambda}\CC_{p} \) (as a quotient of $W^a[\mathfrak{p}]\otimes_{\QQ_p}\CC_p$) is semisimple and its eigenvalues (without counting multiplicities) are the same as the negatives of the $\tau$-Hodge-Tate-Sen weights of $\rho_{\lambda}$. 
\item  \label{cor-extplecSen-6} $\theta_{W,\tau},\tau\in\Sigma_F$ commute with each other.
\end{enumerate}
Then there is a unique continuous $A^{(r)}$-module structure on $W$ extending the $R_{\bar{D}}^{(r)}[G_{\QQ}]$-action. Moreover the induced action of $\theta^{(r)}_\tau$ on $W\widehat\otimes_{\QQ_p}\CC_p$ is equal to $\theta_{W,\tau}$ for each $\tau$.
\end{cor}

\begin{proof}
In view of the previous lemma, it suffices to show that on each $W^a[\mathfrak{p}]$ there is a unique action of $A^{(r)}$ that extends the action of $G_{\QQ}$, and this unique $A^{(r)}$-action is compatible with  \( \theta_{W,\tau} \), $\tau\in\Sigma_F$. 
Fix a $\mathfrak{p}\in I$ in the rest of the proof.

\begin{lem}
Let $\lambda:R^{(r)}_{\bar{D}} \to\overline{\QQ_p}$  be a homomorphism such that $\oInd^{\QQ}_{F}\rho_{\lambda}$ is irreducible. Denote by $B$ the image of $R^{(r)}_{\bar{D}}[G_{\QQ}]$  in $\End_{\overline{\QQ_p}} (\oInd^{\QQ}_{F}\rho_{\lambda})$.  Let $k=k(\ker\lambda)$. There are  natural isomorphisms of $k[G_{\QQ}]$-modules
\[B\cong E^{plec,(r)}/(\ker\lambda) \cong A^{(r)}/(\ker\lambda).\] 
\end{lem}

\begin{proof}
Consider the $n^d$-dimensional determinant $D_k:k[G_\QQ]\to k$ obtained by restricting $D^{plec}\mod \ker\lambda$ to $k[G_{\QQ}]$ cf. Proposition \ref{prop-plecinddet}. Equivalently it's  the determinant of $\oInd^{\QQ}_{F}\rho_{\lambda}$. Hence by \cite[Definition-Proposition 2.18]{Chenevier-2014} and the irreducibility assumption,  $B=k[G_{\QQ}]/\mathrm{CH}(D_k)$. On the other hand, the plectic induction $\Ind_F^{plec}\rho_\lambda$ is irreducible. So $E^{plec,(r)}/(\ker\lambda)$ is a central simple  $k$-algebra of rank-$(n^d)^2$ and 
\[E^{plec,(r)}/(\ker\lambda)\otimes_{k,\lambda}\overline{\QQ_p} \cong \End_{\overline{\QQ_p}} (\Ind_F^{plec}\rho_\lambda).\] 
Since $(\Ind_F^{plec}\rho_\lambda)|_{G_\QQ}=\oInd^{\QQ}_{F}\rho_{\lambda}$ is irreducible, the image of $G_\QQ$ in $E^{plec,(r)}/(\ker\lambda)$ generates $E^{plec,(r)}/(\ker\lambda)$ as a $k$-algebra. (Indeed, $G_\QQ$ generates $\End_{\overline{\QQ_p}} (\Ind_F^{plec}\rho_\lambda)$ as a $\overline{\QQ_p}$-algebra.) Thus $E^{plec,(r)}/(\ker\lambda)$ is a quotient of $k[G_\QQ]$. The kernel clearly contains $\mathrm{CH}(D_k)$ by the construction of $E^{plec,(r)}$. We get a natural surjective homomorphism $B\to E^{plec,(r)}/(\ker\lambda)$ of simple algebras, which has to be an isomorphism. 

The natural map $A^{(r)}/(\ker\lambda) \to E^{plec,(r)}/(\ker\lambda)$ is surjective as the image contains the image of $G_\QQ$. Hence by Nakayama's lemma, the inclusion $A^{(r)}\subseteq E^{plec,(r)}$ becomes an equality after localizing at $\ker\lambda$. Reduction modulo $\ker\lambda$ gives our claim.
\end{proof}

By our assumption \eqref{assumption3}, the action of $R^{(r)}_{\bar{D}} [G_{\QQ}]$  on $W^a[\mathfrak{p}]$ factors through $B$. Hence  it extends uniquely to an action of $A^{(r)}$ by the previous lemma. To apply Lemma \ref{lem-extplecSen}, we need to check that the induced action of $\theta^{(r)}_{\tau}$ on $W^a[\mathfrak{p}]\otimes_{\QQ_p} \CC_p$ is the same as the action of $\theta_{W,\tau}$. 
Equivalently it suffices to show that for any $\lambda:R^{(r)}_{\bar{D}} \to \overline{\QQ_p}$  factoring through $k(\mathfrak{p})$,  the actions of $\theta^{(r)}_{\tau}$ and $\theta_{W,\tau}$ on $W^a[\mathfrak{p}]\otimes_{R^{(r)}_{\bar{D}} ,\lambda} \CC_p$ are the same (because $W^a[\mathfrak{p}]\otimes_{\QQ_p} \CC_p$ is a direct sum of these $W^a[\mathfrak{p}]\otimes_{R^{(r)}_{\bar{D}},\lambda} \CC_p$). We note that both sums $\sum_{\tau\in\Sigma_F} \theta^{(r)}_{\tau}$ and $\sum_{\tau\in\Sigma_F} \theta_{W,\tau}$ agree with the Sen operator when acting on $W^a[\mathfrak{p}]\otimes_{R^{(r)}_{\bar{D}} ,\lambda} \CC_p$ by assumption \eqref{cor-extplecSen-4} and Lemma \ref{lem-Sen=sumpartSen}.

Since $\oInd^{\QQ}_{F}\rho_{\lambda}|_{G_{\QQ_p}}$ has  $n^d$ distinct Hodge-Tate-Sen weights, Proposition \ref{prop-wT=sumpartialwt} implies that $\rho_f$ has $n$ distinct $\tau$-Hodge-Tate-Sen weights, which we will denote  by $-a_{\tau,1},\cdots,-a_{\tau,n}$. Fix an isomorphism
\[W^a[\mathfrak{p}]\otimes_{R^{(r)}_{\bar{D}} ,\lambda} \overline{\QQ_p}\cong (\oInd^{\QQ}_{F}\rho_{\lambda})^{\oplus n_\lambda}=(\Ind_F^{plec}\rho_\lambda)^{\oplus n_\lambda}.\]
By Proposition \ref{prop-parSenev}, the operator $\theta^{(r)}_{\tau}$ on $W^a[\mathfrak{p}]\otimes_{R^{(r)}_{\bar{D}} ,\lambda} \overline{\CC_p}$ is semisimple with eigenvalues $a_{\tau,1},\cdots,a_{\tau,n}$ (because $(\lambda\otimes 1)(\chi^{(r)}_{\tau})$ has $n$ distinct roots). 
The same is true for $(\theta_{W,\tau})|_{W^a[\mathfrak{p}]\otimes_{R^{(r)}_{\bar{D}} ,\lambda} \CC_p}$ by assumption \eqref{cor-extplecSen-5}. We deduce our claim from the following simple linear algebra lemma.
\end{proof}

\begin{lem}
Suppose we have $n$ numbers $a_{\tau,1},\cdots,a_{\tau,n}\in\CC_p$ for each $\tau\in\Sigma_F$.
Let $V$ be a finite-dimensional vector space over $\CC_p$. Let $\theta\in\End_{\CC_p}(V)$ be a  semisimple operator with eigenvalues $\sum_{\tau\in\Sigma_F} a_{\tau,i_\tau}$ without counting multiplicities for $\{i_\tau\}_{\tau\in\Sigma_F}\in \{1,\cdots,n\}^{\Sigma_F}$. If all such $\sum_{\tau\in\Sigma_F} a_{\tau,i_\tau}$ are distinct, there is a unique way to write $\theta=\sum_{\tau}\theta_\tau$ as a sum of commuting operators $\theta_\tau\in\End_{\CC_p}(V)$ such that each $\theta_\tau$ is semisimple with eigenvalues $a_{\tau,1},\cdots,a_{\tau,n}$.
\end{lem}

\begin{proof}
Existence is clear by defining $\theta_\tau$ as the unique operator that acts on the ($\sum_{\tau\in\Sigma_F} a_{\tau,i_\tau}$)-eigenspace of $\theta$ by $a_{\tau,i_\tau}$. To see the uniqueness, since $\theta_\tau$ commutes with each other, they commute with $\theta=\sum_{\tau}\theta_\tau$ as well. Hence the ($\sum_{\tau\in\Sigma_F} a_{\tau,i_\tau}$)-eigenspace of $\theta$ is invariant under the action of each $\theta_\tau$. Since all possible sums $\sum_{\tau\in\Sigma_F} a_{\tau,i_\tau}$ are distinct, the eigenvalue of $\theta_\tau$ showing up in this ($\sum_{\tau\in\Sigma_F} a_{\tau,i_\tau}$)-eigenspace has to be $a_{\tau,i_{\tau}}$. This uniquely determines $\theta_\tau$ as it is semisimple.
\end{proof}

\section{Completed cohomology of Hilbert modular varieties} \label{sec-CCHMV}
In this section, we construct an action of $A^{(r)}$ on the non-CM part of the locally analytic completed cohomology of Hilbert modular varieties in middle degree and discuss several applications. 

\subsection{Main result}
\label{subsec-Main result}
First we introduce some notation.
Let \(F\) be a totally real number field of degree \(d\) over $\QQ$.
Let $\mathrm{G}=\mathrm{Res}_{F/\QQ}\GL_2$ the restriction of scalars of $\GL_2$ from $F$ to $\QQ$. We have $\mathrm{G}(\RR)=\prod_{\tau:F\to \RR} \mathrm{GL}_2(\RR)$, where $\tau$ runs through all embeddings of $F$ into $\RR$, and it acts naturally on $\mathrm{X}:=\prod_{\tau:F\to \RR} (\CC\setminus\RR)$ by the usual M\"obius transformation. 
For a neat compact open subgroup $K$ of $\mathrm{G}(\mA_f)$ with $\mA_f$ denoting the ring of finite ad\`eles, we have the usual (ad\`elic) Hilbert modular variety $\mathrm{Sh}_{K}$ over $\QQ$ with $\CC$-points
\[\mathrm{Sh}_{K}(\CC):=\mathrm{G}(\QQ)\setminus (\mathrm{X}\times \mathrm{G}(\mA_f))/K.\]
When $K$ varies,  $\{\Sh_{K}\}_{K\subseteq \mathrm{G}(\mA_f)}$ forms a tower with a natural action of $\mathrm{G}(\mA_f)$.  Fix an open compact subgroup $K^p$  of $\mathrm{G}(\mA_f^p)$, where $\mA_f^p$ denotes the finite ad\`eles away from $p$. 
The completed cohomology with tame level $K^p$ introduced by Emerton \cite{Emerton-2006} is defined as
\[
\tilde{H}^i(K^p):=\varprojlim_n\varinjlim_{K_p\subseteq \mathrm{G}(\QQ_p)} H^i(\Sh_{K^pK_p}(\CC),\ZZ/p^n)
\]
where $K_p$ runs through all open compact subgroups of $\mathrm{G}(\QQ_p)$. The rational completed cohomology $\tilde{H}^i(K^p)_{\QQ_p}:=\tilde{H}^i(K^p)\otimes_{\ZZ_p} \QQ_p$ is a unitary admissible $p$-adic Banach space representation of $\mathrm{G}(\QQ_p)$ with the unit ball given by the torsion-free quotient of $\tilde{H}^i(K^p)$. By identifying $H^i(\Sh_{K^pK_p}(\CC),\ZZ/p^n)$ with the \'etale cohomology of $\Sh_{K^pK_p}\times_{\Spec\QQ}\Spec\overline{\QQ}$ using comparison theorems, we get a natural action of the absolute Galois group $G_{\QQ}$ on the completed cohomology $\tilde{H}^i(K^p)$.

Our convention for the hodge cocharacter is the \textit{cohomological one}. For example when $F=\QQ$,  $\mathrm{Sh}_{K}$ parametrizes elliptic curves with a level structure $K$ on its $H_\et^1$. The action of $G_{\QQ}$ and $\mathrm{G}(\mA_f)$ on the connected components of the tower $\{\mathrm{Sh}_{K}\}_{K\subseteq\mathrm{G}(\mA_f)}$  is compatible with the Artin reciprocity map sending \textit{geometric Frobenii} to uniformizers. More precisely, the action of $G_\QQ$ on $\tilde{H}^0(K^p)$ is abelian, and hence by pre-composing with the Artin map $\QQ^\times_{>0}\setminus \mA_f^\times \to G_{\QQ}^{\mathrm{ab}}$, we get an action of $\QQ^\times_{>0}\setminus \mA_f^\times$ on $\tilde{H}^0(K^p)$. On the other hand, as explained in \cite[\S 2]{Deligne-1971}, $\mathrm{G}(\mA_f)$ acts on $\tilde{H}^0(K^p)$ factoring through the determinant map, inducing an action of $F^\times_{>>0}\setminus(\mA_f\otimes_\QQ F)^\times$, where $F^\times_{>>0}$ denotes the totally positive numbers of $F$.  By pre-composing with the natural map $\QQ^\times_{>0}\setminus \mA_f^\times \to F^\times_{>>0}\setminus(\mA_f\otimes_\QQ F)^\times$, we recover the action of  $\QQ^\times_{>0}\setminus \mA_f^\times$ defined previously.

Let $S$ be a finite set of places of $F$ containing all $p$-adic and Archimedean places such that $K\cap \mathrm{GL}_2(F_v)$ is a maximal compact subgroup for $v\notin S$. 
For each $K_p$, let $\TT_{K^pK_p}$ be the $\ZZ_p$-subalgebra of $\End(R\Gamma(\Sh_{K^pK_p}(\CC),\ZZ_p))$ generated by the usual spherical Hecke operators $S_v,T_v$ at places not in $S$. The inverse limit
\[\TT:=\varprojlim_{K_p} \TT_{K^pK_p}\]
acts naturally on the completed cohomology and commutes with the actions of $\mathrm{G}(\QQ_p)$ and $G_{\QQ}$.

\begin{thm} \label{thm-big T}
 $\TT$ is a Noetherian complete semi-local $\ZZ_p$-algebra and  there is  a unique continuous $2$-dimensional determinant $D:\ZZ_p[[G_{F,S}]]\to \TT$ such that for each $v\notin S$, the geometric Frobenius $\Fr_v$ at $v$ has characteristic polynomial $D(t-\Fr_v)=t^2-T_v t + \Nm(v) S_v$.
\end{thm}

\begin{proof}
This follows from the method in \cite{Scholze-2015}. See Appendix of \cite{Jiang2025HMF}. 
The argument in Section \ref{sec-proof of MT} below also gives a proof if $\TT$ is replaced by its image in $\End(\tilde{H}^d(K^p)_{\QQ_p})$ (which is enough for our purpose).
\end{proof}

Recall that for a homomorphism $\lambda:\TT\to \overline{\QQ_p}$, we denote by $\rho_\lambda:G_{F,S}\to\GL_2(\overline{\QQ_p})$ the semisimple representation with determinant $\lambda\circ D$, and denote by $\oInd^{\QQ}_{F}\rho_\lambda:G_{\QQ}\to \GL_{2^d}(\overline{\QQ_p})$ its tensor induction. 

\begin{thm} \label{thm-mainthm}
Let $\lambda:\TT\to \overline{\QQ_p}$ be a homomorphism such that the $\lambda$-isotypic part $\tilde{H}^d(K^p)_{\overline{\QQ_p}}[\lambda]$ of $\tilde{H}^i(K^p)\otimes_{\ZZ_p} \overline{\QQ_p}$ is non-zero. Suppose  that $\rho_\lambda$ is strongly irreducible and $\Lie \rho_\lambda(I_{F_v})\subseteq \mathfrak{gl}_2(\overline{\QQ_p})$ are not scalars for a $p$-adic place $v$ of $F$ (cf. Proposition \ref{prop-suppA(r)}).
\begin{enumerate}
\item There is a natural isomorphism of $G_{\QQ}\times \mathrm{G}(\QQ_p)$-representations
\[\tilde{H}^d(K^p)_{\overline{\QQ_p}}^{\la}[\lambda]=\oInd^{\QQ}_F \rho_\lambda\otimes_{\overline{\QQ_p}} \Pi^{\la}_{\lambda}\]
for some $\mathrm{G}(\QQ_p)$-representation $\Pi^\la_\lambda$,
where $\tilde{H}^d(K^p)_{\overline{\QQ_p}}^{\la}=\tilde{H}^i(K^p)^{\la}\otimes_{\QQ_p}\overline{\QQ_p}$, and $\tilde{H}^i(K^p)^{\la}\subseteq \tilde{H}^i(K^p)_{\QQ_p}$ denotes the subspace of $\mathrm{G}(\QQ_p)$-locally analytic vectors.
\item $\tilde{H}^{<d}(K^p)_{\overline{\QQ_p}}[\lambda]=0$.
\end{enumerate}
\end{thm}

\begin{rmk}
Since $\oInd^{\QQ}_{F}\rho_\lambda$ is a semisimple representation, the image of $\QQ_p[G_{\QQ}]$ in $\End\left(\tilde{H}^d(K^p)^\la_{\overline{\QQ_p}}[\lambda]\right)$ is a semisimple algebra. By  Schneider-Teitelbaum \cite{ST-2003}, $\tilde{H}^d(K^p)_{\overline{\QQ_p}}^{\la}[\lambda]\cap \tilde{H}^d(K^p)_{\QQ_p}$ is dense inside of $\tilde{H}^d(K^p)_{\overline{\QQ_p}}[\lambda]\cap \tilde{H}^d(K^p)_{\QQ_p}$. 
We see that the action of $G_\QQ$ on $\tilde{H}^d(K^p)_{\overline{\QQ_p}}[\lambda]$ is semisimple as well. This can also be deduced from Nekov\'a\v{r}'s work \cite{Nekovar-2018}. Combined with the Eichler-Shimura's relation, it shows that every irreducible subrepresentation of $\tilde{H}^d(K^p)_{\overline{\QQ_p}}[\lambda]$ must be a simple factor of $\oInd^{\QQ}_{F}\rho_\lambda$. However it's not clear to us whether his method can show that all irreducible factors of $\oInd^{\QQ}_{F}\rho_\lambda$ would appear in $\tilde{H}^d(K^p)_{\overline{\QQ_p}}[\lambda]$.
\end{rmk}

The rest of this section is devoted to the proof of Theorem \ref{thm-mainthm}. Since $\TT$ is semi-local,
\[\tilde{H}^i(K^p)=\bigoplus_{\mathfrak{m}\in \mathrm{Spm} \TT}\tilde{H}^i(K^p)_{\mathfrak{m}} \]
where $\mathrm{Spm} \TT$ denotes the set of maximal ideals of $\TT$.
Let $\mathfrak{m}\in \mathrm{Spm} \TT$. The residue field $\FF:=\TT/\mathfrak{m}$ is a finite extension of $\FF_p$. Let $\mO_L$ be the ring of Witt vectors $W(\FF)$ and $L=\mO_L[1/p]$. By Hensel's lemma, there is a natural map $\mO_L\to \TT_{\mathfrak{m}}$ lifting the identity map $\FF=\FF$. Hence $\TT_\mathfrak{m}$ is a natural $\mO_L$-algebra.

Let $\bar{D}$ be the $2$-dimensional determinant $D \mod \mathfrak{m}$ of $G_{F,S}$ valued in $\FF$. Consider the universal deformation ring $R_{\bar{D}}$  parametrizing all continuous lifts of $\bar{D}$ with universal determinant $D^u$, cf. Subsection \ref{subsec-PSoaLa}. The localization $D_\mathfrak{m}$ of $D$ at $\mathfrak{m}$ induces a surjective map
\[\phi:R_{\bar{D}}\to \TT_{\mathfrak{m}}\]
such that $D_\mathfrak{m}=\phi\circ D^u$. We will view $\tilde{H}^i(K^p)_{\mathfrak{m}}$ as an $R_{\bar{D}}$-module via $\phi$ from now on. 

\begin{lem}
The subset of $x\in \Spec R_{\bar{D}}$ such that 
\begin{itemize}
\item $\rho_x|_{G_E}:G_E\to\GL_2( \overline{k(x)})$ is reducible for a quadratic extension $E$ of $F$
\end{itemize}
is Zariski closed in $\Spec R_{\bar{D}}$. Here $ \overline{k(x)}$ is an algebraic closure of $k(x)$, and $\rho_x$ denotes the semisimple representation associated to $D\mod x$.
Similarly the subset of $x\in \Spec R_{\bar{D}}$ such that 
\begin{itemize}
\item $\rho_x|_{G_E}:G_E\to\GL_2( \overline{k(x)})$ is reducible for a CM quadratic extension $E$ of $F$
\end{itemize}
is Zariski closed in $\Spec  R_{\bar{D}}$. 
\end{lem}

\begin{proof}
For each quadratic extension $E$, the subset of $x$ such that $\rho_x|_{G_E}$ is reducible is Zariski closed, cf.\cite[Example 2.20]{Chenevier-2014}. Since $H^1(G_{F,S},\FF_2)=\Hom(G_{F,S},\FF_2)$ is finite, we only need to consider finitely many quadratic extensions. Thus both loci considered in the lemma are Zariski closed because they are finite unions of Zariski closed subsets.
\end{proof}

\begin{dfn}[Non-strongly irreducible locus and CM locus] \label{dfn-idealInsi}
For each  $\mathfrak{m}\in \mathrm{Spm} \TT$,
Let $I_{nsi,\mathfrak{m}}$ (resp. $I_{CM,\mathfrak{m}}$) be the ideal of $R_{\bar{D}}$ defined by the Zariski closed subset of $x\in \Spec R_{\bar{D}}$ such that 
 $\rho_x|_{G_E}:G_E\to\GL_2( \overline{k(x)})$ is reducible for a quadratic (resp. CM quadratic) extension $E$ of $F$ (equivalently, $I_{nsi,\mathfrak{m}}$ (resp. $I_{CM,\mathfrak{m}}$) is the intersection of such $x$'s.). We define $I_{nsi,\TT}\subseteq\TT$ as the image of
 \[\prod_{\mathfrak{m}\in\mathrm{Spm}\TT} I_{nsi,\mathfrak{m}} \to \TT,\]
i.e. it is the ideal defining the non-strongly irreducible locus of $\Spec \TT$.
\end{dfn}

Fix $\mathfrak{m}\in \mathrm{Spm} \TT$ and hence  $\bar{D}$ in the rest of this section. We will drop $\mathfrak{m}$ in $I_{nsi,\mathfrak{m}}$ and  $I_{CM,\mathfrak{m}}$.

Recall that in Definition \ref{defn-A(r)}, for $r\in p^{\QQ}\cap (0,1)$ we introduce an $L$-algebra $A^{(r)}$ which receives a natural map from $R_{\bar{D}}[G_{\QQ}]$. Our main technical results of this section are the following.

\begin{thm} \label{thm-plecticactnonCM}
For every $K_p\subseteq \mathrm{G}(\QQ_p)$ uniform pro-$p$ open subgroup (for example a sufficiently small principal congruence subgroup), let  $\tilde{H}^d(K^p)^{K_p-\an}\subseteq \tilde{H}^d(K^p)_{\QQ_p}$ denote the subspace of $K_p$-analytic vectors. There exists an $r=r_{K_p}\in p^{\QQ}\cap (0,1)$ and  a  continuous  $A^{(r)}$-module structure on $(I_{CM}\tilde{H}^d(K^p)_{\QQ_p})^{K_p-\an}$  (note that $I_{CM}$ acts on $\tilde{H}^d(K^p)_{\QQ_p}$ via $I_{CM}\to \TT_{\mathfrak{m}} \to \TT$ and  $I_{CM}\tilde{H}^d(K^p)_{\QQ_p}$ is a closed subspace of $\tilde{H}^d(K^p)_{\QQ_p}$ by Lemma \ref{lem-ICMclosed} below) that extends the $R_{\bar{D}}[G_\QQ]$-module structure and is compatible with the $\mathrm{G}(\QQ_p)$-action in the following sense: 
\begin{itemize}
\item Let $g\in \mathrm{G}(\QQ_p)$ and $K'_p\subseteq gK_pg^{-1}$ a uniform pro-$p$ open subgroup. Set $r'=\max (r_{K_p},r_{K'_p})$. The composite map 
\[(I_{CM}\tilde{H}^d(K^p)_{\QQ_p})^{K_p-\an}\xrightarrow{\times g} (I_{CM}\tilde{H}^d(K^p)_{\QQ_p})^{gK_pg^{-1}-\an}\subseteq (I_{CM}\tilde{H}^d(K^p)_{\QQ_p})^{K'_p-\an}\] 
is $A^{(r')}$-linear. (Recall that there is a natural map $A^{(r')}\to A^{(r)}$ if $r<r'$.)
\end{itemize}
Moreover the $A^{(r_{K_p})}$-module structure on each $(I_{CM}\tilde{H}^d(K^p)_{\QQ_p})^{K_p-\an}$ is uniquely determined by these conditions.
\end{thm}

\begin{thm}\label{thm-belmidnsi}
There exists $n>0$ such that $I_{nsi,\TT}^n$ annihilates $\tilde{H}^{<d}(K^p)_{\QQ_p}$.
\end{thm}
Clearly Theorem \ref{thm-belmidnsi} implies the second part of  Theorem \ref{thm-mainthm}.

\begin{lem}\label{lem-ICMclosed}
$I_{CM}\tilde{H}^d(K^p)_{\QQ_p}$ is a closed $\mathrm{G}(\QQ_p)$-invariant subspace of $\tilde{H}^d(K^p)_{\QQ_p}$.
\end{lem}

\begin{proof}
Since $R_{\bar{D}}$ is Noetherian, $I_{CM}$ is finitely generated, say by $f_1,\cdots, f_t$. Then $I_{CM}\tilde{H}^d(K^p)_{\QQ_p}$ is the image of the map of admissible Banach space representations of $\mathrm{G}(\QQ_p)$
\[(\tilde{H}^d(K^p)_{\QQ_p})^{\oplus t} \to \tilde{H}^d(K^p)_{\QQ_p}\]
sending $(x_1,\cdots,x_t)$ to $f_1x_1+\cdots+ f_tx_t$, which is closed by the main result of \cite{ST-2002}.
\end{proof}

Theorem \ref{thm-plecticactnonCM} gives the first part of  Theorem \ref{thm-mainthm}. 
Indeed, let $\lambda$ be as in Theorem \ref{thm-mainthm}. Since $\rho_\lambda$ is strongly irreducible, $\ker\lambda$ does not contain $I_{CM}$. Hence
\[(I_{CM}\tilde{H}^d(K^p)_{\QQ_p})^{K_p-\an}_{\overline{\QQ_p}}[\lambda]=\tilde{H}^d(K^p)^{K_p-\an}_{\overline{\QQ_p}}[\lambda].\]
By our assumption on $\rho_\lambda$ and Proposition \ref{prop-suppA(r)} and its proof,
\[A^{(r)}\otimes_{R_{\bar{D}},\lambda}\overline{\QQ_p}=\End_{\overline{\QQ_p}}(\Ind_F^{plec}\rho_\lambda) \]
is a simple algebra. Therefore Theorem \ref{thm-plecticactnonCM} implies that as an $A^{(r)}$-module,
\[\tilde{H}^d(K^p)^{K_p-\an}_{\overline{\QQ_p}}[\lambda]=\Ind_F^{plec}\rho_\lambda\otimes_{\overline{\QQ_p}} \Pi^{K_p-an}_{\lambda}\]
with $\Pi^{K_p-an}_{\lambda}=\Hom_{A^{(r)}\otimes_{R_{\bar{D}},\lambda} \overline{\QQ_p}}\left(\Ind_F^{plec}\rho_\lambda, \tilde{H}^d(K^p)^{K_p-\an}_{\overline{\QQ_p}}[\lambda]\right)$. By the uniqueness of the $A^{(r)}$-action, this decomposition is compatible when $K_p$ varies. We get the first part of  Theorem \ref{thm-mainthm} by taking the direct limit over $K_p$ and noting that $(\Ind_F^{plec}\rho_\lambda)|_{G_\QQ}=\oInd^{\QQ}_F \rho_\lambda$.

\begin{rmk}
The uniqueness in Theorem \ref{thm-plecticactnonCM} implies that the action of $A^{(r)}$ is compatible when shrinking $K_p$ and enlarging $r$, i.e. suppose that $ K_p$ is sufficiently small and $K'_p\subseteq K_p$ is an open uniform subgroup, and the action of $R_{\bar{D}}[G_\QQ]$ on $(I_{CM}\tilde{H}^d(K^p)_{\QQ_p})^{K'_p-\an}$ extends to an action of $A^{(r')}$ for some $r'\in p^{\QQ}\cap (r,1)$, then the natural inclusion $(I_{CM}\tilde{H}^d(K^p)_{\QQ_p})^{K'_p-\an}\subseteq (I_{CM}\tilde{H}^d(K^p)_{\QQ_p})^{K_p-\an}$ is $A^{(r')}$-equivariant with respect to the natural map $A^{(r')}\to A^{(r)}$. In particular we get a natural action of 
\[\varprojlim_{r} A^{(r)}\]
on $(I_{CM}\tilde{H}^d(K^p)_{\QQ_p})^{\la}$.
\end{rmk}

\begin{rmk}
Theorem \ref{thm-introd-plec-2} in the introduction follows from Theorem \ref{thm-plecticactnonCM} and Corollary \ref{cor-G^plec_F maps to loc A^r}.
\end{rmk}

We will prove Theorem \ref{thm-plecticactnonCM} by appealing to Corollary \ref{cor-extplecSen}. This requires a nice dense subspace ``$W^a$'' of $I_{CM}\tilde{H}^d(K^p)^{K_p-\an}$ and a construction of partial Sen operators. Both ingredients will be provided in the next two subsections. After that we will prove Theorem \ref{thm-belmidnsi}.

\subsection{A density result} \label{subsec-density}
In practice, it is more convenient to work with another subspace of analytic vectors. The reference here is \cite[Chap. IV]{CD-2014}.
 Let $H$ be a uniform pro-$p$ group, i.e. a torsion-free topologically finitely generated pro-$p$ group satisfying $[H,H]\subseteq H^{p}$ 
 if $p>2$ and $[H,H]\subseteq H^{4}$ if $p=2$. Fix a system of minimal topological generators (also called an ordered basis) $h_1,h_2,\cdots, h_n$ of $H$.
  It induces a homeomorphism $c:\ZZ_p^n\to H$ by sending $x=(x_1,\cdots,x_n)$ to $h^x=h_1^{x_1}\cdots h_n^{x_n}$. For $\alpha=(\alpha_1,\cdots,\alpha_n)\in\NN^n$, 
  define $\phi_{\alpha}:H\to \ZZ_p$ by 
  $\phi_\alpha(h^x)={\binom{x_1}{\alpha_1}}\cdots{\binom{x_n}{\alpha_n}}$. 
A classical theorem of Mahler says that a continuous function $\phi:H\to \QQ_p$ has a unique expansion
 \[\phi=\sum_{\alpha\in\NN^n} a_{\alpha}(\phi)\cdot \phi_\alpha\]
where $a_{\alpha}(\phi)\in\QQ_p$ and $\{a_{\alpha}(\phi)\}$ converges to $0$ when $|\alpha|=\alpha_1+\cdots+\alpha_n\to \infty$.

\begin{dfn}
For $h\geq 1$, let $r_h=\frac{1}{p^{h-1}(p-1)}$ and 
\[\mathrm{LA}^{(h)}(H):=\{\phi:H\to\QQ_p\mbox{ continuous},\, \lim_{|\alpha|\to\infty} (v_p(a_{\alpha}(\phi))-r_h|\alpha|)=\infty\}.\]
It is a $p$-adic Banach space with the unit ball \( \mrm{LA}^{(h)}(H)^{\circ} \) consisting of $\phi\in \mathrm{LA}^{(h)}(H)$ with $v_p(a_{\alpha}(\phi))-r_h|\alpha|\geq 0$ for all $\alpha$.
\end{dfn}
The union of all $\mathrm{LA}^{(h)}(H)$ is the space of locally analytic functions on $H$ by a classical result of Amice. In fact for each $m\geq 0$, the subspace of $H^{p^m}$-analytic functions on $H$ is contained in $\mathrm{LA}^{(m+1)}(H)$, (in particular $H$-analytic functions are contained in $\mathrm{LA}^{(1)}(H)$) and conversely, for each $h\geq 1$, the functions in $\mathrm{LA}^{(h)}(H)$ are $H^{p^m}$-analytic for some $m$. Here $H^{p^m}=\{h^{p^m},h\in H\}$ is an open normal subgroup of $H$. The space of analytic functions $C^{\an}(H,\QQ_p)$ are dense in $\mathrm{LA}^{(h)}(H)$, $h\geq 1$. Indeed, the coordinate functions $x_1,\cdots,x_n$ are clearly analytic and dense in $\mathrm{LA}^{(h)}(H)$. The inclusion map $C^{\an}(H,\QQ_p)\subseteq \mathrm{LA}^{(h)}(H)$ has norm $\leq 1$ by Amice's result.

The topological dual of $\mathrm{LA}^{(h)}(H)$ is the distribution algebra $D_{<p^{-r_h}}(H,\QQ_p)$ introduced by Schneider-Teitelbaum in \cite[\S 4]{ST-2003}. In particular $\mathrm{LA}^{(h)}(H)$ is independent of the choice of $h_1,h_2,\cdots, h_n$ and is invariant under the left and right translations of $H$.

For a $\QQ_p$-Banach space representation of $H$, we put
\[W^{(h)}:=(\mathrm{LA}^{(h)}(H)\widehat\otimes_{\QQ_p} W)^{H}\]
where $H$ acts diagonally on $H$ and on $\mathrm{LA}^{(h)}(H)$ via the left translation $(g\cdot \phi)(x)=\phi(g^{-1}x)$, $g,x\in H$. It inherits a Banach space structure from the completed tensor product.
The evaluation map at the identity element of $H$ induces a continuous injective map $W^{(h)}\to W$. Hence we will view $W^{(h)}$ as a subspace of $W$ from now on. Equivalently, $w\in W^{(h)}$ if and only if the orbit map $o_w:H\to W$ sending $g$ to $g\cdot w$ is an element in $\mathrm{LA}^{(h)}(H)\widehat\otimes_{\QQ_p} W$. We say vectors of $W^{(h)}$ are $(h)$-analytic.
Dually, by \cite[Prop. IV.13]{CD-2014}, there is a natural isomorphism
\begin{eqnarray} \label{eqn-h-analytic}
(W^{(h)})^*\cong D_{<p^{-r_h}}(H,\QQ_p)\otimes_{\ZZ_p[[H]][1/p]} W^*,
\end{eqnarray}
where $(W^{(h)})^*$ and $W^*$ denote the topological dual of $W^{(h)}$ and $W$ respectively. We note that $D_{<p^{-r_h}}(H,\QQ_p)$ is a flat algebra of $\ZZ_p[[H]][1/p]$ by \cite[Chap. 4]{ST-2003}, \cite[Prop.IV.2]{CD-2014}.

\begin{prop} \label{prop-(h)-an}
Let $W$ be a $\QQ_p$-Banach space representation of $H$.
\begin{enumerate}
\item \label{prop-(h)-an-1} For each $m\geq 1$, $W^{H^{p^m}-\an}\subseteq W^{(h)}$ for some $h$.
\item \label{prop-(h)-an-2} For each $h\geq 1$, $W^{(h)} \subseteq W^{H^{p^m}-\an}$ for some $m$.
\item \label{prop-(h)-an-3} If $W$ is an admissible representation and $W'$ is a quotient Banach space representation of $W$, the natural map $W^{(h)}\to W'^{(h)}$ is surjective.
\end{enumerate}
\end{prop}
\begin{proof}
The first two claims follow from Amice's result. The last part is a result of Schneider-Teitelbaum. See Corollaire IV.14. of \cite{CD-2014}.
\end{proof}

Now we assume that $H$ is a uniform pro-$p$ open subgroup of $\mathrm{G}(\QQ_p)=\GL_2(F\otimes \QQ_p)$ of the form $H_1\times H_0$, where 
\begin{itemize}
\item $H_1=1+p^2M_2(\mathcal{O}_F\otimes\ZZ_p)\cap\SL_2(F\otimes \QQ_p)$,
\item $H_0\subseteq 1+p^2(\mathcal{O}_F\otimes\ZZ_p)\subseteq (F\otimes \QQ_p)^\times$ a pro-$p$ open subgroup in the center of $\mathrm{G}(\QQ_p)$. 
\end{itemize}
It follows that $K^pH$ is neat. The image of $H$ in $\mathrm{PGL}_2(F\otimes \QQ_p)$ is naturally isomorphic to $H_1$. Hence we will also regard $H_1$ as a uniform pro-$p$ open subgroup of $\mathrm{PGL}_2(F\otimes \QQ_p)$.

Let $E$ be a finite Galois extension of $\QQ_p$ inside of $\overline{\QQ_p}$ containing all embeddings of $F$. Let $\Sigma'$ be the set of embeddings $F\to E$. 

\begin{dfn} \label{dfn-rhokVk}
For $k\geq 0$ even, we denote by $(\rho_k, V_k)$ the irreducible algebraic representation of $\mathrm{PGL}_2(E)$ of dimension $k+1$. For $\underline{k}=\{k_\tau\}_{\tau\in \Sigma'}$ with $k_\tau\geq 0$ even, we denote by $( \rho_{\underline{k}}, V_{\underline{k}})$ the representation of $\mathrm{PGL}_2(F\otimes \QQ_p)$ with $\rho_{\underline{k}}$ the tensor product of $\mathrm{PGL}_2(F\otimes \QQ_p)\xrightarrow{\tau}  \mathrm{PGL}_2(E)\xrightarrow{\rho_{k_\tau}} \GL(V_{k_\tau})$ over all $\tau\in\Sigma'$ and $V_{\underline{k}}=\otimes_{\tau} V_{k_\tau}$. 
\end{dfn}
All irreducible algebraic representations of $\mathrm{Res}_{F/\QQ}\PGL_2(E)$ arise in this way. By abuse of notation   $(\rho_k, V_k)$ is also viewed as a representation of $H$ via $H\subseteq\mathrm{Res}_{F/\QQ}\GL_2(\QQ_p)\to  \mathrm{Res}_{F/\QQ}\PGL_2(E)$.

\begin{dfn} \label{dfn-veryregwt}
We say $( \rho_{\underline{k}}, V_{\underline{k}})$ is of very regular weight if all $k_\tau$'s are positive and all $\sum_{\tau\in I} k_{\tau}+ |I|$ are distinct when $I$ runs through all subsets of  $\Sigma'$. The sum is understood as $0$ if $I$ is empty.
\end{dfn}

\begin{rmk}
On the Galois side, this condition is equivalent to requiring that the tensor induction of the Galois representation associated to a Hilbert eigenform of weight $\{k_\tau+2\}_{\tau\in \Sigma'}$ is regular, i.e. has $2^d$ distinct Hodge-Tate weights. 
\end{rmk}

Fix an ordering of $\Sigma' = \{\tau_1,\cdots,\tau_d\}$. We will also write  $\underline{k}$ as $(k_1,\cdots,k_d)$ with $k_i=k_{\tau_i}$ below.

The action of $H$ on the tower of Hilbert modular varieties $\{\mathrm{Sh}_{K^pK_p}\}_{K_p}$ is not faithful in general. We denote by $H'_0$ the image of $H_0$ in the automorphism group of the tower. By shrinking $H_0$ if necessary, we may and will assume that $H'_0$ is torsionfree.

\begin{dfn}
Fix an isomorphism $H'_0\cong \ZZ_p^{e}$ sending $h\in H'_0$ to $(\nu_1(h),\cdots,\nu_e(h))$. Given $\underline{l}=(l_1,\cdots,l_e)\in\ZZ^e$, we let $\psi_{\underline{l}}:H'_0\to \ZZ_p^\times$ be the character 
\[h\mapsto \exp\left(p\sum_{i=1}^e l_i\nu_i(h)\right)\]
if $p\neq 2$, and $h\mapsto \exp\left(4\sum_{i=1}^e l_i\nu_i(h)\right)$ if $p=2$. By abuse of notation we will also view $\psi_{\underline{l}}$ as a character of $H$ via the quotient map $H\to H_0\to H'_0$. 
\end{dfn}

\begin{rmk}
One main reason for introducing this very unnatural definition is that we don't know the value of $e$ in general. It was predicted by Leopoldt's conjecture that $e=1$.
\end{rmk}

\begin{dfn} \label{dfn-veryregularsubspace}
Let $W$ be a Banach space representation of $H$ over $E$. We define  $W^a\subseteq W$ as the image of the natural evaluation map
\[\bigoplus_{(\underline{k},\underline{l})} \Hom_{E[H]}\left(\rho_{\underline{k}}\otimes \psi_{\underline{l}}, W\right)\otimes_E (\rho_{\underline{k}}\otimes \psi_{\underline{l}}) \to W\]
where $\underline{k}\in(2\NN)^{d}$ ranges through all very regular weights and $\underline{l}$ ranges through all $\ZZ^e$, and $\Hom_{E[H]}\left(\rho_{\underline{k}}\otimes \psi_{\underline{l}}, W\right)\otimes_E (\rho_{\underline{k}}\otimes \psi_{\underline{l}}) \to W$ maps $\varphi\otimes v$ to $\varphi(v)$.
\end{dfn}

Back to the completed cohomology. Let $\tilde{H}^d(K^p)_{E}=\tilde{H}^d(K^p)\otimes_{\ZZ_p} E$ a Banach space representation of $H$. Fix $h\geq 1$. Consider the subspace of $(h)$-analytic vectors $\tilde{H}^d(K^p)^{(h)}_{E}\subseteq \tilde{H}^d(K^p)_{E}$. Here is the main density result of this subsection.

\begin{thm} \label{thm-(h)dense}
$\tilde{H}^d(K^p)^{(h),a}_{E}$ is a dense subspace of $\tilde{H}^d(K^p)^{(h)}_{E}$.
\end{thm}

\begin{proof}
Let $H'$ be the image of $H$ in the automorphism group of the tower $\{\mathrm{Sh}_{K^pK_p}\}_{K_p}$. It is well-known that the kernel is in the center so $H'=H_1\times H'_0$.  The following result is Corollary 3.2.6 of \cite{Pan-2025cchmd}

\begin{lem}
As a representation of $H'$, $\tilde{H}^d(K^p)_{\QQ_p}$ is isomorphic to a quotient of $C(H',\QQ_p)^{\oplus m}$ for some $m$, where $C(H',\QQ_p)$ is the space of continuous $\QQ_p$-valued functions on $H'$ with the left translation action of $H'$.
\end{lem}

Fix such a surjective map $C(H',\QQ_p)^{\oplus m}\to \tilde{H}^d(K^p)_{\QQ_p}$.
Since the completed cohomology is an admissible representation of $H$, by Proposition \ref{prop-(h)-an}, this map remains surjective when restricted to the $(h)$-analytic vectors, i.e. (note that $C(H',\QQ_p)^{(h)}=\mathrm{LA}^{(h)}(H')$)
\[\mathrm{LA}^{(h)}(H')^{\oplus m}\to \tilde{H}^d(K^p)^{(h)}_{\QQ_p}\]
is surjective. It suffices to prove that $(\mathrm{LA}^{(h)}(H')\otimes_{\QQ_p}E)^a$ is dense in $\mathrm{LA}^{(h)}(H')\otimes_{\QQ_p}E$. We have
\[\mathrm{LA}^{(h)}(H')_E=\mathrm{LA}^{(h)}(H_1)_E\widehat\otimes_E \mathrm{LA}^{(h)}(H'_0)_E\]
where $\mathrm{LA}^{(h)}(H')_E=\mathrm{LA}^{(h)}(H')\otimes_{\QQ_p}E$ and $\mathrm{LA}^{(h)}(H_1)_E,\mathrm{LA}^{(h)}(H'_0)_E$ are defined similarly. By unravelling the definition of $(\mathrm{LA}^{(h)}(H')\otimes_{\QQ_p}E)^a$, we need to show the following lemma.
\end{proof}

\begin{lem}
\begin{enumerate}
\item The natural map
\[\bigoplus_{\underline{l}\in\ZZ^e} \Hom_{E[H'_0]}(\psi_{\underline{l}}, \mathrm{LA}^{(h)}(H'_0)_E)\otimes \psi_{\underline{l}}\to \mathrm{LA}^{(h)}(H'_0)_E\]
has dense image.
\item The natural map
\[\bigoplus_{\underline{k}\in(2\NN)^d\mbox{ very regular}} \Hom_{E[H_1]}(\rho_{\underline{k}}, \mathrm{LA}^{(h)}(H_1)_E)\otimes \rho_{\underline{k}}\to \mathrm{LA}^{(h)}(H_1)_E\]
has dense image.
\end{enumerate}
\end{lem}

\begin{proof}
The first part is easily reduced to the case $H'_0=\ZZ_p$. We leave it as an exercise. For the second part, we first note  that the natural map of summation over all the weights
\[\bigoplus_{\underline{k}\in(2\NN)^d} \Hom_{E[H_1]}(\rho_{\underline{k}}, \mathrm{LA}^{(h)}(H_1)_E)\otimes \rho_{\underline{k}}\to \mathrm{LA}^{(h)}(H_1)_E\]
has dense image. Indeed by viewing $H_1$ as a principal congruence subgroup of $\mathrm{Res}_{F/\QQ}\PGL_2(\QQ_p)$,  we see that the image contains all the algebraic functions, and hence it is dense in the space of analytic functions and $\mathrm{LA}^{(h)}(H_1)_E$ as well \cite[Appendix]{Paskunas-2014}.
Let 
\[A_{\underline{k}}:=\im\left(\Hom_{E[H_1]}(\rho_{\underline{k}}, \mathrm{LA}^{(h)}(H_1)_E)\otimes\rho_{\underline{k}}\to \mathrm{LA}^{(h)}(H_1)_E\right)\]
 It remains to show that for every $\underline{k}\in(2\NN)^d$, every $v\in A_{\underline{k}} $ can be approximated by linear combinations of vectors in $A_{\underline{k'}}$ with $\underline{k'}$ very regular. For $i\in\{1,\cdots,d\}$, consider the map 
\[\varphi_i:H\subseteq \mathrm{SL}_2(F\otimes \QQ_p)\xrightarrow{\tau_i}  \mathrm{SL}_2(E).\]
We write $\varphi_i(g)=\begin{pmatrix} a_i(g) & b_i(g) \\ c_i(g) & d_i(g)\end{pmatrix}$. Then 
\[\bar{a}_i(g):=\frac{a_i(g)^2}{\det \varphi_i(g)}\]
defines an element $\bar{a}_i\in \mathrm{LA}^{(h)}(H_1)_E$. 
\begin{lem}
$\bar{a}_i$  belongs to $A_{2e_i}$, where $e_i\in \NN^d$ has $1$ at the $i$-th place and zero elsewhere. Similarly, $\bar{a}_i^m\in A_{2me_i}$ for $m\geq 1$. Moreover $\bar{a}_i^{p^m}-1$ has norm at most $p^{-m-2}$ in $\mathrm{LA}^{(h)}(H_1)_E$.
\end{lem}

\begin{proof}
To see the first claim one looks at the matrix coefficients of $\Sym^2 \mathrm{Std}\otimes \det^{-1}$ with basis $v_1^2, v_1v_2, v_2^2$ where $v_1,v_2$ denote the standard basis of the standard representation $\mathrm{Std}$ of $\GL_2$. By considering $\Sym^{2m} \mathrm{Std}\otimes \det^{-m}$ for $m\geq 1$, one gets the claim for $\bar{a}_i^m$. For the last claim, note that $\bar{a}_i$ is analytic and $\bar{a}_i-1$ has norm $\leq p^{-2}$ in $C^{\an}(H_1,E)$ by the definition of $H_1$. Thus $\bar{a}_i-1$ has norm $\leq p^{-2}$ in $\mathrm{LA}^{(h)}(H_1)_E$ as well, and  the bound on  $\bar{a}_i^{p^m}-1$ follows by a standard argument.
\end{proof}

Now given $v\in A_{\underline{k}}$ and $\underline{k}=(k_1,\cdots,k_d)$ and $N\geq 1$, we define $k'_1,\cdots,k'_d$ inductively as follows to approximate $v$. Let $k'_1=p^N$. Suppose we already construct $k'_1,\cdots,k'_i$. Choose $k'_{i+1}$ a power of $p$ satisfying
\[k'_{i+1}>(k_1+\cdots+k_{i+1})/2+k'_1+\cdots+k'_i+d.\]
Define $v':=v\bar{a}_1^{k'_1}\bar{a}_2^{k'_2}\cdots \bar{a}_d^{k'_d}$. By our choice of $k'_i$ and the previous lemma, we have
\[||v'-v||\leq ||v|| p^{-N}.\]
We are left to show that $v'$ is a linear combinations of vectors in $A_{\underline{j}}$ with $\underline{j}$ very regular. The multiplication structure on $\mathrm{LA}^{(h)}(H_1)_E$ induces a map
\[A_{\underline{k}}\otimes_E A_{2k'_1e_1}\otimes_E \cdots \otimes_E  A_{2k'_de_d}\to \mathrm{LA}^{(h)}(H_1)_E.\] 
To understand the image, it's enough to look at the tensor product 
\[\rho_{\underline{k}}\otimes_E \rho_{2k'_1e_1}\otimes_E \cdots\otimes_E \rho_{2k'_de_d} \]
as a representation of $\mathrm{PGL}_2(F\otimes_{\QQ} E)=\prod_{i=1}^d\mathrm{PGL}_2(E)$. A simple analysis of the highest weights tells us  that its irreducible subrepresentations are of the form $\rho_{\underline{j}}$ with 
\[2k'_i-k_i\leq j_i\leq 2k'_i+k_i.\]
Such a $\underline{j}$ has to be very regular: clearly $j_i>0$. Let $I$ and $I'$ be two distinct subsets of $\Sigma'$. There exists $M\in\{1,\cdots,d\}$ such that for $i>M$, either $\tau_i\in I\cap I'$ or $\tau_i\notin I\cup I'$, and $\tau_M\in I$ and $\tau_M\notin I'$ (after possibly switching $I$ and $I'$). Then
\[\sum_{i\in I} j_i +|I| -(\sum_{i\in I'} j_i+|I'|)\geq j_M-(\sum_{i=1}^{M-1}j_i +d)\geq 2k'_M-k_M-(\sum_{i=1}^{M-1}(2k'_i+k_i) +d)> 0\]
by our choice of $k'_M$. Hence all possible $\sum_{i\in I} j_i +|I|$'s are distinct.
\end{proof}

\begin{rmk} \label{rmk-(h)dense}
If $\underline{k}=\{k_\tau\}_{\tau\in \Sigma'}$ is very regular as in Definition \ref{dfn-veryregwt}, a permutation of these $k_\tau$'s is very regular as well. Therefore the action of $\Gal(E/\QQ_p)$ on $\tilde{H}^d(K^p)^{(h)}_{E}=\tilde{H}^d(K^p)^{(h)}\otimes_{\QQ_p} E$ (by acting on $E$) preserves the dense subspace $\tilde{H}^d(K^p)^{(h),a}_{E}$. Let 
\[\tilde{H}^d(K^p)^{(h),a}:=\tilde{H}^d(K^p)^{(h),a}_{E}\cap \tilde{H}^d(K^p)^{(h)}=(\tilde{H}^d(K^p)^{(h),a}_{E})^{\Gal(E/\QQ_p)}.\]
By Galois descent,
\[\tilde{H}^d(K^p)^{(h),a}\otimes_{\QQ_p} E=\tilde{H}^d(K^p)^{(h),a}_{E},\]
and  by Theorem \ref{thm-(h)dense}, $\tilde{H}^d(K^p)^{(h),a}$ is a dense subspace of $\tilde{H}^d(K^p)^{(h)}$.
\end{rmk}

\subsection{Partial Sen operators}
We construct some commuting operators on $\tilde{H}^i(K^p)^{\la}\widehat\otimes_{\QQ_p} \CC_p:=\varinjlim_{K_p}\tilde{H}^i(K^p)^{K_p-\an}\widehat\otimes_{\QQ_p} \CC_p$. In the next subsection we will show that these operators are the sought-after partial Sen operators. In the case of modular curves ($d=1$), the possibility of such a geometric construction of the Sen operator was observed by the second author in \cite{Pan-2022}. For $d\geq 2$, we will need the work \cite{RC-2025}.  

By \cite[\S 6]{Pan-2025JEMS}, the action of $G_{\QQ}$ on $\tilde{H}^i(K^p)^{\la}$ is locally analytic, i.e. for any uniform pro-$p$ open subgroup $K_p$ of $\mathrm{G}(\QQ_p)$, the natural map  $G_{\QQ}\to \End(\tilde{H}^i(K^p)^{K_p-\an})$ is continuous with respect to the $p$-adic topology of the target. Hence by Sen's result, there is a Sen operator $\theta_{\tilde{H}^i}$ on $\tilde{H}^i(K^p)^{\la}\widehat\otimes_{\QQ_p} \CC_p$.

Let $\mathfrak{g}=\mathfrak{gl}_2(F\otimes_{\QQ} \CC_p)$. There is a natural decomposition induced by embeddings of $F$ into $\CC_p$
\[\mathfrak{g}=\prod_{\tau:F\to \CC_p} \mathfrak{g}_\tau\]
with $\mathfrak{g}_\tau\cong \mathfrak{gl}_2(\CC_p)$. Let $\mathfrak{b}\subseteq \mathfrak{g}$ be the usual upper-triangular Borel subalgebra and $\mathfrak{h}\subseteq \mathfrak{g}$ the diagonal matrices, which is a Cartan subalgebra.
Similarly there are decompositions of $\mathfrak{b}$ and $\mathfrak{h}$ according to embeddings of $F$ into $\CC_p$. We will use $\diag(a,b)$ to denote the matrix with diagonal entries $a,b$.
Let $W$ be the Weyl group of $\mathfrak{g}$ and $\rho$ the half-sum of positive roots (with respect to $\mathfrak{b}$). The Harish-Chandra isomorphism produces an injection of algebras
\[\gamma^{\mathfrak{m}}: Z(U(\mathfrak{g}))\to U(\mathfrak{h})\]
whose image consists of functions on $\mathfrak{h}^*$ invariant under the action of $W$ \textit{centered at $\rho$}. 
Note that under our normalization here, an irreducible representation of $\mathfrak{g}$ of \textit{lowest weight} $\lambda$ has infinitesimal character corresponding to the $W$-orbit $\{w(\lambda-\rho)+\rho\}_{w\in W}$. 

The Lie algebra action of $\mathrm{G}(\QQ_p)$ on $\tilde{H}^i(K^p)^{\la}$ induces a natural $\CC_p$-algebra action of $Z(U(\mathfrak{g}))$.

\begin{thm} \label{thm-Sendiag01}
There is a natural action of $U(\mathfrak{h})$ on $\tilde{H}^i(K^p)^{\la}\widehat\otimes_{\QQ_p} \CC_p$ extending the action of $Z(U(\mathfrak{g}))$ via $\gamma^{\mathfrak{m}}$. This action commutes with the action of $G_{\QQ_p}\times \mathrm{G}(\QQ_p)$ and $\TT$. Moreover the action of $\diag(0,1)\in \mathfrak{h}$ agrees with the Sen operator on $\tilde{H}^i(K^p)^{\la}\widehat\otimes_{\QQ_p} \CC_p$.
\end{thm}

\begin{proof}
The existence of the action of $U(\mathfrak{h})$ follows from \cite[Th.1.1.9]{RC-2025}, as explained in \cite[Th.4.4.4]{Lan-Pan}, using the isomorphism \cite[Th.1.1.8]{RC-2025}.  The claim for the Sen operator follows from  \cite[Th.1.1.10]{RC-2025}, where a sign difference comes from the fact that the convention for the class field theory used in the reference is inverse to our convention.
\end{proof}

\begin{dfn} \label{dfn-thetataula}
For each embedding $\tau:F\to \CC_p$, we define $\theta^{\la}_\tau\in \End(\tilde{H}^i(K^p)^{\la}\widehat\otimes_{\QQ_p} \CC_p)$ as the action of the element $\diag(0,1)\in\mathfrak{h}_\tau\subseteq U(\mathfrak{h})$ in the previous theorem.
\end{dfn} 
The action of $\sum_{\tau\in\Sigma'}\theta_\tau^{\la}$ on $\tilde{H}^i(K^p)^{\la}\widehat\otimes_{\QQ_p} \CC_p$ agrees the Sen operator.

\begin{cor}\label{cor-thetatauev}
For $\tau:F\to \CC_p$, let $k_\tau\in \ZZ_{\geq 0}$ and $l_\tau\in \CC_p$ and denote by $V_{k_\tau,l_\tau}$ the irreducible representation of $\mathfrak{g}_\tau=\mathfrak{gl}_2(\CC_p)$ of highest weight $(k_\tau+l_\tau,l_\tau)$. Let $V$ be the tensor product of all $V_{k_\tau,l_\tau}$, viewed as a representation of $\mathfrak{g}=\prod \mathfrak{g}_\tau$. Then the action of $\theta^{\la}_{\tau}$  on 
\[\Hom_{\mathfrak{g}}(V,\tilde{H}^i(K^p)^{\la}\widehat\otimes_{\QQ_p} \CC_p)\]
satisfies $(\theta^{\la}_\tau-k_\tau-l_\tau)(\theta^{\la}_\tau-l_\tau+1)=0$. In particular, it's a semisimple operator with eigenvalues $k_\tau+l_\tau$ and $l_\tau-1$. 
\end{cor}

\begin{proof}
This follows directly from the previous theorem by noting that the lowest weight of $V_{k_\tau,l_\tau}$ is $(l_\tau,k_\tau+l_\tau)$ and has $(k_\tau+l_\tau+1,l_\tau-1)$ in its $W$-orbit (centered at $\rho$).
\end{proof}

\subsection{Proof of Theorem \ref{thm-plecticactnonCM}} \label{sec-proof of MT}
First we need a few results. Recall that as in Definition \ref{dfn-rhokVk}, given $\underline{k}=\{k_\tau\}_{\tau\in\Sigma'}\in (2\mathbb{N})^{\Sigma'}$, we introduced a representation $( \rho_{\underline{k}}, V_{\underline{k}})$ of $\mathrm{PGL}_2(F\otimes \QQ_p)$. 

\begin{prop} \label{prop-Zdec}
 Let $K_p$ be an open subgroup of $\mathrm{G}(\QQ_p)=\mathrm{GL}_2(F\otimes \QQ_p)$ so that $K^pK_p$ is neat. Let $\underline{k}=\{k_\tau\}_{\tau\in\Sigma'}$ be a weight such that all $k_\tau$'s are distinct and positive. Let 
 \[Z=\Hom_{E[K_p]}(V_{\underline{k}}, \tilde{H}^d(K^p)_{E})\otimes_{E}\overline{\QQ_p}\]
As a module of $\TT[G_\QQ]$, there is a natural finite decomposition
\[Z=\bigoplus_{\lambda:\TT\to \overline{\QQ_p}}Z[\lambda].\]
For $\lambda$ showing up on the right hand side, $\rho_\lambda$ has $\tau$-Hodge-Tate weights $k_\tau/2+1,-k_\tau/2$ for $\tau:F\to E\subseteq \overline{\QQ_p}$. 
Moreover if $\lambda$ is not in the CM locus, the $G_{\QQ}$-representation  $\oInd^{\QQ}_F \rho_\lambda$ is irreducible and each $Z[\lambda]$ is $\oInd^{\QQ}_F \rho_\lambda$-isotypic, i.e. 
\[I_{CM} Z[\lambda]= (\oInd^{\QQ}_F \rho_\lambda)^{\oplus n_\lambda}\]
for some $n_\lambda\geq 0$. 
\end{prop}

\begin{rmk}
By Nekov\'a\v{r}'s result, the claim holds for $\lambda$ in the CM-locus as well. 
\end{rmk}

\begin{proof}
Given a weight $\underline{k}$, the $\mathrm{PGL}_2(F\otimes \QQ_p)$-representation $( \rho_{\underline{k}}, V_{\underline{k}})$ (viewed as a representation of $K_p$) defines an \'etale $E$-local system $\mathbb{L}_{\underline{k}}$ on $\Sh_{K^pK_p}$. 

\begin{lem}
Suppose that $\underline{k}=\{k_\tau\}_{\tau\in\Sigma'}$ satisfies 
that $k_\tau>0$ for every $\tau\in\Sigma'$. There is a natural isomorphism 
\[\Hom_{E[K_p]}(V_{\underline{k}}^\vee, \tilde{H}^d(K^p)_{E})= H^d_\et(\Sh_{K^pK_p,\overline{\QQ}}, \mathbb{L}_{\underline{k}})\]
where $V_{\underline{k}}^\vee$ denotes the dual representation of $V_{\underline{k}}$ (which is isomorphic to $V_{\underline{k}}$). 
\end{lem}

\begin{proof}
This is essentially \cite[Remark 6.1.6]{Lan-Pan}. We sketch a proof here.
Let $\widetilde{K_p}$ be the image of $K_p$ in the automorphism group of the tower $\{\Sh_{K^pK'_p}\}_{K'_p\subseteq \mathrm{G}(\QQ_p)}$ and $\tilde{\mathfrak{g}}$ be its Lie algebra which is a quotient of $\mathfrak{g}$. 
By \cite[Th.0.5, Th.1.1.3]{Emerton-2006}, there is a natural spectral sequence
\[E_2^{i,j}=\Ext^i_{\widetilde{\mathfrak{g}}}(V_{\underline{k}}^\vee, \tilde{H}^{j}(K^p)^{\la}_{E})\Longrightarrow \varinjlim_{K'_p} H^{i+j}_\et(\Sh_{K^pK'_p,\overline{\QQ}}, \mathbb{L}_{\underline{k}}).\]
(Emerton takes the direct limit over all tame levels. We can take the $K^p$-invariants of the spectral sequence to get the above  one.)
By \cite[Th.1.13]{RC-2025}, $\tilde{H}^{>d}(K^p)^{\la}=0$. Moreover since all $k_\tau$'s are not zero, the highest weight of $V_{\underline{k}}^\vee$ is regular. Theorem 1.1.2 of \cite{Lan-Pan} shows that the action of $Z(U(\mathfrak{g}))$ on $\tilde{H}^{<d}(K^p)^{\la}$ factors through $Z(U(\mathfrak{g}))/I^m$ for some $m>0$, where $I$ is the ideal  defining non-regular infinitesimal characters, and hence the infinitesimal character of $V_{\underline{k}}$ doesn't factor through $Z(U(\mathfrak{g}))/I$. The same argument as the proof of
\cite[Th.1.1]{Milicic-2014} shows that for any $i\geq0$, 
\[\Ext^{i}_{\widetilde{\mathfrak{g}}}(V_{\underline{k}}^\vee, \tilde{H}^{<d}(K^p)^{\la}_{E})=0.\]
Therefore the previous spectral sequence degenerates and implies that
\[\Hom_{\widetilde{\mathfrak{g}}}(V_{\underline{k}}^\vee, \tilde{H}^{d}(K^p)^{\la}_{E})\cong \varinjlim_{K'_p} H^{d}_\et(\Sh_{K^pK'_p,\overline{\QQ}}, \mathbb{L}_{\underline{k}}).\]
Both sides have natural actions of $K_p$ and the isomorphism is $K_p$-equivariant. We get our claim by taking the $K_p$-invariants.
\end{proof}

\begin{lem}
Let $j:\Sh_{K^pK_p}\to \Sh_{K^pK_p}^*$ denote the open immersion of $\Sh_{K^pK_p}$ into the minimal compactification. Suppose that weight $\underline{k}$ is not parallel. The natural map
\[H^d_{\et}((\Sh_{K^pK_p}^*)_{\overline{\QQ}},j_{!*}\mathbb{L}_{\underline{k}})\to H^d_{\et}(\Sh_{K^pK_p,\overline{\QQ}}, \mathbb{L}_{\underline{k}}) \]
is an isomorphism. 
\end{lem}

\begin{proof}
By the comparison theorem, it suffices to consider the analytic variety $\Sh_{K^pK_p}(\CC)$ and prove the corresponding statement for the $\CC$-local systems.
Since $\underline{k}$ is not parallel, $\mathbb{L}_{\underline{k}}$ has no boundary cohomology \cite[Th.1]{Harder-1987}. Hence the natural map $H^d_{c,\et}(\Sh_{K^pK_p,\overline{\QQ}}, \mathbb{L}_{\underline{k}})\to H^d_{\et}(\Sh_{K^pK_p,\overline{\QQ}}, \mathbb{L}_{\underline{k}})$ is an isomorphism and factors through $H^d_{\et}((\Sh_{K^pK_p}^*)_{\overline{\QQ}},j_{!*}\mathbb{L}_{\underline{k}})$. 
Moreover the cohomology of $j_{!*}\mathbb{L}_{\underline{k}}$ only has contribution from the cuspidal representations and maps injectively into $H^d_{\et}(\Sh_{K^pK_p,\overline{\QQ}}, \mathbb{L}_{\underline{k}})$, cf. \cite[\S 3.2]{Harder-1987}. Thus $H^d_{\et}((\Sh_{K^pK_p}^*)_{\overline{\QQ}},j_{!*}\mathbb{L}_{\underline{k}})\cong H^d_{\et}(\Sh_{K^pK_p,\overline{\QQ}}, \mathbb{L}_{\underline{k}})$.
\end{proof}

\begin{lem}
Suppose that $\underline{k}=\{k_\tau\}_{\tau\in\Sigma'}$ with all $k_\tau$'s being distinct. Let 
\[X=H^d_{\et}((\Sh_{K^pK_p}^*)_{\overline{\QQ}},j_{!*}\mathbb{L}_{\underline{k}})\otimes_E \overline{\QQ_p}\]
Then $X$ is finite-dimensional and there is a natural finite decomposition
\[X=\bigoplus_{\lambda:\TT\to \overline{\QQ_p}}X[\lambda] \]
and $\rho_\lambda$ has $\tau$-Hodge-Tate weights $k_\tau/2+1,-k_\tau/2$ for $\tau:F\to E\subseteq \overline{\QQ_p}$ if $\lambda$ shows up in $X$. 
Moreover for $\lambda$ whose kernel does not contain $I_{CM}$, $\oInd^{\QQ}_F \rho_\lambda$ is an irreducible representation, and the $G_{\QQ}$-representation $X[\lambda]$ is $\oInd^{\QQ}_F \rho_\lambda$-isotypic, i.e. 
\[X[\lambda]\cong \oInd^{\QQ}_F \rho_\lambda\otimes \Hom_{\overline{\QQ_p}[G_{\QQ}]}(\oInd^{\QQ}_F\rho_\lambda, X[\lambda]).\]
\end{lem}

\begin{proof}
This is essentially \cite[Th.5.20]{Nekovar-2018}. We recall its argument here.
By the discussion in \cite[\S 5.9-5.11]{Nekovar-2018}, there is a natural decomposition of $\mathrm{G}(\mA_f)\times G_{\QQ}$-representation
\[\varinjlim_{K\subseteq \mathrm{G}(\mA_f)} H^d_{\et}((\Sh_{K^pK_p}^*)_{\overline{\QQ}},j_{!*}\mathbb{L}_{\underline{k}})\otimes_E \overline{\QQ_p} =\bigoplus_{\pi=\pi^{\infty}\otimes\pi_\infty} \pi^{\infty}\otimes V(\pi^\infty)\]
where $\pi$ corresponds to a representation of $\GL_2(\mathbb{A}\otimes F)$ attached (up to a twist) to a holomorphic cuspidal Hilbert modular eigenform (under a fixed isomorphism between $\CC$ and $\overline{\QQ_p}$), and $ V(\pi^\infty)$ is a representation of $G_{\QQ}$ of dimension $2^d$ whose semi-simplification 
\[ V(\pi^\infty)^{\mathrm{ss}}\cong \oInd^{\QQ}_F\rho_\lambda,\]
where $\lambda$ is the system of Hecke eigenvalues associated to $\pi$.
Note that because $\underline{k}$ is not parallel, the case $A$ in the reference cannot happen. As remarked by Nekov\'a\v{r} before (5.12), since our Hodge cocharacter is inverse to his choice, we have $\rho_\lambda$ instead of $\rho_\lambda^{\vee}(-1)$ in the tensor induction. Take the $K^pK_p$-invariants. We get
\[X=\bigoplus_{\lambda:\TT\to \overline{\QQ_p}}X[\lambda], \]
and each $\lambda$-eigenspace is a direct sum of $V(\pi^\infty)$'s with $V(\pi^\infty)^{\mathrm{ss}}\cong  \oInd^{\QQ}_F\rho_\lambda$.
Now if $\ker\lambda$ does not contain $I_{CM}$, the system of Hecke eigenvalues $\lambda$ corresponds (up to a twist) to a holomorphic non-CM cuspidal Hilbert modular eigenform. Thus by the known local-global compatibility \cite[Prop.6.1]{Nekovar-2018} $\rho_\lambda$ is strongly irreducible  and the difference of the $\tau$-Hodge-Tate weights  are $k_\tau+1$ for $\tau:F\to \overline{\QQ_p}$. Since the local system comes from a representation of $\mathrm{PGL}_2$, by Theorem \ref{thm-big T}  $\det \rho_\lambda$ is the inverse of the cyclotomic character up to a finite character. Hence $\rho_\lambda$ has $\tau$-Hodge-Tate weights  $k _\tau/2+1,-k_\tau/2$.  Since all $k_\tau$'s are distinct, it follows from Proposition \ref{proptenindirred} that $ \oInd^{\QQ}_F\rho_\lambda$ is irreducible. Hence $V(\pi^\infty)\cong  \oInd^{\QQ}_F\rho_\lambda$ if $V(\pi^\infty)^{\mathrm{ss}}\cong  \oInd^{\QQ}_F\rho_\lambda$. This concludes the proof.
\end{proof}

Proposition \ref{prop-Zdec} follows easily from the combination of previous lemmas.
\end{proof}

We also need to allow twists of $V_{\underline{k}}$ in Proposition \ref{prop-Zdec}. 

\begin{prop} \label{prop-twist}
Let $\phi\in\tilde{H}^0(K^p)_E$ be a non-zero vector on which $\mathrm{G}(\mA_f)$ acts via the character $\psi\circ\det$ for some continuous character
\[\psi: F^{\times}_{>>0}\setminus (\mA_f\otimes_{\QQ} F)^\times/\det K^p\to E^\times.\]
Set $\psi_p=\psi\circ\det|_{\mathrm{G}(\QQ_p)}$ and denote by $\rho_\psi:\Gal(\overline{F}/F)\to  E^\times$  the character corresponding to $\psi$ under the class field theory.
Let $K_p$ be an open subgroup of $\mathrm{G}(\QQ_p)$ and $\underline{k}$ a weight as before. The multiplication by $\phi$ on $\tilde{H}^d(K^p)_{E}$  induces a $G_\QQ$-equivariant isomorphism
\[\times \phi: \Hom_{E[K_p]}(V_{\underline{k}}, \tilde{H}^d(K^p)_{E}) \otimes_E (\oInd^{\QQ}_F\rho_{\psi}) \stackrel{\cong}{\to} \Hom_{E[K_p]}(V_{\underline{k}}\otimes (\psi_p\circ\det), \tilde{H}^d(K^p)_{E}), \]
mapping $\lambda$-eigenspace to $\lambda\otimes\psi$-eigenspace for $\lambda:\TT\to\overline{\QQ_p}$ showing up in $\Hom_{E[K_p]}(V_{\underline{k}}, \tilde{H}^d(K^p)_{E})\otimes_E \overline{\QQ_p}$, where $\lambda\otimes \psi:\TT\to\overline{\QQ_p}$ is determined by $\rho_{\lambda\otimes\psi}\cong \rho_\lambda\otimes \rho_{\psi}$.
\end{prop}

\begin{proof}
$\phi$ is a function on $\pi_0(K^p):=\varinjlim_{K'_p\subseteq{\mathrm{G}(\QQ_p)}}\pi_0(\Sh_{K^pK'_p,\overline{\QQ}})$. It follows from the description of connected components of Shimura varieties that $\mathrm{G}(\mA_f)$ acts transitively on $\pi_0(K^p)$, and hence $\phi$ is an invertible function on $\pi_0(K^p)$. Thus $\times\phi$ is an isomorphism because it has an inverse $\times\phi^{-1}$. On the other hand, we have a commutative diagram involving the natural actions of $G_{\QQ}$ and $F^\times_{>>0}\setminus  (\mA_f\otimes_{\QQ} F)^\times$ on $\pi_0(K^p)$
\[\begin{tikzcd}
\QQ^\times_{>0}\setminus \mA_f^\times \arrow[r,"\mathrm{Art}"] \arrow[d] & G_{\QQ}^{\mathrm{ab}} \arrow[d]\\
F^\times_{>>0}\setminus  (\mA_f\otimes_{\QQ} F)^\times \arrow[r] & \mathrm{Aut}(\pi_0(K^p))
\end{tikzcd}\]
where $\mathrm{Art}$ denotes the Artin map and the left vertical map is the natural inclusion. By class field theory, $G_{\QQ}$ acts on $\phi$ via the transfer of $\psi$ which is nothing but $\oInd^{\QQ}_F\rho_{\psi}$ cf. Remark \ref{rmk-transfer}. The claim for the Hecke action can be checked easily. We omit the detail here.
\end{proof}

\begin{thm} \label{thm-eigenspace}
Same setup as in the previous proposition. Assume moreover that all $k_\tau$'s are distinct and positive. The action of $\TT$ on the finite-dimensional $E$-vector space 
\[X:=\Hom_{E[K_p]}(V_{\underline{k}}\otimes (\psi_p\circ\det), \tilde{H}^d(K^p)_{E})\] 
is semisimple, and for $\lambda:\TT\to\overline{\QQ_p}$ showing up as an eigenvalue and not in the CM locus,
\begin{enumerate}
\item the $G_{\QQ}$-representation  $\oInd^{\QQ}_F \rho_\lambda$ is irreducible and the $\lambda$-eigenspace of $ X\otimes_E \overline{\QQ_p}$ is $\oInd^{\QQ}_F \rho_\lambda$-isotypic. 
\item For $\tau:F\to E\subseteq \overline{\QQ_p}$,
 $\rho_\lambda$ has $\tau$-Hodge-Tate weights $k_\tau/2+1-w_\tau,-k_\tau/2-w_\tau$, where $w_\tau\in E$ is the weight of $d\psi_p$ at $\tau$ (note that $d\psi_p: \Lie (F\otimes \QQ_p)^\times \to E$ induces an $E$-linear map $\Lie (F\otimes \QQ_p)^\times\otimes_{\QQ_p} E=F\otimes_{\QQ} E=\prod_{\tau\in\Sigma'} E\to E$.)
 \item The $G_{\QQ_p}$-representation $\oInd^{\QQ}_F \rho_\lambda$ has $2^d$ distinct Hodge-Tate-Sen weights if
$(\rho_{\underline{k}}, V_{\underline{k}})$ is of very regular weight as in Definition \ref{dfn-veryregwt}.
\item \label{thm-eigenspace-4} For $\tau\to \overline{\QQ_p}$, the action of $\theta^{\la}_\tau$ (introduced in Definition \ref{dfn-thetataula}) on $X\widehat\otimes_{E}\CC_p$ is semisimple with eigenvalues $-k_\tau/2-1+w_\tau$ and $k_\tau/2+w_\tau$. 
\item The action of $\sum_{\tau\in\Sigma'}\theta_\tau^{\la}$ on $X\widehat\otimes_{E}\CC_p$ agrees with the Sen operator.
\end{enumerate} 
\end{thm}

\begin{proof}
The first two parts follow  from Proposition \ref{prop-Zdec} and Proposition \ref{prop-twist} by noting that for $\lambda:\TT\to\overline{\QQ_p}$, there is an isomorphism of  $G_{\QQ}$-representations
\[\oInd^{\QQ}_F \rho_{\lambda\otimes\psi} \cong (\oInd^{\QQ}_F \rho_\lambda) \otimes (\oInd^{\QQ}_F \rho_\psi).\]
Indeed the plectic induction is compatible with tensor product (Remark \ref{rmk-resprodGtauF}), so is the tensor induction as well. It follows from Proposition \ref{prop-wT=sumpartialwt} that $\oInd^{\QQ}_F \rho_\lambda|_{G_{\QQ_p}}$ has Hodge-Tate-Sen weights $\sum_{\tau\in I}k_\tau+|I|+w$, where $w=\sum_{\tau\in\Sigma'}(-k_\tau/2-w_\tau)$ and $I$ ranges over all subsets of $\Sigma'$. Hence they are distinct by the definition of very regular weights. Part \eqref{thm-eigenspace-4} is a direct consequence of Corollary \ref{cor-thetatauev} by noting that the $\tau$-part of $V_{\underline{k}}\otimes (\psi_p\circ\det)$ has highest weight $(k_\tau/2+w_\tau,-k_\tau/2+w_\tau)$. The last claim is Theorem \ref{thm-Sendiag01}.
\end{proof}

The next lemma allows us to choose $\psi$ freely on the center of $\widetilde{K_p}$ after enlarging $E$ if necessary. 

\begin{lem} \label{lem-centeractonpi0}
Let $K_p$ be an open subgroup of $\mathrm{G}(\QQ_p)$ so that $K^pK_p$ is neat. Denote the center of $\widetilde{K_p}$ by $Z(\widetilde{K_p})$. Suppose that $Z(\widetilde{K_p})$ is torsionfree. 
\begin{enumerate}
\item The action of $Z(\widetilde{K_p})$ on $\pi_0(K^p)$  is free and has finitely many orbits.
\item After possibly enlarging $E$, for any continuous character $\psi_0:Z(\widetilde{K_p})\to \QQ_p^\times$, we can find  $\phi\in\tilde{H}^0(K^p)_E$  such that $\phi$ is an eigenfunction of $\mathrm{G}(\mA_f)$ and $Z(\widetilde{K_p})$ acts on it via $\psi_0$.
\end{enumerate}
\end{lem}

\begin{proof}
The first claim is \cite[Lem.3.2.8]{Pan-2025cchmd}. It follows that the natural map (induced by  the determinant map)
$Z(\widetilde{K_p})\to \pi_0(K^p)=(\mA_f\otimes_{\QQ} F)^\times/\overline{F^{\times}_{>>0}\det K^p}$ is injective with finite cokernel, where $\overline{F^{\times}_{>>0}\det K^p}$ denotes the closure of $F^{\times}_{>>0}\det K^p$. Say $N=[\pi_0(K^p):Z(\widetilde{K_p})]$. We can take $E$ to contain all $N$-th roots of numbers in $\QQ_p$ in the second part. We omit the detail here.
\end{proof}

Now we prove Theorem \ref{thm-plecticactnonCM}.
\begin{proof}[Proof of Theorem \ref{thm-plecticactnonCM}]
 Let $H=H_1\times H_0$ be an open subgroup of $\mathrm{G}(\QQ_p)$ above Definition \ref{dfn-rhokVk} and by shrinking $H_0$ if necessary, we may assume that $Z(\widetilde{H})$ is torsionfree. Fix $h\geq 1$. We can consider the subspace of $(h)$-analytic vectors $\tilde{H}^d(K^p)^{(h)}$ as in Section \ref{subsec-density}. By \cite[Th.6.1]{Pan-2025JEMS} and Proposition \ref{prop-(h)-an}, the action of $R_{\bar{D}}$ on $\tilde{H}^d(K^p)^{(h)}$ can be extended to an action of $R_{\bar{D}}^{(r)}$ for some $r=r_h\in p^{\QQ}\cap (0,1)$. Let 
\[W:=(I_{CM}\tilde{H}^d(K^p))^{(h)}=I_{CM}\tilde{H}^d(K^p)^{(h)},\]
where the second equality follows from Proposition \ref{prop-(h)-an}\eqref{prop-(h)-an-3} and the argument in the proof of Lemma \ref{lem-ICMclosed}. Let 
\[W^a:=I_{CM} \tilde{H}^d(K^p)^{(h),a},\] 
where $\tilde{H}^d(K^p)^{(h),a}$ was introduced in Remark \ref{rmk-(h)dense} and Definition \ref{dfn-veryregularsubspace}. Since  $\tilde{H}^d(K^p)^{(h),a}$ is dense in  $\tilde{H}^d(K^p)^{(h)}$ by Remark \ref{rmk-(h)dense}, $W^a\subseteq W$ is dense as well.

We now want to apply Corollary \ref{cor-extplecSen} to the $R_{\bar{D}}[G_{\QQ}]$-submodule $W^a$ of $W$ by taking the operators $\theta_{W,\tau}$ to be $\theta_\tau^{\la}|_{W\widehat\otimes_{\QQ_p}\CC_p}$ (cf. Definition \ref{dfn-thetataula}).
To see that it satisfies all the conditions, by Theorem \ref{thm-eigenspace} and Lemma \ref{lem-centeractonpi0} and the construction of $\tilde{H}^d(K^p)^{(h),a}$, 
the action of $\TT$ on $\tilde{H}^d(K^p)^{(h),a}$ is semisimple with countable eigenvalues. Hence the action of $\TT$ on $W^a$ is semisimple with eigenvalues in the non-CM-locus, and the eigenspaces satisfy the conditions in Corollary \ref{cor-extplecSen} by Theorem \ref{thm-eigenspace}. We thus obtain a \textit{unique}  $A^{(r_h)}$-module structure on $(I_{CM}\tilde{H}^d(K^p))^{(h)}$ extending the $R_{\bar{D}}[G_{\QQ}]$-module structure. The uniqueness implies that it's compatible with the $\mathrm{G}(\QQ_p)$-action: given $g\in\mathrm{G}(\QQ_p)$, suppose that 
\[(I_{CM}\tilde{H}^d(K^p))^{(h)}\xrightarrow{\times g} g(I_{CM}\tilde{H}^d(K^p))^{(h)}\subseteq (I_{CM}\tilde{H}^d(K^p))^{(h')}\]
for some $h'\geq 1$; the composite map is $A^{(r')}$-linear for $r'=\max (r_h,r_{h'})$.  In particular the $A^{(r)}$-action on $(I_{CM}\tilde{H}^d(K^p))^{(h)}$ commutes with the $H$-action.

To prove Theorem \ref{thm-plecticactnonCM}, for any uniform pro-$p$ open subgroup $K_p$  of $\mathrm{G}(\QQ_p)$, the open subgroup $K_p^{p^n}:=\{g^{p^n}, g\in K_p\}\subseteq H$ for some $n\geq 1$ and $(I_{CM}\tilde{H}^d(K^p)_{\QQ_p})^{K_p^{p^n}-\an}\subseteq (I_{CM}\tilde{H}^d(K^p)_{\QQ_p})^{(h)}$ for some $h\geq 1$.
Therefore we can restrict the $A^{(r_h)}$-action on $(I_{CM}\tilde{H}^d(K^p)_{\QQ_p})^{(h)}$ to
\[(I_{CM}\tilde{H}^d(K^p)_{\QQ_p})^{K_p^{p^n}-\an}= ((I_{CM}\tilde{H}^d(K^p)_{\QQ_p})^{(h)})^{K_p^{p^n}-\an}.\]
(Note that $((I_{CM}\tilde{H}^d(K^p)_{\QQ_p})^{(h)})^{K_p^{p^n}-\an}$ is $A^{(r_h)}$-invariant as it commutes with $H$-action.) Using that the action of $A^{(r_h)}$ on $(I_{CM}\tilde{H}^d(K^p))^{(h)}$ is compatible with the $\mathrm{G}(\QQ_p)$-action, it's easy to see that the action of $K_p$ on $(I_{CM}\tilde{H}^d(K^p)_{\QQ_p})^{K_p^{p^n}-\an}$ is $A^{(r_h)}$-linear. Hence we can further restrict the $A^{(r_h)}$-action to $(I_{CM}\tilde{H}^d(K^p)_{\QQ_p})^{K_p-\an}$. The uniqueness is clear from the construction.
\end{proof}

We end this subsection with a result that will be used in the next subsection.
\begin{thm} \label{thm-U(h)factorR}
Same setup as in Theorem \ref{thm-plecticactnonCM}. We denote by $A'$ the image of $A^{(r)}$ in $\End((I_{CM}\tilde{H}^d(K^p)_{\QQ_p})^{K_p-\an})$.
There is a homomorphism
\[\phi:U(\mathfrak{h})\to A'\widehat\otimes_{\QQ_p}\CC_p\]
satisfying the following
\begin{enumerate}
\item The action of $U(\mathfrak{h})$ in Theorem \ref{thm-Sendiag01} on $(I_{CM}\tilde{H}^d(K^p)_{\QQ_p})^{K_p-\an}\widehat\otimes_{\QQ_p}\CC_p$  factors through $\phi$. In particular, $\phi$ maps $\diag(0,1)$ to the Sen operator.
\item $\phi|_{Z(U(\mathfrak{g}))}$ factors through $R_{\bar{D}}^{(r)}\widehat\otimes_{\QQ_p}\CC_p$, i.e. it can be written as the composite
\[Z(U(\mathfrak{g})) \xrightarrow{\phi_Z} R_{\bar{D}}^{(r)}\widehat\otimes_{\QQ_p}\CC_p \to  A^{(r)}\widehat\otimes_{\QQ_p}\CC_p\to  A'\widehat\otimes_{\QQ_p}\CC_p\]
for a homomorphism $\phi_Z:Z(U(\mathfrak{g}))\to R_{\bar{D}}^{(r)}\widehat\otimes_{\QQ_p}\CC_p$ such that for any continuous $f:R_{\bar{D}}^{(r)}\to\CC_p$, if $k_\tau,k'_\tau\in\CC_p$ denote the $\tau$-Hodge-Tate-Sen weights of $\rho_f$, the infinitesimal character $Z(U(\mathfrak{g})) \xrightarrow{\phi_Z} R_{\bar{D}}^{(r)}\xrightarrow{f\otimes 1} \CC_p$ corresponds to the $W$-orbit containing $(k_\tau+1,k'_\tau)$ (and $(k'_\tau+1,k_\tau)$ as well).
\end{enumerate}
\end{thm}

\begin{proof}
For each $\tau:F\to\CC_p$, we introduced the partial Sen operator $\theta^{(r)}_{\tau}\in A^{(r)}\widehat\otimes_{\QQ_p}\CC_p$ (cf. Proposition \ref{prop-partialSen}) and (partial) Sen polynomial $\chi_\tau^{(r)}(t):=t^2-b_{\tau}t +c_\tau\in R_{\bar{D}}^{(r)}\widehat\otimes_{\QQ_p}\CC_p[t]$ in Definition \ref{defn-partSenpoly}. On the other hand,
$Z(U(\mathfrak{g}))=U(\mathfrak{h})^W$ via $\gamma^{\mathfrak{m}}$. Explicitly $\mathfrak{h}=\prod_{\tau} \mathfrak{h}_\tau$ and each $\mathfrak{h}_\tau$ consists of $2\times 2$ diagonal matrices. Let $x_\tau, y_\tau\in\mathfrak{h}_\tau$ denote the elements $\diag(1,0)$ and $\diag(0,1)$ respectively. Then $Z(U(\mathfrak{g}))$ is a polynomial ring of generators by $x_\tau+y_\tau$ and $(x_\tau-y_\tau-1)^2$, $\tau:F\to\CC_p$.  Define 
\[\phi_Z:Z(U(\mathfrak{g}))\to R_{\bar{D}}^{(r)}\widehat\otimes_{\QQ_p}\CC_p\]
by sending $x_\tau+y_\tau$ to $b_\tau+1$ and sending $(x_\tau-y_\tau-1)^2$ to $b_\tau^2-4c_\tau$. It is easy to check that $\phi_Z$ has the desired property.

We claim that $\phi_Z$ extends to a homomorphism $\phi:U(\mathfrak{h})\to A'\widehat\otimes_{\QQ_p}\CC_p$ sending $y_\tau$ to $\theta^{(r)}_{\tau}$. Note that $U(\mathfrak{h}_\tau)$ is generated by $Z(U(\mathfrak{g}_\tau))$ and $y_\tau$ with the quadratic relation  $t^2-(x_\tau+y_\tau-1)t+(x_\tau+y_\tau-1)^2/4-(x_\tau-y_\tau-1)^2/4=0$ for $t=y_\tau$.  It's enough to check this quadratic relation on a dense subspace, more precisely, we can argue as in the proof of Theorem \ref{thm-plecticactnonCM} and check it on 
\[X=\Hom_{E[K_p]}(V_{\underline{k}}\otimes (\psi_p\circ\det), \tilde{H}^d(K^p)_{E})\]
as in Theorem \ref{thm-eigenspace}.
By a simple calculation,
the image of this quadratic relation under $\phi_Z$ is $\chi_\tau^{(r)}(t)$, hence is satisfied by $\theta_\tau^\la$ by Theorem \ref{thm-eigenspace}. 
Moreover a direct calculation shows that the action of $Z(U(\mathfrak{g}))$ on $X$ via $\phi_Z$ agrees with the natural action coming from the derivation of the $\mathrm{G}(\QQ_p)$-action. This finishes the proof.
\end{proof}

\begin{rmk}
A more careful study shows that $\phi_Z$ takes values in $R_{\bar{D}}^{(r)}$ when restricted to $Z(U(\Lie\mathrm{G}(\QQ_p)))$. This was first proved by Dospinescu-Paskunas-Schraen in a much more general setup (with essentially the same argument) \cite{DPS-2025}.
\end{rmk}

\subsection{Proof of Theorem \ref{thm-belmidnsi}}
Our goal is to show that $I_{nsi}^n$ annihilates $\tilde{H}^{<d}(K^p)_{\QQ_p}$ for some $n> 0$. The strategy is as follows. Using the Poincar\'e duality of the completed cohomology, we may extend the plectic symmetry $A^{(r)}$ we constructed to (some part of) the locally analytic completed cohomology below middle degree. By a result proved in the joint work of the second-name author with Kai-Wen Lan, there are some explicit elements in $A^{(r)}$ to annihilate $\tilde{H}^{<d}(K^p)^{\la}$. Finally we will show that the ideal generated by these elements contain a power of $I_{nsi}$.

Throughout this subsection we fix $K_p\subseteq \mathrm{G}(\QQ_p)$ open such that $\widetilde{K_p}$ is pro-$p$ uniform.  The completed cohomology with compact support with tame level $K^p$ is defined as
\[
\tilde{H}_c^i(K^p):=\varprojlim_n\varinjlim_{K'_p\subseteq \mathrm{G}(\QQ_p)} H_c^i(\Sh_{K^pK'_p}(\CC),\ZZ/p^n).
\]
There are natural Hecke and Galois actions on $\tilde{H}_c^i(K^p)$.
Let $\Lambda:=\ZZ_p[[\widetilde{K_p}]]$. This is a left and right Noetherian integral domain.
Let 
\[\tilde{H}_{i,\QQ_p}:=\Hom_{\ZZ_p}(\tilde{H}^i(K^p),\QQ_p)\]
\[\tilde{H}^{BM}_{i,\QQ_p}:=\Hom_{\ZZ_p}(\tilde{H}_c^i(K^p),\QQ_p)\]
be the rational completed homology and its Borel-Moore variant (cf. \cite{CE-2012}). Both are finitely generated left modules of $\Lambda[1/p]$ and have natural actions of $G_{\QQ}$ when computed using the \'etale cohomology.

\begin{prop} \label{prop-Poincaredualss}
There is a  $G_{\QQ}\times\Lambda[1/p]$-equivariant spectral sequence (called the Poincar\'e duality spectral sequence)
\begin{eqnarray} \label{eqn-Poincare}
E_2^{i,j}=\Ext^i_{\Lambda[1/p]}(\tilde{H}_{j,\QQ_p}, \Lambda[1/p])(d)\Longrightarrow \tilde{H}^{BM}_{2d-i-j,\QQ_p}
\end{eqnarray}
where $(d)$ denotes the $d$-th Tate twist. Here we compute $\Ext^i_{\Lambda[1/p]}$ using the left $\Lambda[1/p]$-module structure on $\Lambda[1/p]$, leaving a right $\Lambda[1/p]$-module structure on $E_2^{i,j}$, then we convert this to a left-module structure via the canonical anti-involution of $\ZZ_p[[\widetilde{K_p}]]$ induced by $g$ to $g^{-1}$.  The spectral sequence is Hecke equivariant with respect to the transpose of the Hecke action on $\Ext^i_{\Lambda[1/p]}(\tilde{H}_{j,\QQ_p}, \Lambda[1/p])(d)$ and the Hecke action on $\tilde{H}^{BM}_{2d-i-j,\QQ_p}$ composed by the involution $\iota_\TT:\TT\to \TT$ induced by $[KgK]\mapsto [Kg^{-1}K]$.
\end{prop}

\begin{proof}
This is \cite[\S 1.3]{CE-2012} except for its compatibility with the Galois and Hecke action. We give a sketch of the construction here. Given an open subgroup $K'_p$ of $K_p$, the Poincar\'e pairing on $\Sh_{K^pK'_p}$ induces a natural isomorphism
\[R\Gamma(X,\ZZ_p(d))\cong R\Hom_{\ZZ_p}(R\Gamma_c(X,\ZZ_p(d)),\ZZ_p(d))[2d]\]
where $X=\Sh_{K^pK'_p,\overline{\QQ}}$. Since $X$ is an \'etale $G:=\widetilde{K_p}/\widetilde{K'_p}$-cover of $\Sh_{K^pK_p,\overline{\QQ}}$, we have 
\begin{eqnarray*}
R\Hom_{\ZZ_p}(R\Gamma_c(X,\ZZ_p(d)),\ZZ_p(d))&\cong& (R\Hom_{\ZZ_p}(R\Gamma_c(X,\ZZ_p(d)),\ZZ_p(d))\otimes_{\ZZ_p} \ZZ_p[G])^G\\
&=&  R\Hom_{\ZZ_p[G]}(R\Gamma_c(X,\ZZ_p(d)),\ZZ_p[G](d)).
\end{eqnarray*}
Hence
\[R\Gamma(X,\ZZ_p(d))\cong R\Hom_{\ZZ_p[G]}(R\Gamma_c(X,\ZZ_p(d)),\ZZ_p[G](d))[2d]. \]
Note that since $X$ is smooth, $R^i\Gamma(X,\ZZ_p(d))$ is the $(2d-i)$-th \'etale homology of $X$ and similarly $R\Gamma_c(X,\ZZ_p(d))$ is essentially the Borel-Moore homology.
The desired spectral sequence is obtained by taking inverse limits over $K'_p$ via the trace map. See \cite[\S 1.1]{CE-2012} for the construction of the completed homology and its Borel-Moore variant. 

It is clear from the construction that the spectral sequence is Galois-equivariant. For the Hecke action, see \cite[Prop.3.7]{NT-2016}.
\end{proof}

\begin{rmk} \label{rmk-iotaT}
In Theorem \ref{thm-big T}, there is  a two-dimensional determinant $D:\ZZ_p[[G_{F,S}]]\to \TT$ such that the geometric Frobenius $\Fr_v$ at $v$ has characteristic polynomial $D(t-\Fr_v)=t^2-T_v t + \Nm(v) S_v$  for $v\notin S$. If we compose it with the involution  $\iota_{\TT}$, the resulting determinant  $D\circ\iota_{\TT}(t-\Fr_v)=t^2-T_vS_v^{-1} t + \Nm(v) S^{-1}_v$. In terms of two-dimensional Galois representation, $\iota_{\TT}$ sends $\rho$ to $\rho^{\vee}(1)$. It follows that the ideal $I_{nsi,\TT}$ is $\iota_{\TT}$-stable, cf. Definition  \ref{dfn-idealInsi}.
\end{rmk}

The proof shows that the integral version of this spectral sequence holds as well. But the following result is only known after inverting $p$.

\begin{prop} \label{prop-GK<d}
$\Ext^i_{\Lambda[1/p]}(\tilde{H}_{j,\QQ_p}, \Lambda[1/p])=0$ if $i< 3(d-j)$.
\end{prop}

\begin{proof}
This is \cite[Th.3.1.8]{Pan-2025cchmd}. See also 3.1.1 in the reference.
\end{proof}

We need the following vanishing result first proved rationally by Rodriguez Camargo in \cite{RC-2025} and later integrally by He \cite{He-2025}.
\begin{prop} \label{prop-Van>d}
$\tilde{H}_{>d,\QQ_p}=0$.
\end{prop}


For $h\geq 1$, recall that in the beginning of Section \ref{subsec-density}, we introduced $\mathrm{LA}^{(h)}(\widetilde{K_p})$ the space of $h$-analytic functions on $\widetilde{K_p}$ and its topological dual
\[D_h:=D_{<p^{-r_h}}(\widetilde{K_p},\QQ_p)=(\mathrm{LA}^{(h)}(\widetilde{K_p}))^*,\]
which has a lattice \( D_{h}^{\circ}:=\Hom(\mrm{LA}^{(h)}(\widetilde{K_{p}})^{\circ},\ZZ_{p}) \). Following Definition \ref{dfn-weakTopology}, we equip \( D_{h}^{\circ} \) with the weak topology, and \( D_{h}=D_{h}^{\circ}[1/p] \) with the colimit topology. 

 By the results in \cite[\S 4]{ST-2003}, $D_h$ is a left and right Noetherian integral domain and flat over $\Lambda$.  It also follows from the discussion in the reference that the anti-involution of $\Lambda$ induced by $g\mapsto g^{-1}$ naturally extends to $D_h$. Hence for a finitely generated $D_h$-module $M$, we get a left $D_h$-module structure on 
\[\Ext^i_{D_h}(M, D_h)\]
 as in Proposition \ref{prop-Poincaredualss}. We will need the following technical result whose proof will be given at the end of this subsection.
 
\begin{prop} \label{prop-ext vs Hcont}
Let $W$ be an admissible Banach space representation of $\widetilde{K_p}$, equivalently the topological dual $W^*$ of $W$ is a finitely generated $\Lambda[1/p]$-module. 
Let $V$ be a Banach space representation of $\widetilde{K_p}$ so that  $V^*$ is a finitely generated $D_h$-module extending the $\widetilde{K_p}$-action. 
\begin{enumerate}
\item There is a natural isomorphism functorial in $W$:
\[\Ext^i_{\Lambda[1/p]}(W^*,\Lambda[1/p])\otimes_{\Lambda[1/p]} D_h \cong \Ext^i_{D_h}(D_h\otimes_{\Lambda[1/p]}W^*, D_h)\]
\item There is a natural isomorphism functorial in $V$:
\[  \Ext^i_{D_h}(V^*, D_h)\cong H^i_{\mathrm{cont}}(\widetilde{K_p}, V\widehat\otimes_{\QQ_p} D_h)\]
where the completed tensor product is taken with respect to the $p$-adic topology on $V$ and the weak topology on $D_h$ and $\widetilde{K_p}$ acts diagonally on it. Equivalently $V\widehat\otimes_{\QQ_p} D_h=\Hom_{\mathrm{cont}}(\mathrm{LA}^{(h)}(\widetilde{K_p}), V)$. Therefore $H^i_{\mathrm{cont}}(\widetilde{K_p}, V\widehat\otimes_{\QQ_p} D_h)$ is a finitely generated $D_h$-module. By writing it as a quotient of $D_h^{\oplus n}$ for some $n$, we endow $H^i_{\mathrm{cont}}(\widetilde{K_p}, V\widehat\otimes_{\QQ_p} D_h)$ with the quotient topology which doesn't depend on the choice of the quotient map $D_h^{\oplus n}\to H^i_{\mathrm{cont}}(\widetilde{K_p}, V\widehat\otimes_{\QQ_p} D_h)$. We call this the canonical topology.
\item There is a natural isomorphism
\[H^i_{\mathrm{cont}}(\widetilde{K_p}, V\widehat\otimes_{\QQ_p} \CC_p \widehat\otimes_{\QQ_p} D_h)\cong H^i_{\mathrm{cont}}(\widetilde{K_p}, V\widehat\otimes_{\QQ_p} D_h)\widehat\otimes_{\QQ_p} \CC_p.\]
\end{enumerate}
\end{prop}

\begin{cor}\label{cor-Poincare-hC}
There is a  $G_{\QQ_p}\times D_h$ and Hecke-equivariant spectral sequence 
\[E_2^{i,j}=H^i_{\mathrm{cont}}(\widetilde{K_p}, \tilde{H}^{j,(h)} \widehat\otimes_{\QQ_p} \CC_p \widehat\otimes_{\QQ_p} D_h)(d)\Longrightarrow (\tilde{H}^{2d-i-j,(h)}_c)^*\widehat\otimes_{\QQ_p} \CC_p\]
with $E_2^{>d,j}=0$ and $E_2^{i,d-j}=0$ if $i<3j$.
\end{cor}

\begin{proof}
The spectral sequence is obtained by first taking the tensor product of the Poincar\'e duality spectral sequence \eqref{eqn-Poincare} with $D_h$ over $\Lambda[1/p]$ and then taking the completed tensor product with $\CC_p$. Here we use that $D_h$ is flat over $\Lambda[1/p]$ and every term appearing in the Poincar\'e duality spectral sequence is finite over $\Lambda[1/p]$, and hence all the morphisms in the tensor product of it with $D_h$ are strict with respect to the canonical topology. Thus the completed tensor product with $\CC_p$ is exact as well. Finally we apply the formula \eqref{eqn-h-analytic} and Proposition \ref{prop-ext vs Hcont} to write the resulting spectral sequence as in the claim. The vanishing result follows from Propositions \ref{prop-GK<d} and \ref{prop-Van>d}.
\end{proof}

The $G_{\QQ_p}$-action induces a Sen operator acting on $\tilde{H}^{j,(h)} \widehat\otimes_{\QQ_p} \CC_p(d)$, hence a Sen operator $\theta^{i,j}$ on $H^i_{\mathrm{cont}}(\widetilde{K_p}, \tilde{H}^{j,(h)} \widehat\otimes_{\QQ_p} \CC_p \widehat\otimes_{\QQ_p} D_h)(d)$. Similarly by writing  $(\tilde{H}^{2d-i-j,(h)}_c)^*\widehat\otimes_{\QQ_p} \CC_p$ as $\Hom_{\CC_p,\mathrm{cont}}(\tilde{H}^{2d-i-j,(h)}_c \widehat\otimes_{\QQ_p} \CC_p,\CC_p)$, we obtain a Sen operator $\theta^{2d-i-j}$ on it from the Sen operator on $\tilde{H}^{2d-i-j,(h)}_c \widehat\otimes_{\QQ_p} \CC_p$, whose existence can be shown by the same argument as in \cite[\S 6]{Pan-2025JEMS}.
\begin{prop} \label{prop-Seneq}
The spectral sequence in Corollary \ref{cor-Poincare-hC} is equivariant with respect to the actions of $\theta^{i,j}$ on $E_2^{i,j}$ and $-\theta^{2d-i-j}$ on $ (\tilde{H}^{2d-i-j,(h)}_c)^*\widehat\otimes_{\QQ_p} \CC_p$.
\end{prop}

\begin{proof}
First observe that there is another Sen operator on $N:=H^i_{\mathrm{cont}}(\widetilde{K_p}, \tilde{H}^{j,(h)} \widehat\otimes_{\QQ_p} \CC_p \widehat\otimes_{\QQ_p} D_h)$: using Proposition \ref{prop-ext vs Hcont}, we can rewrite it as 
\[\Ext^i_{D_h}(D_h\otimes_{\Lambda[1/p]} \tilde{H}_{j,\QQ_p}, D_h)\widehat\otimes_{\QQ_p}\CC_p\cong
\Hom_{\CC_p,\mathrm{cont}}\left(\Ext^i_{D_h}(D_h\otimes_{\Lambda[1/p]} \tilde{H}_{j,\QQ_p}, D_h)^*\widehat\otimes_{\QQ_p}\CC_p,\CC_p\right).\]
The same argument as in \cite[\S 6]{Pan-2025JEMS} shows that we can apply Sen's result in Theorem \ref{thm-Sen-inf} to the $G_{\QQ_p}$-representation $\Ext^i_{D_h}(D_h\otimes_{\Lambda[1/p]} \tilde{H}_{j,\QQ_p}, D_h)^*$ and obtain a Sen operator on $\Ext^i_{D_h}(D_h\otimes_{\Lambda[1/p]} \tilde{H}_{j,\QQ_p}, D_h)^*\widehat\otimes_{\QQ_p}\CC_p$, hence on its dual $N$. We denote it by  $\theta'$. Since the spectral sequence in Corollary \ref{cor-Poincare-hC} is $G_{\QQ_p}$-equivariant, it's easy to see that it's equivariant with respect to $\theta'$ and $\theta^{2d-i-j}$. Hence it suffices to prove that  $-\theta'$  agrees with $\theta^{i,j}$. Roughly speaking this means that the construction of the Sen operator here is compatible with taking $\Ext^i_{D_h}(\bullet^*,D_h)^*$ up to a sign.

To see this we need to recall the following characterization of the Sen operator. See the discussion below Theorem \ref{thm-Sen-inf}  for the notation. Let  $V=\tilde{H}^{j,(h)}$. There exists a finite extension $K$ of $\QQ_p$ such that 
\[V\widehat\otimes_{\QQ_p}\CC_p\cong D(V) \widehat\otimes_{K}\CC_p\]
where $D(V)=(V\widehat\otimes_{\QQ_p}\CC_p)^{H_K,\Gamma_0-\an}$ is a $\QQ_p$-Banach space representation of $\widetilde{K_p}$. The Sen operator on $V\widehat\otimes_{\QQ_p}\CC_p$  is $\CC_p$-linear and acts on $D(V)$ via the canonical generator of $\Lie\Gamma_0$.

\begin{lem}
$D(V)^*$  is a finitely generated $D_h$-module.
\end{lem}
\begin{proof}
By the construction of $D_h$ in \cite[\S IV.3]{CD-2014}, there is a natural valuation $v^{(h)}$ on $D_h$. Let $D_h^+$ be the subring of elements of non-negative valuations. It is a complete local ring with maximal ideal $J^{(h)}$ generated by $p$ and the intersection of $D_h^+$ and the augmentation ideal of $\QQ_p[\widetilde{K_p}]$. The residue field $D_h^+/ J^{(h)}=\FF_p$. Equivalently, under the duality $D_h=\mathrm{LA}^{(h)}(\widetilde{K_p})^*$, $D_h^+$ are the functionals on $\mathrm{LA}^{(h)}(\widetilde{K_p})$ of norm $\leq 1$, and $J^{(h)}$ is generated by $p$ and functionals whose kernel contain the constant functions in $\mathrm{LA}^{(h)}(\widetilde{K_p})$. Let $D(V)^+\subseteq D(V)$ be the unit ball. By the complete Nakayama's lemma, it suffices to prove that \( D(V)^{+,*} \otimes_{D_{h}^{+}}D_{h}^{+}/J^{(h)} \) is a finite-dimensional $\FF_p$-vector space, or equivalently, its $\FF_p$-linear dual space
\[\Hom_{D_h^+}(D_h^+/J^{(h)}, D(V)^+/p)\]
is finite.

Let $V^+$ denote the unit ball of $V$. It follows from the admissibility of the completed cohomology that $V^*$ is a finitely generated $D_h$-module and hence $\Hom_{D_h^+}(D_h^+/J^{(h)}, V^+/p)$ is finite-dimensional. On the other hand, by the open mapping theorem, for some $n\geq 0$,
\[p^n V^+\widehat\otimes_{\ZZ_p}\mO_{\CC_p} \subseteq D(V)^+\widehat\otimes_{\mO_K}\mO_{\CC_p} \subseteq 
V^+\widehat\otimes_{\ZZ_p}\mO_{\CC_p} .\]
Thus $\Hom_{D_h^+}(D_h^+/J^{(h)}, D(V)^+/p)\otimes_{\mO_K} \mO_{\CC_p}$ is filtered by $\mO_{\CC_p}$-submodules of 
\[\Hom_{D_h^+}(D_h^+/J^{(h)}, p^rV^+/p^{r+1}V^+)\otimes_{\FF_p} \mO_{\CC_p},\,r=0,\cdots,n.\]
Now we view all of these as almost $\mO_{\CC_p}$-modules with respect to the maximal ideal of $\mO_{\CC_p}$, then $\Hom_{D_h^+}(D_h^+/J^{(h)}, p^rV^+/p^{r+1}V^+)\otimes_{\FF_p} \mO_{\CC_p}$ is an almost finitely generated $\mO_{\CC_p}/p$-module. By \cite[Prop.2.7]{Scholze-2013}, the category of almost finitely generated $\mO_{\CC_p}$-modules is abelian. Hence $\Hom_{D_h^+}(D_h^+/J^{(h)}, D(V)^+/p)\otimes_{\mO_K} \mO_{\CC_p}$ is almost finitely generated. It follows that $\Hom_{D_h^+}(D_h^+/J^{(h)}, D(V)^+/p)$ is finite-dimensional.
\end{proof}

By Proposition \ref{prop-ext vs Hcont},
\begin{eqnarray*}
H^i_{\mathrm{cont}}(\widetilde{K_p}, \tilde{H}^{j,(h)} \widehat\otimes_{\QQ_p} \CC_p \widehat\otimes_{\QQ_p} D_h)=
H^i_{\mathrm{cont}}(\widetilde{K_p}, D(V) \widehat\otimes_{K} \CC_p \widehat\otimes_{\QQ_p} D_h)\\
\cong
H^i_{\mathrm{cont}}(\widetilde{K_p}, D(V) \widehat\otimes_{\QQ_p} D_h) \widehat\otimes_{K} \CC_p 
\cong \Ext^i_{D_h}(D(V)^*, D_h) \widehat\otimes_{K} \CC_p.
\end{eqnarray*}
Taking the continuous $\CC_p$-linear dual, we get an isomorphism of  $G_{\QQ_p}$-semilinear representations
\[\Ext^i_{D_h}(D(V)^*, D_h)^* \widehat\otimes_{K} \CC_p \cong \Ext^i_{D_h}(D_h\otimes_{\Lambda[1/p]} \tilde{H}_{j,\QQ_p}, D_h)^*\widehat\otimes_{\QQ_p}\CC_p.\]
Since $\Ext^i_{D_h}(D(V)^*, D_h)^*$ is fixed by $H_K$  and a $\Gamma_0$-analytic representation, the Sen operator on the right hand side agrees with the action of the canonical generator of $\Lie\Gamma_0$ acting on $D(V)$, hence differs from the action $\theta^{i,j}$ by a sign. (Note that each ${}^*$ and each $\Ext^i$ contribute a sign.)
\end{proof}

Next we produce some operators annihilating the locally analytic completed cohomology below middle degree. Recall that as in Theorem \ref{thm-Sendiag01}, there is a natural action of $U(\mathfrak{h})$ on $\tilde{H}^i(K^p)^{\la}\widehat\otimes_{\QQ_p} \CC_p$ extending the action of $Z(U(\mathfrak{g}))$ via $\gamma^{\mathfrak{m}}$, and the action of $\diag(0,1)\in \mathfrak{h}$ agrees with the Sen operator on $\tilde{H}^i(K^p)^{\la}\widehat\otimes_{\QQ_p} \CC_p$.

\begin{dfn}
We define $Z[\theta_{Sen}]$ as the subalgebra of $U(\mathfrak{h})$ generated by $Z(U(\mathfrak{g}))$ and $\diag(0,1)\in \mathfrak{h}$.
\end{dfn}

Recall that in the proof of Theorem \ref{thm-U(h)factorR}, we chose explicit coordinates $x_\tau,y_\tau$ of $\mathfrak{h}_\tau$ such that $Z(U(\mathfrak{g}))$ is generated by $x_\tau+y_\tau$ and $(x_\tau-y_\tau-1)^2$, and $\theta_{Sen}=\sum_\tau y_\tau$.

\begin{dfn}
We denote by $\mathcal{I}_\mathfrak{h}\subseteq U(\mathfrak{h})$ the ideal generated by $\prod_{\tau} (x_\tau-y_\tau)$ and denote by $\mathcal{I}_{Sen}$ its intersection with $Z[\theta_{Sen}]$.
\end{dfn}

The ideal $\mathcal{I}_\mathfrak{h}$ is the ideal $\mathcal{I}_\mathfrak{m}^\iota$ introduced in \cite[Def.5.1.8]{Lan-Pan} in the case of Hilbert modular varieties which parametrizes the so-called ``non-$\mu_h^\iota$-sufficiently regular characters'', cf. Definition 5.1.5 \textit{loc. it.}. (Note in this case $\mathfrak{m}$ in the reference agrees with $\mathfrak{h}$ here). See Example 5.1.7 \textit{loc. it.} for the explicit description in the Hilbert case. 

\begin{thm}
The action of $\mathcal{I}_\mathfrak{h}^n$  annihilates $\tilde{H}^{<d}(K^p)^{\la}\widehat\otimes_{\QQ_p} \CC_p$ for some $n>0$.
\end{thm}

\begin{proof}
 \cite[Th.6.1.3]{Lan-Pan}.
\end{proof}

\begin{cor} \label{cor-LPvan}
$\mathcal{I}_{Sen}^n$  annihilates $\tilde{H}^{<d}(K^p)^{\la}\widehat\otimes_{\QQ_p} \CC_p$ for some $n>0$.
\end{cor}

The relationship between $\mathcal{I}_{Sen}$ and the ideal $I_{nsi}$ of $R^{(r)}_{\bar{D}}$ in Definition \ref{dfn-idealInsi} is as follows. We first explain a version that is simpler to state but will not be used, then we give the actual technical result. Recall that in (the proof of) Theorem \ref{thm-U(h)factorR}, we defined a quotient $A'$ of $A^{(r)}$ and a homomorphism $\phi:U(\mathfrak{h})\to A'\widehat\otimes_{\QQ_p}\CC_p$.
 
 \begin{prop} \label{prop-A'IsenInsi}
 Let  $f:R^{(r)}_{\bar{D}}\widehat\otimes_{\QQ_p}\CC_p\to \CC_p$ be a homomorphism of $\CC_p$-algebras such that $\rho_f:G_{F,S}\to \GL_2(\CC_p)$ is strongly irreducible and $(A'\widehat\otimes_{\QQ_p}\CC_p)/\ker f\neq 0$. For $n>0$, the image of $\phi(\mathcal{I}_{Sen}^n)$ in $(A'\widehat\otimes_{\QQ_p}\CC_p)/\ker f$ is non-zero. 
 \end{prop}

\begin{rmk} 
In fact, the proof below shows that
for $n>0$, the two-sided ideal of $A'\widehat\otimes_{\QQ_p}\CC_p$ generated by $\phi(\mathcal{I}_{Sen}^n)$ contains a power of $I_{nsi} A'$. 
\end{rmk}

\begin{proof}
\begin{lem} \label{lem-eqHTwIsen}
If $\rho_f$ has equal $\tau$-Hodge-Tate-Sen weights for any $\tau$, the maximal ideal $\ker f\cap Z(U(\mathfrak{g}))$ does not contain $\mathcal{I}_\mathfrak{h}\cap Z(U(\mathfrak{g}))=\mathcal{I}_{Sen}\cap Z(U(\mathfrak{g}))$.
\end{lem}

\begin{proof}
This follows from the explicit description of $\ker f\cap Z(U(\mathfrak{g}))$ in Theorem \ref{thm-U(h)factorR}.
\end{proof}

As a consequence of this lemma, $(\ker f,\phi(\mathcal{I}_{Sen}^n))=A'\widehat\otimes_{\QQ_p}\CC_p$ if $\rho_f$ has equal $\tau$-Hodge-Tate-Sen weights for any $\tau$. Our result is clear in this case.

Next we assume that $\rho_f:G_{F,S}\to \GL_2(\CC_p)$ is strongly irreducible and  $\Lie \rho_f(I_{F_v})\subseteq \mathfrak{gl}_2(\CC_p)$ are not scalars for at least one $p$-adic place $v$ of $F$. 

\begin{lem}
For any maximal ideal $\mathfrak{p}$ of $Z(U(\mathfrak{g}))$, there exists a maximal ideal of $\mathfrak{p}'$ of $Z[\theta_{Sen}]$ containing $\mathfrak{p}$ but not containing $\mathcal{I}_{Sen}$.
\end{lem}

\begin{proof}
The argument is almost the same as \cite[Proof of Proposition 5.1.10]{Lan-Pan}. By Hilbert's Nullstellensatz, a maximal ideal $\mathfrak{p}$ of $Z(U(\mathfrak{g}))$ is the same as a homomorphism $\lambda:Z(U(\mathfrak{g}))\to \CC_p$, and a maximal ideal of $U(\mathfrak{h})$ containing it is the same as a weight $\{(X_\tau,Y_\tau)\}_{\tau}$ in the corresponding $W$-orbit $[\lambda]$ of $\lambda$. Similarly a  maximal ideal of  $Z[\theta_{Sen}]$ containing $\mathfrak{p}$ is the same as a pair $(\lambda, w)$ with $w=\sum_{\tau} Y_\tau$ for some $\{(X_\tau,Y_\tau)\}_{\tau}\in [\lambda]$. Our goal is to produce such a pair $(\lambda, w)$ so that if $w=\sum_{\tau} Y_\tau$ with $\{(X_\tau,Y_\tau)\}_{\tau}\in [\lambda]$, then all $(X_\tau-Y_\tau)$'s are non-zero. 

By choosing a $\QQ$-basis of the $\QQ$-vector space in $\CC_p$ generated by the coordinates of $\lambda$, we may reduce the problem to the case where all the coordinates are in $\QQ$. Now fixing \( \lambda \), we choose $w$ to be the smallest possible value and claim that it satisfies our requirement. Suppose that $w=\sum_{\tau} Y_\tau$ for some $\{(X_\tau,Y_\tau)\}_{\tau}\in [\lambda]$ and $X_\tau-Y_\tau=0$ for some $\tau$. Since $(X_\tau,Y_\tau)$ and $(Y_\tau+1,X_\tau-1)$ are in the same $W$-orbit, after changing $(X_\tau,Y_\tau)$ by $(Y_\tau+1,X_\tau-1)$ and keeping other components the same, we see that the resulting $w$ has strictly smaller value. Contradiction.
\end{proof}

In the proof of Proposition \ref{prop-suppA(r)}, we showed that
\[B:=(A'\widehat\otimes_{\QQ_p}\CC_p)/\ker f\cong \End_{\CC_p}(\Ind_F^{plec}\rho_f)=\End_{\CC_p}(\oInd^{\QQ}_{F} \rho_f).\]
Since $\phi:U(\mathfrak{h})\to A'\widehat\otimes_{\QQ_p} \CC_p$ maps $\diag(0,1)$ to the Sen operator, as a $Z[\theta_{Sen}]$-module, the support of $B$ are the pairs $(\ker f\cap Z(U(\mathfrak{g})), w)$, where $w$ ranges \textit{all} the Hodge-Tate-Sen weights of $\oInd^{\QQ}_{F} \rho_f$. Equivalently, $B$ has full supports on $\Spec Z[\theta_{Sen}]/(\ker f\cap Z(U(\mathfrak{g})))$. Our claim now follows from the previous lemma.
\end{proof}

To state the more technical version, we fix $f:R^{(r)}_{\bar{D}}\widehat\otimes_{\QQ_p}\CC_p\to \CC_p$ such that $\rho_f:G_{F,S}\to \GL_2(\CC_p)$ is strongly irreducible. By Proposition \ref{proptenindirred}, $\oInd^{\QQ}_{F} \rho_f$  has a unique $G_{\QQ}$-subrepresentation $V$ of the largest dimension among all irreducible $\Lie(\oInd^{\QQ}_{F} \rho_f(G_{\QQ}))$-subrepresentations. Let $C$ denote the image of $\CC_p[G_{\QQ}]$ in $\End_{\CC_p}(\oInd^{\QQ}_{F} \rho_f)$. Consider the following commutative diagram
\[\begin{tikzcd}
\CC_p[G_{\QQ}] \arrow[r,] \arrow[d,] &C \arrow[r,"pr_V"]  \arrow[d,] &  \End_{\CC_p}(V) \\
A'\widehat\otimes_{\QQ_p}\CC_p \arrow[r, ] & (A'\widehat\otimes_{\QQ_p}\CC_p)/\ker f=\End_{\CC_p}(\oInd^{\QQ}_{F} \rho_f) &
\end{tikzcd},\]
where $pr_V$ denotes the projection map induced by restriction to $V$.
The map $\phi|_{Z[\theta_{Sen}]}:Z[\theta_{Sen}] \to A'\widehat\otimes_{\QQ_p}\CC_p \to \End_{\CC_p}(\oInd^{\QQ}_{F} \rho_f)$ factors through $C\subseteq \End_{\CC_p}(\oInd^{\QQ}_{F} \rho_f)$. Indeed $Z(U(\mathfrak{g}))$ maps to $\CC_p$ and $\diag(0,1)$ maps to the Sen operator which is in $C$. We denote it by $\phi_C: Z[\theta_{Sen}]\to C$. 

\begin{prop} \label{prop-prVIsne}
For $n>0$,
$pr_V\circ \phi_C(\mathcal{I}_{Sen}^n)\neq 0$.
\end{prop}

\begin{proof}
The same proof of Proposition \ref{prop-A'IsenInsi} works here by noticing that $V$ has the same set of Hodge-Tate-Sen weights as that of $\oInd^{\QQ}_{F} \rho_f$ by Proposition \ref{proptenindirred}.
\end{proof}

To describe how the action of $U(\mathfrak{h})$ interacts with the Poincar\'e duality, we introduce the following involution.
\begin{dfn} \label{dfn-iotaSen}
We let 
\[\iota_\mathfrak{h}:U(\mathfrak{h})\to U(\mathfrak{h})\] 
be the involution induced by the reflection of $\mathfrak{h}^*$ about $\rho$. Explicitly it sends $x_\tau$ to $1-x_\tau$ and sends $y_\tau$ to $-1-y_\tau$. It commutes with the Weyl group action centered at $\rho$, therefore induces an involution $\iota_Z$ of $Z(U(\mathfrak{g}))$. Since $\iota_\mathfrak{h}(\diag(0,1))=-d-\diag(0,1)$, it also induces an involution of $Z[\theta_{Sen}]$ which will be denoted by $\iota_{Sen}$.
\end{dfn}

\begin{rmk} \label{rmk-iotaZ}
The isomorphism of algebras $U(\mathfrak{g})\xrightarrow{\cong} U(\mathfrak{g})^{op}$ sending $x\in\mathfrak{g}$ to $-x$ defines an involution of $Z(U(\mathfrak{g}))$ which is nothing but $\iota_Z$.
\end{rmk}

We have results similar to lemma \ref{lem-eqHTwIsen} and Proposition \ref{prop-A'IsenInsi}. The proof is essentially identical, so we omit the detail here.

\begin{lem} \label{lem-eqHTwIsen'}
Let  $f:R^{(r)}_{\bar{D}}\widehat\otimes_{\QQ_p}\CC_p\to \CC_p$ be a homomorphism of $\CC_p$-algebras
such that $\rho_f$ has equal $\tau$-Hodge-Tate-Sen weights for any $\tau$. Then $\ker f\cap Z(U(\mathfrak{g}))$ does not contain $\iota_Z(\mathcal{I}_\mathfrak{h})\cap Z(U(\mathfrak{g}))$.
\end{lem}

\begin{prop} \label{prop-A'IsenInsisim}
Let  $f:R^{(r)}_{\bar{D}}\widehat\otimes_{\QQ_p}\CC_p\to \CC_p$ be a homomorphism of $\CC_p$-algebras such that $\rho_f:G_{F,S}\to \GL_2(\CC_p)$ is strongly irreducible and $(A'\widehat\otimes_{\QQ_p}\CC_p)/\ker f\neq 0$. For $n>0$, the image of $\phi(\iota_{Sen}(\mathcal{I}_{Sen})^n)$ in $(A'\widehat\otimes_{\QQ_p}\CC_p)/\ker f$ is non-zero.
\end{prop}

Finally we can prove Theorem \ref{thm-belmidnsi}. We restate it in a slightly more general form. 

\begin{thm} \label{thm-belmidcnsi}
For $i<d$, $I_{nsi,\TT}^n$ annihilates $\tilde{H}^{i}(K^p)_{\QQ_p}$ and $\tilde{H}^{i}_{c}(K^p)_{\QQ_p}$ (equivalently, $\tilde{H}_{i,\QQ_p}$ and $\tilde{H}^{BM}_{i,\QQ_p}$) for some $n>0$.
\end{thm}

As explained in \cite[\S 1.5]{CE-2012}, by considering the Borel-Serre compactification $\overline{\mathrm{Sh}_{K^pK_p}(\CC)}$ of $\mathrm{Sh}_{K^pK_p}(\CC)$, we get a Hecke-equivariant long exact sequence
\[\cdots\to \tilde{H}^{i}_{c}(K^p)_{\QQ_p} \to \tilde{H}^{i}(K^p)_{\QQ_p}\to \tilde{H}^i(\partial)_{\QQ_p}\to\cdots\]
where $\tilde{H}^i(\partial)_{\QQ_{p}}:=\left(\varprojlim_n\varinjlim_{K'_p\subseteq \mathrm{G}(\QQ_p)} H_c^i(\overline{\Sh_{K^pK'_p}(\CC)}\setminus \Sh_{K^pK'_p}(\CC),\ZZ/p^n)\right)\otimes\QQ_p$ is the boundary (rational) completed cohomology. 

\begin{lem} \label{lem-Insikillboundary}
For $i\geq0$, $I_{nsi,\TT}$ annihilates $\tilde{H}^i(\partial)_{\QQ_{p}}$.
\end{lem}

\begin{proof}
$\mathrm{G}$ has only one proper rational parabolic subgroup up to $\mathrm{G}(\QQ)$-conjugacy. Fix such a parabolic subgroup $\mathrm{P}$ and let $\mathrm{M}$ denote a Levi subgroup. By the discussion at the end of  \cite[\S 1.5]{CE-2012}, the boundary completed cohomology $\tilde{H}_{\QQ_p}^\bullet(\partial)$, after passing to the direct limit over all $K^p$, can be expressed as an induced representation of the completed cohomology $\mathrm{M}\cong \mathrm{Res}_{F/\QQ}\mathbb{G}_m\times\mathbb{G}_m$ from $\mathrm{P}(\mA_f)$ to $\mathrm{G}(\mA_f)$. Consequently the Hecke action on $\tilde{H}_{\QQ_p}^\bullet(\partial)$ factors through the Hecke action on the completed cohomology $\mathrm{M}$ via the Satake transform from $\mathrm{G}$ to $\mathrm{M}$. We refer to the proof of Theorem 4.2 of \cite{NT-2016} and especially Lemma 4.6 for more details. (The reference considers cohomology of $\FF_p$-coefficients but the formulas there work equally for $\ZZ/p^n$.)

It follows that the associated determinant of $G_F$ on the image of $\TT$ in $\End(\tilde{H}_{\QQ_p}^\bullet(\partial))$ is reducible, hence $I_{nsi,\TT}$ annihilates $\tilde{H}^i(\partial)_{\QQ_{p}}$.
\end{proof}

\begin{proof}[Proof of Theorem \ref{thm-belmidcnsi}]
We do backward induction on $i$. Suppose that $I_{nsi,\TT}^n$ annihilates $\tilde{H}^{k}(K^p)_{\QQ_p}$ and $\tilde{H}_{c}^{k}(K^p)_{\QQ_p}$ for $i<k<d$ and some $n$. By Lemma \ref{lem-Insikillboundary}, it suffices to show that $I_{nsi,\TT}^n$ annihilates  $\tilde{H}_c^{i}(K^p)_{\QQ_p}$ for some $n$. In fact, it is enough to show that for each $h\geq 1$,  $\tilde{H}^{i}_c(K^p)_{\QQ_p}^{(h)}$ is annihilated by a power of $I_{nsi,\TT}$. Indeed, let $\tilde{H}^{i}(K^p)_{\QQ_p}[I_{nsi,\TT}^{\infty}]\subseteq \tilde{H}^{i}(K^p)_{\QQ_p}$ be the subspace of elements annihilated by a power of $I_{nsi,\TT}$. As a  $\mathrm{G}(\QQ_p)$-subrepresentation of the admissible representation $\tilde{H}^{i}(K^p)_{\QQ_p}$, it is a closed subspace by \cite{ST-2002}. On the other hand, 
\[\tilde{H}^{i}(K^p)_{\QQ_p}[I_{nsi,\TT}^{\infty}]=\bigcup_{n} \tilde{H}^{i}(K^p)_{\QQ_p}[I_{nsi,\TT}^{n}]\]
where $\tilde{H}^{i}(K^p)_{\QQ_p}[I_{nsi,\TT}^{n}]\subseteq \tilde{H}^{i}(K^p)_{\QQ_p}$ denotes the closed subspace annihilated by $I_{nsi,\TT}^n$. By the Baire category theorem, we have $\tilde{H}^{i}(K^p)_{\QQ_p}[I_{nsi,\TT}^{\infty}]= \tilde{H}^{i}(K^p)_{\QQ_p}[I_{nsi,\TT}^{n}]$ for some $n$.  If all of the locally analytic vectors of $ \tilde{H}^{i}(K^p)_{\QQ_p}$ belong to $\tilde{H}^{i}(K^p)_{\QQ_p}[I_{nsi,\TT}^{\infty}]=\tilde{H}^{i}(K^p)_{\QQ_p}[I_{nsi,\TT}^{n}]$, the density result of Schneider-Teitelbaum \cite{ST-2003} implies that $\tilde{H}^{i}(K^p)_{\QQ_p}$ is  annihilated by  $I_{nsi,\TT}^n$ as well.

Fix $h\geq 1$ in the rest of the proof. Consider the spectral sequence in Corollary \ref{cor-Poincare-hC}
\begin{eqnarray} \label{eqn-laPD}
E_2^{i,j}=H^i_{\mathrm{cont}}(\widetilde{K_p}, \tilde{H}^{j,(h)} \widehat\otimes_{\QQ_p} \CC_p \widehat\otimes_{\QQ_p} D_h)(d)\Longrightarrow (\tilde{H}^{2d-i-j,(h)}_c)^*\widehat\otimes_{\QQ_p} \CC_p
\end{eqnarray}
with $E_2^{>d,j}=0$ and $E_2^{i,d-j}=0$ if $i<3j$. The $Z[\theta_{Sen}]$-action on $\tilde{H}^{j,(h)} \widehat\otimes_{\QQ_p} \CC_p $ endows each $E_2^{i,j}$ with a $Z[\theta_{Sen}]$-module structure. It follows from Proposition \ref{prop-Seneq} that all the connecting homomorphisms in the spectral sequence are $Z[\theta_{Sen}]$-linear.

By our knowledge of the vanishing of $E_2^{i,j}$, there is a natural map (coming from the map $(\tilde{H}^{i,(h)}_c)^*\widehat\otimes_{\QQ_p} \CC_p \to E_\infty^{d-i,d}$)
\[\phi_i:(\tilde{H}^{i,(h)}_c)^*\widehat\otimes_{\QQ_p} \CC_p\to E_2^{d-i,d}=H^{d-i}_{\mathrm{cont}}(\widetilde{K_p}, \tilde{H}^{d,(h)} \widehat\otimes_{\QQ_p} \CC_p \widehat\otimes_{\QQ_p} D_h)(d) \]
such that  the kernel of $\phi_i$ is filtered by  subquotients of 
\[E_2^{d-i+j,d-j}=H^{d-i+j}_{\mathrm{cont}}(\widetilde{K_p}, \tilde{H}^{d-j,(h)} \widehat\otimes_{\QQ_p} \CC_p \widehat\otimes_{\QQ_p} D_h)(d),~0<j\leq (d-i)/2<d-i\]
and the  cokernel of $\phi_i$ is filtered by subquotients of 
\[E_2^{d-i+1+j,d-j}=H^{d-i+1+j}_{\mathrm{cont}}(\widetilde{K_p}, \tilde{H}^{d-j,(h)} \widehat\otimes_{\QQ_p} \CC_p \widehat\otimes_{\QQ_p} D_h)(d),~0<j.\]
By  Corollary \ref{cor-LPvan}, we see that

\begin{lem} \label{lem-cokerphii}
 The cokernel of $\phi_i$ is annihilated by a power of $\mathcal{I}_{Sen}$. 
 \end{lem} 
 
The induction hypothesis implies that $\ker\phi_i$ is annihilated by a power of $I_{nsi,\TT}$. By the Hecke equivariance of the spectral sequence in Proposition \ref{prop-Poincaredualss} and  Remark \ref{rmk-iotaT}, it remains to show that a power of $I_{nsi,\TT}$ annihilates  the image of $\phi_i$. This is actually true for the target of $\phi_i$ by the following proposition.
\end{proof}

\begin{prop}
$H^{d-i}_{\mathrm{cont}}(\widetilde{K_p}, \tilde{H}^{d,(h)} \widehat\otimes_{\QQ_p} \CC_p \widehat\otimes_{\QQ_p} D_h)(d)$ is annihilated by a power of $I_{nsi,\TT}$. 
\end{prop}

\begin{proof}
We localize $\phi_i$ at the maximal ideal $\mathfrak{m}$ and get
\[\phi_{i,\mathfrak{m}}:(\tilde{H}^{i,(h)}_{c,\mathfrak{m}^\vee})^*\widehat\otimes_{\QQ_p} \CC_p\to H^{d-i}_{\mathrm{cont}}(\widetilde{K_p}, \tilde{H}^{d,(h)}_{\mathfrak{m}} \widehat\otimes_{\QQ_p} \CC_p \widehat\otimes_{\QQ_p} D_h)(d) \]
by Proposition \ref{prop-Poincaredualss}, where $\mathfrak{m}^\vee=\iota_{\TT}(\mathfrak{m})$. Consider the natural map 
\[(\tilde{H}^{i,(h)}_{\mathfrak{m}^\vee})^*\widehat\otimes_{\QQ_p} \CC_p\to (\tilde{H}^{i,(h)}_{c,\mathfrak{m}^\vee})^*\widehat\otimes_{\QQ_p} \CC_p\]
whose  cokernel is annihilated by $I_{nsi}$ by Lemma \ref{lem-Insikillboundary} and the composite map
\[\phi':(\tilde{H}^{i,(h)}_{\mathfrak{m}^\vee})^*\widehat\otimes_{\QQ_p} \CC_p\to (\tilde{H}^{i,(h)}_{c,\mathfrak{m}^\vee})^*\widehat\otimes_{\QQ_p} \CC_p
\xrightarrow{\phi_{i,\mathfrak{m}}} H^{d-i}_{\mathrm{cont}}(\widetilde{K_p}, \tilde{H}^{d,(h)}_{\mathfrak{m}} \widehat\otimes_{\QQ_p} \CC_p \widehat\otimes_{\QQ_p} D_h)(d).\]
We endow the source of $\phi'$ with a $Z[\theta_{Sen}]$-module structure as follows: there is a natural isomorphism
\[(\tilde{H}^{i,(h)}_{\mathfrak{m}^\vee})^*\widehat\otimes_{\QQ_p} \CC_p\cong \Hom_{\CC_p,\mathrm{cont}}(\tilde{H}^{i,(h)}_{\mathfrak{m}^\vee}\widehat\otimes_{\QQ_p} \CC_p,\CC_p)
\]
and $Z[\theta_{Sen}]$ acts on $\tilde{H}^{i,(h)}_{\mathfrak{m}^\vee}\widehat\otimes_{\QQ_p} \CC_p$ via the composite of its natural action and the involution $\iota_{Sen}$, cf. Definition \ref{dfn-iotaSen}. We claim that $\phi'$ is $Z[\theta_{Sen}]$-linear:
\begin{itemize}
\item By Proposition \ref{prop-Poincaredualss}, $\phi'$ is $K_p$-equivariant where $g\in K_p$ acts on $\tilde{H}^{i,(h)}_{\mathfrak{m}^\vee}$ via $g^{-1}$ to turn $(\tilde{H}^{i,(h)}_{\mathfrak{m}^\vee})^*$ into a left $K_p$-representation. Consequently the $Z(U(\mathfrak{g}))$-action is changed by the involution $\iota_Z$ by Remark \ref{rmk-iotaZ}.
\item For the action of $\diag(0,1)$, using the notation before Proposition \ref{prop-Seneq}, it acts on $E_2^{i,j}$ by $\theta^{i,j}-d$, where $d$ comes from the Tate twist $(d)$, and acts on $(\tilde{H}^{i,(h)}_{\mathfrak{m}^\vee})^*\widehat\otimes_{\QQ_p} \CC_p$ via $\theta^i$. By Proposition \ref{prop-Seneq}, $\phi'$ is equivariant with respect to $\theta^i$ on the source and $-\theta^{d-i,d}$ on the target. Our claim follows as $\iota_{Sen}(\diag(0,1))=-\diag(0,1)-d$.
\end{itemize} 

Therefore Corollary \ref{cor-LPvan} implies that the image of $\phi'$ is annihilated by $\iota_{Sen}(\mathcal{I}_{Sen})^n$ for some $n$. By enlarging $n$, we may assume that the cokernel of $\phi_i$ is annihilated by $\mathcal{I}_{Sen}^n$ by Lemma \ref{lem-cokerphii}. We thus prove the following (where the $I_{nsi}$ comes from the difference between $\phi'$ and $\phi_{i,\mathfrak{m}}$)
\begin{lem}
 Let $X'= I_{nsi}H^{d-i}_{\mathrm{cont}}(\widetilde{K_p}, \tilde{H}^{d,(h)}_{\mathfrak{m}} \widehat\otimes_{\QQ_p} \CC_p \widehat\otimes_{\QQ_p} D_h)(d)$ and $Y':=X'\cap \im(\phi')$  and $Z'=X'/Y'$. Then $\iota_{Sen}(\mathcal{I}_{Sen})^nY'=0$ and $\mathcal{I}_{Sen}^nZ'=0$.
\end{lem}

For $k\geq 0$, consider the natural map
\[\varphi_k:I_{nsi}^kH^{d-i}_{\mathrm{cont}}(\widetilde{K_p}, I_{CM}\tilde{H}^{d,(h)}_{\mathfrak{m}} \widehat\otimes_{\QQ_p} \CC_p \widehat\otimes_{\QQ_p} D_h)(d) 
\to
I_{nsi}^k H^{d-i}_{\mathrm{cont}}(\widetilde{K_p}, \tilde{H}^{d,(h)}_{\mathfrak{m}} \widehat\otimes_{\QQ_p} \CC_p \widehat\otimes_{\QQ_p} D_h)(d) \]
induced by 
\[\Phi:H^{d-i}_{\mathrm{cont}}(\widetilde{K_p}, I_{CM}\tilde{H}^{d,(h)}_{\mathfrak{m}} \widehat\otimes_{\QQ_p} D_h)
\to
H^{d-i}_{\mathrm{cont}}(\widetilde{K_p}, \tilde{H}^{d,(h)}_{\mathfrak{m}} \widehat\otimes_{\QQ_p} D_h).\]
Note that $I_{CM}\cdot\ker\Phi=0$ and  \( I_{nsi}\subseteq I_{CM} \). By Proposition \ref{prop-ext vs Hcont}, $H^{d-i}_{\mathrm{cont}}(\widetilde{K_p}, I_{CM}\tilde{H}^{d,(h)}_{\mathfrak{m}} \widehat\otimes_{\QQ_p} D_h)$ is a  finitely generated $D_h$-module.  Because $D_h$ is Noetherian, we can choose $k\geq 1$ sufficiently large such that $X:=I_{nsi}^k H^{d-i}_{\mathrm{cont}}(\widetilde{K_p}, I_{CM}\tilde{H}^{d,(h)}_{\mathfrak{m}} \widehat\otimes_{\QQ_p} D_h)$ has no non-zero $I_{nsi}$-torsion. Then $\varphi_k$ is injective. Since $\Phi$ is clearly $Z[\theta_{Sen}]$-linear, if we let $Y:=X\cap Y'$ and $Z:=X/Y\subseteq Z'$, 
we find that
\[\iota_{Sen}(\mathcal{I}_{Sen})^n Y= \mathcal{I}_{Sen}^n Z=0.\]
The exact sequence
\[0\to Y\to X\to Z\to 0\] is $R_{\bar{D}}^{(r)}[G_{\QQ}]$-linear by the construction. It's enough to show that $X=0$ because $I_{CM}\cdot\mathrm{coker}\Phi=0$.

Suppose that $X\neq 0$. In Theorem \ref{thm-plecticactnonCM}, we constructed an action of $A^{(r)}$ on $I_{CM}\tilde{H}^{d,(h)}$ extending the action of $R_{\bar{D}}[G_{\QQ}]$. Let $A''$ denote the image of $A^{(r)}\widehat\otimes_{\QQ_p}\CC_p$ in $\End(X)$ which is a finitely generated $R_{\bar{D}}^{(r)}\widehat\otimes_{\QQ_p}\CC_p$-module. Because $X$ has no $I_{nsi}$-torsion, there is a homomorphism $f:R_{\bar{D}}^{(r)}\widehat\otimes_{\QQ_p}\CC_p\to \CC_p$ such that $\rho_f$ is strongly irreducible and $A''/(\ker f)\neq 0$. Since $\iota_{Sen}(\mathcal{I}_{Sen})^n  \mathcal{I}_{Sen}^n X=0$, Lemma \ref{lem-eqHTwIsen} and \ref{lem-eqHTwIsen'} imply that $\Lie \rho_f(I_{F_v})\subseteq \mathfrak{gl}_2(\CC_p)$ are not scalars for at least one $p$-adic place $v$ of $F$.  Hence by Proposition \ref{prop-suppA(r)}
\[A''/\ker f\cong \End_{\CC_p}(\oInd^{\QQ}_{F} \rho_f)\]
Consider
\[B_X:=\im(R_{\bar{D}}^{(r)}\widehat\otimes_{\QQ_p}\CC_p[G_{\QQ}]\to \End(X))\subseteq A''\]
 and similarly $B_Y:=\im(R_{\bar{D}}^{(r)}\widehat\otimes_{\QQ_p}\CC_p[G_{\QQ}]\to \End(Y))$ and $B_Z:=\im(R_{\bar{D}}^{(r)}\widehat\otimes_{\QQ_p}\CC_p[G_{\QQ}]\to \End(Z))$. From the $R_{\bar{D}}^{(r)}\widehat\otimes_{\QQ_p}\CC_p [G_{\QQ}]$-linear exact sequence  $0\to Y \to X \to Z\to 0$, we see that
 \begin{enumerate}
 \item $B_Y=B_X/I_Y$ and $B_Z=B_X/I_Z$ for some two-sided  ideals $I_Y$ and $I_Z$ of $B_X$;
 \item $(I_Y\cap I_Z)^2=0$.
 \end{enumerate}
As in the discussion before Proposition \ref{prop-prVIsne}, we get a natural surjective map to a simple algebra
\[\phi_X: B_X\to B_X/\ker f \xrightarrow{pr_V} \End_{\CC_p}(V)\]
where $V$ denotes the $G_{\QQ}$-subrepresentation of $\oInd^{\QQ}_{F} \rho_f$ of the largest dimension.  Since $ \mathcal{I}_{Sen}^n Z=0$, it follows from Proposition \ref{prop-prVIsne}  that there exists an element $a_Z\in I_Z$ such that $\phi_X(a_Z)=1$. Similarly since $\iota_{Sen}(\mathcal{I}_{Sen})^n Y=0$, by Proposition \ref{prop-A'IsenInsisim}, there exists an element $a_Y\in I_Y$ such that $\phi_X(a_Y)=1$. Hence $a_Ya_Z\in I_Y\cap I_Z$ is nilpotent but $\phi_X(a_Ya_Z)=1$. Contradiction.
\end{proof}

\begin{rmk}
One main reason for us to switch from the $A^{(r)}$-action to the $R_{\bar{D}}[G_{\QQ}]$-action is that we decided not to establish the compatibility between the plectic symmetry we constructed and the Poincar\'e duality spectral sequence. Concretely, we do not know whether $Y$ is $A^{(r)}$-stable, although we believe that this should follow from a more careful study of our construction. 
\end{rmk}

\begin{rmk}
If we are able to construct  the full plectic symmetry of  on $\tilde{H}^{d,\la}$ i.e. extend the action of $R_{\bar{D}}[G_{\QQ}]$ to an action of $R_{\bar{D}}[G_{F}^{plec}]$, the same argument would show that a power of the ideal of reducible locus of $\TT$ annihilates $\tilde{H}^{<d}(K^p)_{\QQ_p}$. 
\end{rmk}

\begin{proof}[Proof of Proposition \ref{prop-ext vs Hcont}]
Note that $D_h$ is flat as a left $\Lambda[1/p]$-module.
By choosing a  resolution of the $\Lambda[1/p]$-module $W^*$ by finite free $\Lambda[1/p]$-modules and a standard argument, we reduce to the case $W=C(\widetilde{K^p},\QQ_p)=\Lambda[1/p]^*$ and the first isomorphism becomes clear. 

For the second isomorphism, similarly by choosing a  resolution of $V^*$ by finite free $D_h$-modules, we reduce to the case $V= \mathrm{LA}^{(h)}(\widetilde{K_p})$.

\begin{lem} \label{lem-orbiso}
There is a $\widetilde{K_p}$-equivariant isomorphism
\[\mathrm{LA}^{(h)}(\widetilde{K_p})\widehat\otimes_{\QQ_p} D_h \cong\mathrm{LA}^{(h)}(\widetilde{K_p})\widehat\otimes_{\QQ_p} D_h\]
with $\widetilde{K_p}$ acting on every term except $D_h$ on the right.
\end{lem}

\begin{lem} \label{lem-LAhvancoh}
$H^{>0}_{\mathrm{cont}}(\widetilde{K_p}, \mathrm{LA}^{(h)}(\widetilde{K_p}))=0$.
\end{lem}

Assume both lemmas at the moment. Let $C^\bullet$ be the standard complex computing the continuous  $\widetilde{K_p}$-cohomology of $\mathrm{LA}^{(h)}(\widetilde{K_p})$. It follows that the complex
\[  \mathrm{LA}^{(h)}(\widetilde{K_p})^{\widetilde{K_p}} \to C^\bullet\]
is  exact and strict. The same holds for its completed tensor product with $D_h$ as well. (We note that $C^\bullet\widehat\otimes D_h\cong \Hom_{\mathrm{cont}}(\mathrm{LA}^{(h)}(\widetilde{K_p}),C^\bullet)$.) Hence for $i>0$,
\[H^i_{\mathrm{cont}}(\widetilde{K_p}, \mathrm{LA}^{(h)}(\widetilde{K_p})\widehat\otimes_{\QQ_p} D_h)\cong H^i_{\mathrm{cont}}(\widetilde{K_p}, \mathrm{LA}^{(h)}(\widetilde{K_p})\widehat\otimes_{\QQ_p} D_h)=0,\]
where the isomorphism is induced from the isomorphism in Lemma \ref{lem-orbiso}. This agrees with that
\[ \Ext^{>0}_{D_h}(D_h\otimes_{\Lambda[1/p]} \Lambda[1/p], D_h)=0.\]
Finally for $i=0$, we have
\[ \mathrm{LA}^{(h)}(\widetilde{K_p})\widehat\otimes_{\QQ_p} D_h\cong\Hom_{\mathrm{cont}}\left(\mathrm{LA}^{(h)}(\widetilde{K_p}),\mathrm{LA}^{(h)}(\widetilde{K_p})\right)\cong\Hom_{\mathrm{cont}}(D_h,D_h)\]
where the first isomorphism comes from $D_h=(\mathrm{LA}^{(h)}(\widetilde{K_p}))^*$ and the second isomorphism is obtained by the Schikhof duality. Taking $\widetilde{K_p}$-invariants of both sides and observing that the natural map $\QQ_p[\widetilde{K_p}]\to D_h$ has dense image, we get
\[\Ext^0_{D_h}(D_h, D_h)\cong H^0_{\mathrm{cont}}(\widetilde{K_p}, \mathrm{LA}^{(h)}(\widetilde{K_p})\widehat\otimes_{\QQ_p} D_h).\]

The last part can be proved similarly and we leave it to the reader.
\end{proof}

\begin{proof}[Proof of Lemma \ref{lem-orbiso}]
First we claim that $\mathrm{LA}^{(h)}(\widetilde{K_p})\subseteq C(\widetilde{K^p},\QQ_p)$ is a subring, where the ring structure on $C(\widetilde{K^p},\QQ_p)$ is defined pointwise on $\widetilde{K^p}$. Consider the multiplication map
\[C(\widetilde{K^p},\QQ_p)\widehat\otimes_{\QQ_p} C(\widetilde{K^p},\QQ_p)\to C(\widetilde{K^p},\QQ_p)\]
which is equivariant with respect to left translation actions of $\widetilde{K^p}$ on both sides. Hence by passing to $h$-analytic vectors, it suffices to show that 
\[\mathrm{LA}^{(h)}(\widetilde{K_p})\widehat\otimes_{\QQ_p} \mathrm{LA}^{(h)}(\widetilde{K_p})\subseteq (C(\widetilde{K^p},\QQ_p)\widehat\otimes_{\QQ_p} C(\widetilde{K^p},\QQ_p))^{(h)}.\]
If $C(\widetilde{K^p},\QQ_p)\widehat\otimes_{\QQ_p} C(\widetilde{K^p},\QQ_p)$ is viewed as a $\widetilde{K_p}\times \widetilde{K_p}$-representation via left translations, its $h$-analytic vectors are nothing but $\mathrm{LA}^{(h)}(\widetilde{K_p})\widehat\otimes_{\QQ_p} \mathrm{LA}^{(h)}(\widetilde{K_p})$, hence are contained in the $h$-analytic vectors with respect to the diagonal $\widetilde{K_p}$ ($\subseteq \widetilde{K_p}\times \widetilde{K_p}$)-action. (This can be seen by using \cite[Def.IV.1]{CD-2014} and by choosing an ordered basis of $ \widetilde{K_p}$ and extending it to an ordered basis of $ \widetilde{K_p}\times \widetilde{K_p}$.)

$\mathrm{LA}^{(h)}(\widetilde{K_p})\widehat\otimes_{\QQ_p} D_h$ now has a natural ring structure. On the other hand we let $Id_{\widetilde{K_p}}\in \mathrm{LA}^{(h)}(\widetilde{K_p})\widehat\otimes_{\QQ_p} D_h$ be the element corresponding to the identity function of
\[\Hom_{\mathrm{cont}}(\mathrm{LA}^{(h)},\mathrm{LA}^{(h)})\cong \mathrm{LA}^{(h)}(\widetilde{K_p})\widehat\otimes_{\QQ_p} D_h.\]
(We emphasize that this isomorphism does not preserve the ring structure we consider.) Equivalently $Id_{\widetilde{K_p}}$ can be viewed as the $D_h$-valued function on $\widetilde{K_p}$ sending $g\in \widetilde{K_p}$ to $g\in D_h$. It's easy to see that $Id_{\widetilde{K_p}}$ is invertible in $\mathrm{LA}^{(h)}(\widetilde{K_p})\widehat\otimes_{\QQ_p} D_h$. Hence multiplication by $Id_{\widetilde{K_p}}^{-1}$ on the left is an automorphism of $\mathrm{LA}^{(h)}(\widetilde{K_p})\widehat\otimes_{\QQ_p} D_h$ and gives the isomorphism in the lemma.
\end{proof}

\begin{proof} [Proof of Lemma \ref{lem-LAhvancoh}]
This is essentially a result of Schneider-Teitelbaum. We briefly sketch a proof. See also \cite[Th.2.2.3]{Pan-2022}. It suffices to show that for any admissible Banach space representation $V$ of $\widetilde{K_p}$ and $i>0$,
\[H^{i}_{\mathrm{cont}}(\widetilde{K_p}, V\widehat\otimes_{\QQ_p}\mathrm{LA}^{(h)}(\widetilde{K_p}))=0.\]
This is true for $V=C(\widetilde{K_p},\QQ_p)$ by Shapiro's lemma. In general we can embed $V$ into $C=C(\widetilde{K_p},\QQ_p)^{\oplus m}$ for some $m>0$ and denote $C/V$ by $D$. The sequence
\[0\to H^0(\widetilde{K_p}, V\widehat\otimes\mathrm{LA}^{(h)}(\widetilde{K_p}))\to H^0(\widetilde{K_p}, C\widehat\otimes\mathrm{LA}^{(h)}(\widetilde{K_p})) \to H^0(\widetilde{K_p}, D\widehat\otimes\mathrm{LA}^{(h)}(\widetilde{K_p}))\to 0\]
is exact because it is the same as $0\to V^{(h)}\to C^{(h)} \to D^{(h)} \to 0$. By looking at the long exact sequence of the $\widetilde{K_p}$-cohomology of $0\to V\to C\to D\to 0$, we see that $H^{1}_{\mathrm{cont}}(\widetilde{K_p}, V\widehat\otimes_{\QQ_p}\mathrm{LA}^{(h)}(\widetilde{K_p}))=0$. An induction argument on $i$ gives the general case.
\end{proof}

\section*{Acknowledgements}

We would like to thank Vytautas Pa\v{s}k\={u}nas for helpful comments on the preprint. Y.J. thanks his PhD advisor, Vincent Pilloni, for his guidance. Both of us would like to thank Princeton University and the Simons Foundation for its support during various stages of this collaboration.
Y.J. was supported by a CDSN grant from ENS during his PhD at Université Paris-Saclay, and subsequently by the Max Planck Institute for Mathematics, Bonn. He is grateful to both institutions for their support and excellent working conditions.
L.P. is partially supported by a Sloan Research Fellowship.

\newpage

\printbibliography 

\end{document}